\documentclass[12pt,reqno]{amsart}

\usepackage{amsbsy}
\usepackage{amscd}
\usepackage{amsfonts}
\usepackage{amsmath}
\usepackage{amssymb}
\usepackage{amsthm}
\usepackage{amsxtra}

\usepackage[utf8]{inputenc}
\usepackage[margin=1in]{geometry}

\usepackage{comment}
\usepackage{datetime}
\usepackage{tikz-cd}
\usepackage{tikz}
\usepackage{todonotes}
\usepackage{verbatim}

\theoremstyle{definition}

\newtheorem{thm}{Theorem}[section]
\newtheorem*{thm*}{Theorem}

\newtheorem*{conj*}{Conjecture}

\newtheorem*{conv*}{Convention}
\newtheorem{cor}[thm]{Corollary}
\newtheorem*{cor*}{Corollary}

\newtheorem*{defn*}{Definition}

\newtheorem*{exa*}{Example}
\newtheorem{exc}[thm]{Exercise}
\newtheorem*{exc*}{Exercise}
\newtheorem{lem}[thm]{Lemma}
\newtheorem*{lem*}{Lemma}

\newtheorem*{prob*}{Problem}

\newtheorem*{prop*}{Proposition}
\newtheorem{ques}[thm]{Question}
\newtheorem*{ques*}{Question}
\newtheorem{rmk}[thm]{Remark}
\newtheorem*{rmk*}{Remark}

\newcommand{\C}{\mathbb{C}}

\newcommand{\Q}{\mathbb{Q}}

\newcommand{\R}{\mathbb{R}}
\newcommand{\X}{\mathbb{X}}
\newcommand{\Z}{\mathbb{Z}}

\newcommand{\cA}{\mathcal{A}}

\newcommand{\cC}{\mathcal{C}}

\newcommand{\cF}{\mathcal{F}}

\newcommand{\cI}{\mathcal{I}}
\newcommand{\cK}{\mathcal{K}}
\newcommand{\cM}{\mathcal{M}}

\newcommand{\cS}{\mathcal{S}}
\newcommand{\cT}{\mathcal{T}}
\newcommand{\cU}{\mathcal{U}}
\newcommand{\cV}{\mathcal{V}}

\newcommand{\al}{\alpha}
\newcommand{\bet}{\beta}
\newcommand{\Gam}{\Gamma}
\newcommand{\gam}{\gamma}
\newcommand{\Del}{\Delta}
\newcommand{\del}{\delta}
\newcommand{\eps}{\varepsilon}

\newcommand{\tht}{\theta}

\newcommand{\kap}{\kappa}
\newcommand{\Lam}{\Lambda}

\newcommand{\sig}{\sigma}

\newcommand{\Om}{\Omega}
\newcommand{\om}{\omega}

\newcommand{\ra}{\rightarrow}

\newcommand{\hra}{\hookrightarrow}

\newcommand{\ol}{\overline}
\newcommand{\pr}{\prime}

\newcommand{\wt}{\widetilde}

\newcommand{\hmid}{\hspace{-3pt}\mid} 
\newcommand{\hmod}{\hspace{-7pt}\mod} 

\newcommand{\lr}[1]{\langle #1 \rangle}

\newcommand{\sm}{\setminus}

\DeclareMathOperator{\GL}{GL}
\DeclareMathOperator{\SL}{SL}
\DeclareMathOperator{\SO}{SO}
\DeclareMathOperator{\Sp}{Sp}

\DeclareMathOperator{\Area}{Area}

\DeclareMathOperator{\Aut}{Aut}

\DeclareMathOperator{\im}{Im}
\DeclareMathOperator{\ind}{ind}

\DeclareMathOperator{\Mod}{Mod}

\DeclareMathOperator{\lcm}{lcm}

\DeclareMathOperator{\Per}{Per}

\DeclareMathOperator{\re}{Re}

\newcommand{\be}{\begin{equation*}}
\newcommand{\ee}{\end{equation*}}
\newcommand{\bex}{\begin{exc}}
\newcommand{\eex}{\end{exc}}
\newcommand{\bpf}{\begin{proof}}
\newcommand{\epf}{\end{proof}}

\settimeformat{hhmmsstime}

\title[Dynamics of the absolute period foliation of a stratum]{Dynamics of the absolute period foliation of a stratum of holomorphic 1-forms}
\author[Karl Winsor]{Karl Winsor \\ \\ \monthname[\the\month] \the\day, \the\year}

\begin{document}

\begin{abstract}
Let $\cC$ be a connected component of a stratum of the moduli space of holomorphic $1$-forms of genus $g$. We show that the absolute period foliation of $\cC$ is ergodic on the area-$1$ locus, and that the non-dense leaves lie in an explicit countable union of suborbifolds, subject to a mild assumption on $\cC$. We show similar results for subspaces of $\cC$ defined by topological restrictions on the absolute periods. We obtain these dynamical results by showing that for a typical positive cohomology class in $H^1(S_g;\C)$, the associated space of isoperiodic forms in $\cC$ is connected. Lastly, we show that certain covering constructions provide examples of spaces of isoperiodic forms with positive dimension and infinitely many connected components.
\end{abstract}

\maketitle


\section{Introduction}

Let $\cM_g$ be the moduli space of closed Riemann surfaces of genus $g \geq 2$. Let $\Om\cM_g \ra \cM_g$ be the bundle of pairs $(X,\om)$ with $X \in \cM_g$ and $\om$ a nonzero holomorphic $1$-form on $X$. The space $\Om\cM_g$ is a union of strata $\Om\cM_g(\kap)$ consisting of holomorphic $1$-forms whose zero orders form a given partition $\kap = (m_1,\dots,m_n)$ of $2g - 2$. The stratum $\Om\cM_g(\kap)$ admits a holomorphic absolute period foliation $\cA(\kap)$, whose leaves are navigated by varying a holomorphic $1$-form without changing its integrals along closed loops.

In this paper, we study the dynamics of the absolute period foliation and the topology of spaces of isoperiodic forms in a stratum. Let $\cC$ be a connected component of $\Om\cM_g(\kap)$. We show that for a typical $(X,\om) \in \cC$, the space of holomorphic $1$-forms in $\cC$ isoperiodic to $(X,\om)$ is connected, with a mild assumption on $\cC$. This topological result provides a bridge from dynamics on homogeneous spaces to dynamics on strata, with strong consequences for $\cA(\kap)$. In particular, we show that $\cA(\kap)$ is ergodic on the area-1 locus of $\cC$, and that the non-dense leaves in the area-1 locus lie in an explicit countable union of suborbifolds. We obtain similar results for subspaces of $\cC$ defined by restricting the absolute periods to a closed subgroup of the complex numbers $\C$. Our results suggest that a version of Ratner's theorems for unipotent flows on homogeneous spaces may hold for the absolute period foliation of a stratum. \\

\paragraph{\bf Absolute periods.} Let $S_g$ be the connected closed oriented surface of genus $g$. The absolute periods of a cohomology class $\phi \in H^1(S_g;\C)$ are defined by
\be
\Per(\phi) = \left\{\phi(c) : c \in H_1(S_g;\Z)\right\} \subset \C .
\ee
A holomorphic $1$-form $(X,\om) \in \Om\cM_g$ determines a cohomology class $[\om] \in H^1(X;\C)$, and we define $\Per(\om)$ similarly. Equivalently, $\Per(\om)$ is the set of integrals of $\om$ along closed loops on $X$. For $\kap = (m_1,\dots,m_n)$ a partition of $2g-2$, we denote $|\kap| = n$. Leaves of $\cA(\kap)$ have complex dimension $|\kap| - 1$. Two holomorphic $1$-forms lie on the same leaf of $\cA(\kap)$ if and only if they can be joined by a path in $\Om\cM_g(\kap)$ along which the absolute periods are constant. \\

\paragraph{\bf Main results: dynamics and topology.} Let $\cC$ be a connected component of a stratum $\Om\cM_g(\kap)$, and fix $\phi \in H^1(S_g;\C)$. The space of isoperiodic forms in $\cC$ representing $\phi$ is
\be
\cC(\phi) = \{(X,\om) \in \cC : \phi = f^\ast[\om] \text{ for some } f \in {\rm Homeo}^+(S_g,X) \} .
\ee
We say that $\phi$ is positive if it has positive self-intersection. Note that $\phi$ is positive whenever $\cC(\phi)$ is nonempty. Two holomorphic $1$-forms $(X,\om),(Y,\eta) \in \cC(\phi)$ are isoperiodic, meaning there is a symplectic isomorphism $m : H_1(X;\Z) \ra H_1(Y;\Z)$ such that $\int_{m(c)} \eta = \int_c \om$ for all $c \in H_1(X;\Z)$. Our main results show that $\cC(\phi)$ is typically connected, and that the only obstructions to the connectivity of $\cC(\phi)$ come from algebraic coincidences among absolute periods and the topology of $\cC$.

We state our main results in Theorems \ref{thm:conn}-\ref{thm:ergdenseR+iZ} below. In all of these theorems, $\cC$ is a connected stratum $\Om\cM_g(\kap)$ with $|\kap| > 1$, or the nonhyperelliptic component of $\Om\cM_g(g-1,g-1)$ with $g - 1 \geq 3$ odd. \\

\begin{thm} \label{thm:conn}
Let $\phi \in H^1(S_g;\C)$ be a positive cohomology class such that $\Per(\phi) \cong \Z^{2g}$ and $\Per(\phi) \cap \R z \subset \Q z$ for all $z \in \Per(\phi)$. Then $\cC(\phi)$ is connected.
\end{thm}

The hypotheses on $\Per(\phi)$ ensure that the absolute periods do not satisfy any atypical $\Q$-linear relations, and that there are no atypical pairs of parallel absolute periods.

By \cite{KZ:components}, a stratum $\Om\cM_g(\kap)$ is connected if and only if some $m_j \in \kap$ is odd and not equal to $g - 1$, or $g = 2$. In particular, Theorem \ref{thm:conn} applies to most strata. We exclude the case $|\kap| = 1$, since in that case leaves of $\cA(\kap)$ are points. The orbifold fundamental group $\pi_1(\cC)$ admits a homomorphism $\pi_1(\cC) \ra \Sp(2g,\Z)$ given by the monodromy action on homology. By \cite{Gut:Zorich}, the stratum components in Theorem \ref{thm:conn} are precisely the components of strata $\Om\cM_g(\kap)$ with $|\kap| > 1$ for which this monodromy homomorphism is surjective. Theorem \ref{thm:conn} has strong consequences for the dynamics of $\cA(\kap)$, as follows. Let $\cC_1$ be the set of holomorphic $1$-forms in $\cC$ with area $1$.

\begin{thm} \label{thm:ergdense}
The absolute period foliation of $\cC_1$ is ergodic. For all $(X,\om) \in \cC_1$ such that $\Per(\om) \cong \Z^{2g}$ and $\Per(\om) \cap \R z \subset \Q z$ for all $z \in \Per(\om)$, the leaf of the absolute period foliation of $\cC_1$ through $(X,\om)$ is dense in $\cC_1$.
\end{thm}

Here, ergodicity means a measurable union of leaves has either zero Lebesgue measure or full Lebesgue measure. Ergodicity implies that a typical leaf is dense, but does not provide explicit examples of dense leaves. The density result in Theorem \ref{thm:ergdense} shows that the non-dense leaves in $\cC_1$ lie in an explicit countable union of suborbifolds.

The closure of a leaf of the absolute period foliation in its stratum is constrained by the closure of the associated group of absolute periods in $\C$. Fix $(X,\om) \in \Om\cM_g(\kap)$. If $\Per(\om)$ is closed in $\C$, then $\Per(\om) \cong \Z^2$ is a lattice in $\C$ and the leaf of $\cA(\kap)$ through $(X,\om)$ is closed in $\Om\cM_g(\kap)$. If $\Per(\om)$ is neither closed nor dense, then its closure has the form $\Lam = M \cdot (\R + i\Z)$ for some $M \in \SL(2,\R)$. Let $\Lam_0 = M \cdot \R$ be the identity component of $\Lam$. We define $\cC^\Lam$ to be the set of holomorphic $1$-forms $(X,\om) \in \cC$ whose absolute periods are contained in $\Lam$ and intersect every component of $\Lam$. Equivalently, $\Per(\om) + \Lam_0 = \Lam$. We define $\cC^\Lam_1 = \cC^\Lam \cap \cC_1$. Our next results provide an analogue of Theorems \ref{thm:conn} and \ref{thm:ergdense} for absolute periods that are dense in $\Lam$, with similar dynamical consequences for the absolute period foliation of $\cC^\Lam$. We remark that a priori, it is not even clear that $\cC^\Lam$ is connected.

\begin{thm} \label{thm:connR+iZ}
Let $\phi \in H^1(S_g;\C)$ be a positive cohomology class such that $\Per(\phi) \cong \Z^{2g}$ is not dense in $\C$. Then $\cC(\phi)$ is connected.
\end{thm}

\begin{thm} \label{thm:ergdenseR+iZ}
The absolute period foliation of $\cC_1^\Lam$ is ergodic. For all $(X,\om) \in \cC_1^\Lam$ such that $\Per(\om) \cong \Z^{2g}$, the leaf of the absolute period foliation of $\cC_1^\Lam$ through $(X,\om)$ is dense in $\cC_1^\Lam$.
\end{thm}

Theorem \ref{thm:ergdense} can be deduced from Theorem \ref{thm:conn} using the transfer principle from \cite{CDF:transfer}, by applying Moore's ergodicity theorem and Ratner's orbit closure theorem to the action of $\Sp(2g,\Z)$ on $\Sp(2g,\R)/\Sp(2g-2,\R)$ as done in \cite{Kap:periods}. We review this argument in Section \ref{sec:conn}. Theorem \ref{thm:ergdenseR+iZ} can be deduced similarly from Theorem \ref{thm:connR+iZ}.

Our proofs are inductive, and the inductive steps can be applied to all nonhyperelliptic components of strata $\Om\cM_g(\kap)$ with $|\kap| > 1$. To complete the proof of Theorem \ref{thm:ergdense} for the remaining nonhyperelliptic components, we would need to establish one additional base case, namely, the case of the stratum $\Om\cM_3(2,2)$. \\

\paragraph{\bf Disconnected spaces of isoperiodic forms.}

As a complement to Theorem \ref{thm:conn}, we show that certain covering constructions yield examples where $\cC(\phi)$ is highly disconnected.

\begin{thm} \label{thm:disconn}
Fix $g \geq 4$ even, and let $\cC = \Om\cM_g(2g-3,1)$. There is $\phi \in H^1(S_g;\C)$ such that $\cC(\phi)$ has infinitely many connected components.
\end{thm}

The construction in our proof of Theorem \ref{thm:disconn} admits many variations in most stratum components, and we did not attempt to state this theorem in the greatest possible generality. A similar phenomenon was studied in the stratum $\Om\cM_2(2)$ in \cite{McM:isoperiodic}, where an infinite family of isoperiodic ``fake pentagons'' is described, only one of which has an order $5$ automorphism. In one of our examples with $\cC = \Om\cM_4(5,1)$, each fake pentagon gives rise to a connected component of $\cC(\phi)$ that consists of connected sums of two copies of that fake pentagon. In general, the phenomenon in Theorem \ref{thm:disconn} arises from closed $\GL^+(2,\R)$-invariant subsets of strata with an absolute period foliation that inherits the trivial behavior of the absolute period foliation of $\Om\cM_h(2h-2)$ through a covering construction.

Lastly, we show that Theorems \ref{thm:conn} and \ref{thm:connR+iZ} cannot be extended to any other stratum components, by showing that in the remaining stratum components, spaces of isoperiodic forms are typically disconnected. See Section 2 for definitions regarding types of stratum components. Specifically, we will show that if there is a positive $\phi \in H^1(S_g;\C)$ such that $\Per(\phi) \cong \Z^{2g}$ and $\cC(\phi)$ is connected, then the monodromy homomorphism $\pi_1(\cC) \ra \Sp(2g,\Z)$ is surjective. As a byproduct, we recover the surjectivity of these monodromy homomorphisms for the stratum components in Theorem \ref{thm:conn} from \cite{Gut:Zorich}.

\begin{thm} \label{thm:mono}
Fix $g \geq 3$, and let $\cC$ be a component of a stratum $\Om\cM_g(\kap)$ with $|\kap| > 1$, such that $\cC$ is a spin component or a hyperelliptic component. For all positive $\phi \in H^1(S_g;\C)$ such that $\Per(\phi) \cong \Z^{2g}$, $\cC(\phi)$ is disconnected.
\end{thm}

In fact, we will prove a more general result in Section \ref{sec:conn} that applies to all positive $\phi \in H^1(S_g;\C)$ such that $\Per(\phi)$ is not discrete. \\

\paragraph{\bf Open questions.} Theorems \ref{thm:ergdense} and \ref{thm:ergdenseR+iZ} give hope for a complete classification of closures of leaves of $\cA(\kap)$ in the area-$1$ locus $\Om_1\cM_g(\kap)$. Here, we raise some open questions that suggest a possible classification, in the spirit of Ratner's theorems for unipotent flows on homogeneous spaces \cite{Rat:ICM} and Eskin-Mirzakhani-Mohammadi's theorems for $\GL^+(2,\R)$-orbit closures in strata \cite{EMM:closures}. Let $L$ be the leaf of $\cA(\kap)$ through $(X,\om) \in \Om_1\cM_g(\kap)$, and let $\ol{L}$ be its closure in $\Om_1\cM_g(\kap)$.

\begin{ques} \label{ques:real}
Is $\ol{L}$ always a ``nice'' subset of $\Om_1\cM_g(\kap)$? For instance, is $\ol{L}$ a properly immersed real-analytic suborbifold of $\Om_1\cM_g(\kap)$?
\end{ques}

\begin{ques} \label{ques:clos}
If $\Per(\om)$ is dense in $\C$ and $\ol{L} \neq L$, is it the case that $\ol{L} = \ol{\SL(2,\R) \cdot L}$?
\end{ques}

A more detailed classification question will be raised in Section \ref{sec:conn}. Question \ref{ques:clos} addresses the three known obstructions to the density of $L$ in its connected component in $\Om_1\cM_g(\kap)$. First, $\Per(\om)$ might be contained in a proper closed subgroup of $\C$. Since $\Per(\om)$ contains a lattice in $\C$, up to the action of $\GL^+(2,\R)$ the only possible subgroups are $\R + i\Z$ and $\Z + i\Z$. Second, $L$ might lie in a proper closed $\SL(2,\R)$-invariant subset of its connected component in $\Om_1\cM_g(\kap)$. This occurs, for instance, in certain loci of double covers of quadratic differentials when $|\kap| = 2$. Third, $L$ might be closed and consist of branched covers of holomorphic $1$-forms of lower genus. This occurs in our examples in Theorem \ref{thm:disconn}. An answer to the following question is likely needed for a complete classification of closures of leaves of $\cA(\kap)$. A subset of $\Om\cM_g(\kap)$ is saturated for $\cA(\kap)$ if it is a union of leaves of $\cA(\kap)$.

\begin{ques} \label{ques:invar}
What are the closed $\SL(2,\R)$-invariant subsets of $\Om_1\cM_g(\kap)$ that are saturated for $\cA(\kap)$?
\end{ques}

Although an understanding of the connected components of $\cC(\phi)$ has strong implications for the dynamics of $\cA(\kap)$, Theorems \ref{thm:disconn} and \ref{thm:mono} suggest that a complete classification of connected components of $\cC(\phi)$ may be delicate. Here, we formulate a question which may still have a positive answer. We say that a holomorphic $1$-form in $\cC$ is generic if its $\GL^+(2,\R)$-orbit is dense in $\cC$. Generic holomorphic $1$-forms have no nontrivial automorphisms. Fix generic $(X,\om),(Y,\eta) \in \cC$. Parallel transport along a path $\gam : [0,1] \ra \cC$ from $(X,\om)$ to $(Y,\eta)$ determines a symplectic isomorphism $m_\gam : H_1(X;\Z) \ra H_1(Y;\Z)$. We say that $(X,\om)$ and $(Y,\eta)$ are $\cC$-isoperiodic if there is a path $\gam$ in $\cC$ from $(X,\om)$ to $(Y,\eta)$ such that $\int_{m_\gam(c)} \eta = \int_c \om$ for all $c \in H_1(X;\Z)$.

\begin{ques} \label{ques:conn}
Let $\cC$ be a component of a stratum $\Om\cM_g(\kap)$ with $|\kap| > 1$. Do generic $\cC$-isoperiodic holomorphic 1-forms lie on the same leaf of the absolute period foliation of $\cC$?
\end{ques}

When the monodromy homomorphism $\pi_1(\cC) \ra \Sp(2g,\Z)$ is not surjective, being $\cC$-isoperiodic is a stronger condition than being isoperiodic, since not all symplectic isomorphisms of homology groups arise from paths in $\cC$. Indeed, if $\gam_1,\gam_2 : [0,1] \ra \cC$ are two paths from $(X,\om)$ to $(Y,\eta)$, then the automorphism $m_{\gam_2}^{-1} \circ m_{\gam_1}$ of $H_1(X;\Z)$ must arise from an element of the image of the monodromy homomorphism.

In our examples in Theorem \ref{thm:disconn}, the space $\cC(\phi)$ contains infinitely many connected components that each consist of degree $2$ branched covers of a single holomorphic $1$-form in $\Om\cM_h(2h-2)$, where $g = 2h$. The union of these components is contained in a proper closed $\GL^+(2,\R)$-invariant subset of $\Om\cM_g(2g-3,1)$. However, $\cC(\phi)$ also has at least one connected component containing generic holomorphic $1$-forms. We do not know whether there is exactly one such connected component. \\

\paragraph{\bf Methods.} We now outline the proof of Theorem \ref{thm:conn}, which consists of two inductive arguments. The proof of Theorem \ref{thm:connR+iZ} has a similar structure.

The first inductive argument addresses the case of strata $\Om\cM_g(\kap)$ with $|\kap| = 2$, and we induct on genus. The base case of genus $2$ only involves the stratum $\Om\cM_2(1,1)$, and Theorem \ref{thm:conn} is already known in this case by Theorem 3.3 in \cite{CDF:transfer}, building on an observation in \cite{McM:isoperiodic}. For the inductive step, we construct holomorphic $1$-forms of genus $g+1$ from holomorphic $1$-forms of genus $g$ by forming connected sums with a torus in the following way. Choose a flat torus $T = (\C/(\Z z + \Z w), dz)$ and a closed geodesic $\al \subset T$. Given a holomorphic $1$-form $(X,\om) \in \Om\cM_g(m_1,m_2)$, we can slit $T$ along $\al$, slit $(X,\om)$ along a parallel segment $\al^\pr$ of the same length from a zero $Z$ of order $m_1$ to a regular point, and reglue opposite sides to obtain a new holomorphic $1$-form $(Y,\eta) \in \Om\cM_{g+1}(m_1+2,m_2)$. The connected sum construction is well-defined provided $(X,\om)$ does not have a saddle connection that is parallel to $\al$ and whose length is less than or equal to that of $\al$. A crucial property of this construction is that the starting surface $X$ is only modified on an embedded segment, as opposed to a closed loop.

The leaf of $\cA(m_1,m_2)$ through $(X,\om)$ can be navigated by moving the other zero $Z^\pr$ of $\om$ relative to $Z$. Since $\al^\pr$ is an embedded segment in $X$, one might hope that by moving $Z^\pr$ around $X$ while avoiding $\al^\pr$, one can sweep out essentially the entire leaf of $\cA(m_1,m_2)$ through $(X,\om)$. This is not always possible, which presents a major difficulty in carrying out our inductive approach. However, in Section \ref{sec:leaves}, we will show that this is typically possible in the following sense. Let $L$ be the leaf of $\cA(m_1,m_2)$ through $(X,\om)$, and let $L^\pr$ be the leaf of $\cA(m_1+2,m_2)$ through $(Y,\eta)$. In Section \ref{sec:leaves}, we show that if $\al^\pr$ is not parallel to an absolute period of $(X,\om)$, then the above connected sum construction is well-defined on a path-connected subset of $L$ whose complement in $L$ is a countable union of line segments. This is one of the main observations in this paper, and is broadly applicable beyond the scope of this paper. Applying this observation requires studying when leaves of $\cA(m_1,m_2)$ lift to leaves of the absolute period foliation of a certain finite cover of $\Om\cM_g(m_1,m_2)$, which we also do in Section \ref{sec:leaves}.

For simplicity, we now assume that $g + 1 \geq 4$ and that $m_1,m_2$ are odd. Let $\cC^\pr$ be the component of $\Om\cM_g(m_1,m_2)$ containing $(X,\om)$, and let $\cC$ be the component of $\Om\cM_{g+1}(m_1+2,m_2)$ containing $(Y,\eta)$. Let $\phi \in H^1(S_{g+1};\C)$ be such that $(Y,\eta) \in \cC(\phi)$, and suppose ${\rm Area}(Y,\eta) = 1$. The results in Section \ref{sec:induct} and the hypotheses on $\Per(\eta)$ will ensure that every component of $\cC(\phi)$ contains holomorphic $1$-forms with a presentation as a connected sum with a torus as above. Our simplifying assumption allows us to assume there are two such presentations for which the associated tori $T_1,T_2$ are disjoint. See Figure \ref{fig:sumdisjoint}. Moreover, the hypotheses on $\Per(\eta)$ ensure that for each connected sum presentation, the associated slits are not parallel to an absolute period of the complementary holomorphic $1$-form.

Let $L_1,L_2$ be the leaves of $\cA(m_1,m_2)$ through the complementary holomorphic $1$-forms obtained from $Y \sm T_1$ and $Y \sm T_2$, respectively. By applying the inductive hypothesis to $Y \sm T_1$, we show that $\cC(\phi)$ essentially contains a ``copy'' of $L_1$ mapped into $\cC$ by forming connected sums with $T_1$ whenever this is well-defined. This copy is a connected subset of $\cC(\phi)$ determined by a pair of absolute periods $(z_1,w_1)$ arising from closed loops in $T_1$ that intersect exactly once. The pair $(z_1,w_1)$ satisfies an area constraint $0 < \im(\ol{z}_1 w_1) < 1$. Similarly, $\cC(\phi)$ contains a ``copy'' of $L_2$ determined by a pair of absolute periods $(z_2,w_2)$. Moreover, $(z_1,w_1)$ and $(z_2,w_2)$ determine the same connected component of $\cC(\phi)$, since both components contains $(Y,\eta)$. In this way, we get a function from certain pairs of absolute periods to connected components of $\cC(\phi)$. The construction in Figure \ref{fig:sumdisjoint} can be used to show that two such pairs $(z_1,w_1)$, $(z_2,w_2)$, arising from orthogonal pairs of homology classes and satisfying $\im(\ol{z_1} w_1) + \im(\ol{z_2} w_2) < 1$, determine the same component of $\cC(\phi)$. One can then show that any two such pairs $(z_1,w_1)$, $(z_2,w_2)$ with $0 < \im(\ol{z_1} w_1) < 1$ and $0 < \im(\ol{z}_2 w_2) < 1$ determine the same component of $\cC(\phi)$, by finding a third pair $(z_3,w_3)$ arising from homology classes orthogonal to those of $(z_1,w_1)$, $(z_2,w_2)$ such that $\im(\ol{z}_3 w_3) > 0$ is sufficiently small. This step crucially uses the assumption $g + 1 \geq 4$ and the assumptions on $\Per(\phi)$. To summarize, we are reducing the connectivity of $\cC(\phi)$ to an algebraic problem in terms of pairs of absolute periods.

Unfortunately, the simplified argument above does not work in genus $3$, which presents another major difficulty since our inductive approach relies on this case. To address this difficulty, in Section \ref{sec:pair} we carry out a more complicated study of how a single holomorphic $1$-form in genus at least $3$ can be presented as a connected sum with a torus in multiple ways. This study leads to substantially more difficult algebraic problems, which are solved in Sections \ref{sec:pair} and \ref{sec:conn}. In general, if Theorem \ref{thm:conn} holds for a connected component $\cC^\pr$ of $\Om\cM_g(m_1,m_2)$, and if $\cC$ is a connected component of $\Om\cM_{g+1}(m_1+2,m_2)$ that contains connected sums of holomorphic $1$-forms in $\cC$ with a torus as above, then Theorem \ref{thm:conn} also holds for $\cC$.

Our second inductive argument addresses the general case, and we induct on $|\kap|$. The base case $|\kap| = 2$ was discussed above. The inductive step is easier, and we use the surgery of splitting a zero. Given $(X,\om) \in \Om\cM_g(\kap)$ with a zero $Z$ of order $m \geq 2$, and $1 \leq j < m$, there is a local surgery which splits $Z$ into a pair of zeros of orders $m - j$ and $j$, respectively. This surgery does not change the absolute periods of $(X,\om)$. Let $\kap^\pr = (\kap \sm (m)) \cup (m-j,j)$. We show that if Theorem \ref{thm:conn} holds for a connected component $\cC^\pr$ of $\Om\cM_g(\kap)$, and if $\cC$ is a connected component of $\Om\cM_g(\kap^\pr)$ that contains holomorphic $1$-forms arising from splitting a zero on a holomorphic $1$-form in $\Om\cM_g(\kap)$, then Theorem \ref{thm:conn} also holds for $\cC$. The stratum components appearing in Theorem \ref{thm:conn} are precisely those that can be accessed from $\Om\cM_2(1,1)$ by iteratively forming a connected sum with a torus and then iteratively splitting a zero.

Our methods apply more generally to $\GL^+(2,\R)$-orbit closures with a non-trivial absolute period foliation and to complex relative period geodesics in strata. We pursue these topics, along with Questions \ref{ques:real}-\ref{ques:conn}, in forthcoming work. \\

\paragraph{\bf Notes and references.} The particular case of the dynamics of the absolute period foliation of $\Om\cM_g$ are studied in \cite{CDF:transfer}, \cite{Ham:ergodicity}, and \cite{McM:isoperiodic}. For $g = 2$ and $g = 3$, the fact that any principally polarized abelian variety is the Jacobian of a stable curve is exploited in \cite{McM:isoperiodic} to prove ergodicity on $\Om_1\cM_g$, and this idea is pushed further in \cite{CDF:transfer} to obtain a classification of leaf closures. The approach in \cite{CDF:transfer} then uses an inductive argument involving isoperiodic degenerations to the boundary of moduli space, in order to classify leaf closures and to prove ergodicity on $\Om_1\cM_g$ for all $g \geq 2$. An independent and simpler proof of ergodicity on $\Om_1\cM_g$ for $g \geq 2$ is given in \cite{Ham:ergodicity}, also using induction and degenerations. All of these results apply to the principal stratum $\Om\cM_g(1,\dots,1)$ as well. We remark that the boundary of moduli space does not play a role in our proofs, and so we obtain a new proof of ergodicity on $\Om_1\cM_g$ for $g \geq 3$. The methods in \cite{CDF:transfer}, \cite{Ham:ergodicity}, and \cite{McM:isoperiodic} do not seem to be easily adaptable to non-principal strata, due to our limited understanding of the Schottky locus and a lack of available base cases for induction.

Less is known about the dynamics of the absolute period foliation of non-principal strata. In \cite{HW:Rel}, it is shown that the Arnoux-Yoccoz surfaces of genus $g \geq 3$ give examples of dense leaves in a fixed-area locus in a certain connected component of $\Om\cM_g(g-1,g-1)$. Additional examples of dense leaves in $\Om_1\cM_3(2,1,1)$ arising from Prym loci are given in \cite{Ygo:dense}. In \cite{Win:dense}, it is shown that there exist dense relative period geodesics in the area-$1$ locus of every stratum component with at least two zeros, and explicit examples of dense leaves of many complex $1$-dimensional subfoliations of the absolute period foliation of these stratum components are given. In \cite{McM:foliations} and \cite{Ygo:ergodic}, it is shown that leaves of the absolute period foliation of eigenform loci in $\Om\cM_2(1,1)$, and more generally of rank $1$ affine invariant manifolds, are either closed or dense in the area-$1$ locus.

After the first version of this article was posted, Jon Chaika and Barak Weiss posted a conditional proof that real Rel flows are ergodic on the area-$1$ locus of all stratum components with multiple zeros \cite{CW:realrel}, conditional on a generalization of the measure rigidity results of \cite{EM:stationary}. Their result implies the ergodicity of the absolute period foliation on the area-$1$ locus of all stratum components with multiple zeros. We remark that our proof of the ergodicity part of Theorem \ref{thm:ergdense} can be made to only rely on Moore's ergodicity theorem \cite{Zim:book} and the ergodicity of the $\GL^+(2,\R)$-action on stratum components \cite{Mas:IET}, \cite{Vee:Gauss}, \cite{Vee:flow}. See Remark \ref{rmk:noEMM}. Additionally, the methods in our paper apply to loci that are not $\SL(2,\R)$-invariant, as seen in Theorems \ref{thm:connR+iZ} and \ref{thm:ergdenseR+iZ}. Our proof of Theorems \ref{thm:conn} and the density part of Theorem \ref{thm:ergdense} rely on the explicit full measure sets of dense $\GL^+(2,\R)$-orbits in strata given in \cite{Wri:field}, which in turn relies on the rigidity results for $\GL^+(2,\R)$-orbit closures in strata in \cite{EMM:closures}.

The connected sums we consider are special cases of the surgery of bubbling a handle in \cite{KZ:components} and the figure-eight construction in \cite{EMZ:principal}. These surgeries play an important role in the classification of connected components of strata in \cite{KZ:components}, and in the computation of Siegel-Veech constants for strata in \cite{EMZ:principal}. Detailed studies of presentations of holomorphic $1$-forms in genus $2$ as connected sums are carried out in \cite{McM:SL2R} and \cite{CM:non-ergodic}. In \cite{McM:SL2R}, connected sums are used to classify all $\SL(2,\R)$-orbit closures and invariant measures in $\Om_1\cM_2(1,1)$, and in \cite{CM:non-ergodic}, connected sums are used to exhibit minimal non-uniquely ergodic straight-line flows on every non-Veech surface in genus $2$.

The intrinsic geometry of leaves of the absolute period foliation of $\Om\cM_g$ and of strata are studied in \cite{BSW:horocycle}, \cite{McM:navigating}, \cite{McM:isoperiodic}, and \cite{MW:cohomology}. Completeness results for the natural metric on leaves are given in these papers. In \cite{McM:isoperiodic}, it is shown that the metric completion of a typical leaf in $\Om\cM_2$ is a Riemann surface biholomorphic to the upper half-plane. In contrast, examples of infinite-genus leaves in certain strata of holomorphic $1$-forms with exactly $2$ zeros are given in \cite{Win:genus}. The geometry of leaves in $\Om\cM_2$ is studied in \cite{EMS:billiards} in order to count periodic billiard trajectories in a square with a barrier, and in \cite{Dur:square} to make progress toward classifying square-tiled surfaces in $\Om\cM_2(1,1)$. \\

\paragraph{\bf Acknowledgements.} The author thanks Curt McMullen for his interest, encouragement, and extensive feedback on earlier versions of this work. The author thanks Dawei Chen, Kathryn Lindsey, Tina Torkaman, and Yongquan Zhang for helpful comments and discussions. The author acknowledges support from the National Science Foundation under grants DGE-1144152 and DMS-2303185, and support from the Fields Institute.


\section{Splitting zeros and connected sums} \label{sec:surgery}

We recall relevant material on strata of holomorphic $1$-forms and the $\GL^+(2,\R)$-action on strata. We then discuss the surgeries of splitting zeros and forming a connected sum with a torus. For our purposes, we will need to treat these surgeries as globally defined operations on strata, which requires passing to a certain finite cover by marking a prong. We also set some notation for the rest of the paper. For additional background material, we refer to \cite{Bai:Euler}, \cite{Wri:survey}, \cite{Zor:survey}. \\

\paragraph{\bf Holomorphic $1$-forms.} We denote by $(X,\om)$ a closed Riemann surface $X$ of genus $g \geq 2$ equipped with a holomorphic $1$-form $\om$. We always assume $\om \neq 0$. The zero set $Z(\om)$ is finite and nonempty, and the orders of the zeros form a partition of $2g-2$. Integration of $\om$ on $X \sm Z(\om)$ gives an atlas of charts to the complex plane $\C$, whose transition maps are translations. Geometric structures on $\C$ that are invariant under translations can be pulled back to $X \sm Z(\om)$ using this atlas. In particular, the Euclidean metric on $\C$ determines a singular flat metric $|\om|$ on $X$ with a cone point with angle $2\pi(k+1)$ at a zero of order $k$.

In our figures, we will present holomorphic $1$-forms as finite disjoint unions of polygons in $\C$, possibly with slits, with pairs of edges identified by translations in $\C$. In most cases, the edge identifications will be implicit from the requirement that identified edges must be parallel and of the same length.

A {\em saddle connection} on $(X,\om)$ is an oriented geodesic segment $\gam$ for the metric $|\om|$ with endpoints in $Z(\om)$ and otherwise disjoint from $Z(\om)$. The {\em holonomy} of $\gam$ is the nonzero complex number $\int_\gam \om$. If $s$ is any oriented geodesic segment on $X$, we will also refer to the integral $\int_s \om$ as the {\em holonomy} of $s$. A closed geodesic $\al$ in $X \sm Z(\om)$ is contained in a maximal connected open subset of $X \sm Z(\om)$ foliated by parallel closed geodesics. Such an open subset $C$ is called a {\em cylinder}. The boundary of $C$ consists of a finite union of parallel saddle connections. Each homotopy class of paths in $X$ with endpoints in $Z(\om)$ has a unique {\em geodesic representative} of minimal length in the metric $|\om|$, consisting of finitely many saddle connections such that each angle formed by two consecutive saddle connections is at least $\pi$. Let
\be
\Per(\om) = \left\{\int_c \om : c \in H_1(X;\Z)\right\}
\ee
be the subgroup of $\C$ of {\em absolute periods} of $\om$. Let
\be
\Gam(\om) = \left\{\int_\gam \om : \gam \text{ is a saddle connection on } (X,\om) \right\}
\ee
be the subset of $\C$ of holonomies of saddle connections. The subset $\Gam(\om)$ is discrete. In particular, for any $B > 0$, there are only finitely many saddle connections on $(X,\om)$ of length at most $B$. Let $\C^\ast = \C \sm \{0\}$, and let
\be
\Del(\om) = \C^\ast \sm \left\{t z : t \geq 1, \; z \in \Gam(\om) \right\}
\ee
be the complement of the rays starting at a saddle connection holonomy and emanating away from the origin. Since $\Gam(\om)$ is discrete, $\Del(\om)$ is open. \\

\paragraph{\bf Strata.} Let $S_g$ be the connected closed oriented surface of genus $g \geq 2$. The Teichm\"{u}ller space $\cT_g$ of marked Riemann surfaces $f : S_g \ra X$ of genus $g$ is a complex manifold of dimension $3g-3$. The mapping class group $\Mod_g$ acts properly discontinuously on $\cT_g$ by biholomorphisms. The moduli space of Riemann surfaces of genus $g$ is the complex orbifold $\cM_g = \cT_g / \Mod_g$. The action of $\Mod_g$ on $\cT_g$ induces an action on the bundle $\Om\cT_g \ra \cT_g$ of nonzero holomorphic $1$-forms on marked Riemann surfaces. The {\em moduli space of holomorphic $1$-forms of genus $g$} is the complex orbifold $\Om\cM_g = \Om\cT_g / \Mod_g$. The space $\Om\cT_g$ decomposes into strata $\Om\cT_g(\kap)$ indexed by partitions $\kap = (m_1,\dots,m_n)$ of $2g-2$. The stratum $\Om\cT_g(\kap)$ consists of holomorphic $1$-forms on marked Riemann surfaces with exactly $n$ distinct zeros of orders $m_1,\dots,m_n$. The action of $\Mod_g$ preserves each stratum, and $\Om\cM_g$ decomposes into {\em strata} $\Om\cM_g(\kap) = \Om\cT_g(\kap) / \Mod_g$ which are complex suborbifolds of $\Om\cM_g$. We denote $|\kap| = n$. \\

\paragraph{\bf Period coordinates.} Fix $(X_0,\om_0) \in \Om\cT_g(\kap)$. There is a neighborhood $\cU \subset \Om\cT_g(\kap)$ of $(X_0,\om_0)$, and a natural isomorphism $H^1(X,Z(\om);\C) \cong H^1(X_0,Z(\om_0);\C)$ for any $(X,\om) \in \cU$, provided by the Gauss-Manin connection on the bundle of relative cohomology groups over $\Om\cT_g(\kap)$. {\em Period coordinates} on $\cU$ are defined using these isomorphisms by
\be
\cU \ra H^1(X_0,Z(\om_0);\C), \quad (X,\om) \mapsto [\om] ,
\ee
and this map is a biholomorphism from an open subset of $\Om\cT_g(\kap)$ to an open subset of a complex vector space of dimension $2g+|\kap|-1$. Given a choice of basis $c_1,\dots,c_{2g+|\kap|-1}$ for $H_1(X_0,Z(\om_0);\Z)$, we get a map
\be
\cU \mapsto \C^{2g+|\kap|-1}, \quad (X,\om) \mapsto \left(\int_{c_1} \om, \dots, \int_{c_{2g+|\kap|-1}} \om\right) .
\ee
The components $\int_{c_j} \om$ are the {\em period coordinates} of $(X,\om)$, and they depend on the choice of basis for the integral relative homology group. Transition maps between period coordinate charts are integral linear maps that preserve $H^1(X_0,Z(\om_0);\Z)$. Period coordinates give $\Om\cM_g(\kap)$ the structure of an affine orbifold. \\

\paragraph{\bf Area.} The {\em area} of $(X,\om)$ is the area of $X$ with respect to the metric $|\om|$, and is given by
\be
\Area(X,\om) = \frac{i}{2} \int_X \om \wedge \ol{\om} = \sum_{j=1}^g \im \left(\int_{a_j}\ol{\om} \int_{b_j}\om\right)
\ee
where $\{a_j,b_j\}_{j=1}^g$ is any symplectic basis for $H_1(X;\Z)$. The area of $(X,\om)$ is an invariant of the absolute cohomology class $[\om] \in H^1(X;\C)$. Let
\be
\Om_1\cM_g(\kap) = \left\{(X,\om) \in \Om\cM_g(\kap) : \Area(X,\om) = 1 \right\}
\ee
be the area-$1$ locus in $\Om\cM_g(\kap)$. The area-$1$ locus $\Om_1\cM_g(\kap)$ is a real-analytic orbifold and has a canonical Lebesgue measure class. \\

\paragraph{\bf The $\GL^+(2,\R)$-action.} Let $\GL^+(2,\R)$ be the group of linear automorphisms of $\R^2$ with positive determinant. Let $\SL(2,\R)$ be the subgroup of matrices with determinant $1$. The standard $\R$-linear action of $\GL^+(2,\R)$ on $\C = \R + \R i$ induces an action on $\Om\cM_g$ by postcomposition with an atlas of charts on $X \sm Z(\om)$ as above. The action of $\GL^+(2,\R)$ preserves each stratum $\Om\cM_g(\kap)$, and the action of $\SL(2,\R)$ preserves $\Om_1\cM_g(\kap)$. The action of $\SL(2,\R)$ is {\em ergodic} on each connected component of $\Om_1\cM_g(\kap)$ with respect to the Lebesgue measure class \cite{Mas:IET}, \cite{Vee:Gauss}, \cite{Vee:flow}, meaning that any measurable $\SL(2,\R)$-invariant subset of $\Om_1\cM_g(\kap)$ has either zero measure or full measure. \\

\paragraph{\bf Connected components of strata.} Most strata in $\Om\cM_g$ are connected. However, strata can have up to $3$ connected components, which are classified by hyperellipticity and the parity of a spin structure. We recall their classification from \cite{KZ:components}, and we refer to \cite{KZ:components} and the references therein for further details.

Let $\kap = (m_1,\dots,m_n)$ be a partition of $2g-2$ with all $m_j$ even, and fix $(X,\om) \in \Om\cM_g(\kap)$. The {\em index} of a smooth oriented closed loop $\gam \subset X \sm Z(\om)$ is the degree of the associated Gauss map $\gam \ra S^1$, that is, $1/2\pi$ times the total change in angle of a tangent vector travelling once around $\gam$. We denote the index of $\gam$ by $\ind(\gam)$. Let $\{\al_j,\bet_j\}_{j=1}^g$ be a collection of smooth oriented closed loops in $X \sm Z(\om)$ representing a symplectic basis for $H_1(X;\Z)$. The {\em parity of the spin structure} $\phi(\om)$ is defined by
\be
\phi(\om) = \sum_{j=1}^g (\ind(\al_j) + 1)(\ind(\bet_j) + 1) \hmod 2 .
\ee
It is a fact that $\phi(\om)$ is independent of the choice of symplectic basis of $H_1(X;\Z)$ and the choice of representatives for the symplectic basis. Moreover, $\phi(\om)$ is an invariant of the connected component of $(X,\om)$ in $\Om\cM_g(\kap)$. Components of $\Om\cM_g(\kap)$ are called {\em spin} components. A component $\cC \subset \Om\cM_g(\kap)$ is {\em even} or {\em odd} according to whether $\phi(\om) = 0$ or $\phi(\om) = 1$ for $(X,\om) \in \cC$.

If $\cC \subset \Om\cM_g(2g-2)$ consists of holomorphic $1$-forms on hyperelliptic curves, or if $\cC \subset \Om\cM_g(g-1,g-1)$ consists of holomorphic $1$-forms on hyperelliptic curves whose hyperelliptic involution exchanges the two zeros, then $\cC$ is {\em hyperelliptic}. A component which is not hyperelliptic is {\em nonhyperelliptic}.

\begin{thm} \label{thm:KZ} (\cite{KZ:components}, Theorems 1-2 and Corollary 5)
For $g \geq 4$, the connected components of $\Om\cM_g(\kap)$ are as follows.
\begin{enumerate}
    \item If $\kap = (2g-2)$ or $\kap = (g-1,g-1)$, then $\Om\cM_g(\kap)$ has a unique hyperelliptic component.
    \item If all $m_j \in \kap$ are even, then $\Om\cM_g(\kap)$ has exactly two nonhyperelliptic components: one even component and one odd component.
    \item If some $m_j \in \kap$ is odd, then $\Om\cM_g(\kap)$ has a unique nonhyperelliptic component.
\end{enumerate}
For $2 \leq g \leq 3$, the stratum $\Om\cM_g(\kap)$ is connected unless $\kap = (4)$ or $\kap = (2,2)$, in which case $\Om\cM_g(\kap)$ has exactly two connected components: one odd component, and one hyperelliptic component which is also an even component.
\end{thm}

\begin{cor} \label{cor:KZconn}
A stratum $\Om\cM_g(\kap)$ is connected if and only if there is $m_j \in \kap$ that is odd and not equal to $g-1$, or $g = 2$.
\end{cor}

\paragraph{\bf Finite covers of strata.} Let $\kap$ be a partition of $2g-2$, and choose $m \in \kap$. We will need to work with a finite cover of a stratum
\begin{equation} \label{eq:prong}
p : \wt{\Om}\cM_g(\kap;m) \ra \Om\cM_g(\kap)
\end{equation}
consisting of holomorphic $1$-forms in $\Om\cM_g(\kap)$ equipped with a distinguished rightward horizontal direction at a zero $Z$ of order $m$. We denote elements of $\wt{\Om}\cM_g(\kap;m)$ by $(X,\wt{\om})$. We denote the distinguished direction by $\tht(\wt{\om})$, and we refer to $\tht(\wt{\om})$ as a {\em prong}. The degree of $p$ is $(m+1)N_m$, where $N_m$ is the number of times $m$ appears in $\kap$. An automorphism of $(X,\wt{\om})$ is required to fix the distinguished zero $Z$ and the prong $\tht(\wt{\om})$, so $(X,\wt{\om})$ has no nontrivial automorphisms. \\

\paragraph{\bf Compactness.} A subset $\cK \subset \Om\cM_g(\kap)$ is compact if and only if $\cK$ is closed and there exists $\eps > 0$ such that every saddle connection on every holomorphic $1$-form in $\cK$ has length at least $\eps$. The analogous statement holds for $\wt{\Om}\cM_g(\kap;m)$.

Let $\cU \subset \wt{\Om}\cM_g(\kap;m)$ be a contractible open subset whose closure is compact. Recall that elements of $\wt{\Om}\cM_g(\kap;m)$ have no nontrivial automorphisms. Fix $(X,\wt{\om}) \in \cU$. For each homotopy class $\gam$ of paths on $(X,\wt{\om})$ with endpoints in $Z(\om)$, there is a well-defined continuous {\em length function}
\be
\ell_\gam : \cU \ra \R_{>0}
\ee
whose value at $(Y,\wt{\eta}) \in \cU$ is the length of the geodesic representative of the corresponding homotopy class of paths on $(Y,\wt{\eta})$. The geodesic representative of $\gam$ on $(X,\wt{\om})$ is a finite union of saddle connections $\gam_1,\dots,\gam_j$ on $(X,\wt{\om})$ where the angles between consecutive saddle connections $\gam_k,\gam_{k+1}$ are all at least $\pi$. Saddle connections persist on open neighborhoods in strata, so let $\cU_1 \subset \cU$ be a neighborhood of $(X,\wt{\om})$ on which $\gam_1,\dots,\gam_j$ all persist. The angles between these consecutive saddle connections vary continuously on $\cU_1$. If all of these angles are strictly greater than $\pi$, then the geodesic representative of $\gam$ remains the union of $\gam_1,\dots,\gam_j$ on all holomorphic $1$-forms in a neighborhood $\cU_2 \subset \cU_1$ of $(X,\wt{\om})$. If some of these angles are equal to $\pi$, then in the geodesic representative of $\gam$ on nearby holomorphic $1$-forms, some sequences of consecutive saddle connections $\gam_{j_1},\gam_{j_1+1},\dots,\gam_{j_2}$ may be replaced by a single saddle connection $\gam_{j_1}^\pr$. Here, $\gam_{j_1}^\pr$ is homotopic to the union $\gam_{j_1} \cup \cdots \cup \gam_{j_2}$, and the length of $\gam_{j_1}^\pr$ remains close to the sum of the lengths of $\gam_{j_1},\dots,\gam_{j_2}$.

For any $B > 0$, since there are only finitely many saddle connections on $(X,\wt{\om})$ with length at most $B$, there are only finitely many homotopy classes $\gam$ such that $\ell_\gam(X,\wt{\om}) \leq B$. We will need the following strengthening of this observation.

\begin{lem} \label{lem:Ufin}
For any $B > 0$, there are only finitely many homotopy classes $\gam$ as above such that $\inf_\cU \ell_\gam < B$.
\end{lem}

\begin{proof}
Since the closure of $\cU$ is compact, there is $0 < \eps < B$ such that every saddle connection on every holomorphic $1$-form in $\cU$ has length at least $\eps$. For all $(X,\wt{\om}) \in \cU$, there is an open neighborhood $\cV \subset \cU$ of $(X,\wt{\om})$ such that for all $(Y,\wt{\eta}) \in \cV$, every saddle connection $\gam^\pr$ on $(X,\wt{\om})$ of length at most $B$ persists as a saddle connection on $(Y,\wt{\eta})$ and satisfies
\be
|\ell_{\gam^\pr}(X,\wt{\om}) - \ell_{\gam^\pr}(Y,\wt{\eta})| < \frac{\eps}{2} .
\ee
If $\gam$ is a homotopy class such that $\ell_\gam(X,\wt{\om}) < B$, then on $(X,\wt{\om})$, the geodesic representative of $\gam$ is a finite union of saddle connections $\gam_1,\dots,\gam_j$ whose lengths lie in the interval $[\eps,B]$. For each $\gam_k$, and for any $(Y,\wt{\eta}) \in \cV$, we have
\be
\sup_\cV \ell_{\gam_k} < \ell_{\gam_k}(Y,\wt{\eta}) + \eps \leq 2\ell_{\gam_k}(Y,\wt{\eta}).
\ee
Therefore, since $\inf_\cV \ell_\gam < B$, we have $\sup_\cV \ell_\gam \leq 2 \inf_\cV \ell_\gam < 2 B$.

Now let $\gam$ be any homotopy class such that $\inf_\cU \ell_\gam < B$. Since the closure of $\cU$ is compact, there is a finite covering $\cU = \bigcup_{k=1}^N \cV_k$ by open neighborhoods as above. For each $1 \leq k \leq N$, fix $(X_k,\wt{\om}_k) \in \cV_k$. Since $\inf_{\cV_k} \ell_\gam < B$ for some $1 \leq k \leq N$, we have $\ell_\gam(X_k,\wt{\om}_k) < 2 B$ for some $1 \leq k \leq N$, and thus there are only finitely many possibilities for $\gam$.
\end{proof}

Fix $z \in \C^\ast$, let $I = \{tz : 0 \leq t \leq 1\}$, and let $\Om\cM_g(\kap;I)$ be the set of holomorphic $1$-forms in $\Om\cM_g(\kap)$ with a saddle connection whose holonomy is in $I$.

\begin{lem} \label{lem:Iclosed}
The subset $\Om\cM_g(\kap;I) \subset \Om\cM_g(\kap)$ is closed.
\end{lem}

\begin{proof}
Choose $m \in \kap$, and let $p : \wt{\Om}\cM_g(\kap;m) \ra \Om\cM_g(\kap)$ be the stratum cover by prong-marked holomorphic $1$-forms from (\ref{eq:prong}). Fix $(X,\om) \in \Om\cM_g(\kap)$ such that $(X,\om) \notin \Om\cM_g(\kap;I)$, and fix $(X,\wt{\om}) \in p^{-1}(X,\om)$. By Lemma \ref{lem:Ufin}, there is an open neighborhood $\cU \subset \wt{\Om}\cM_g(\kap;m)$ of $(X,\wt{\om})$ such that there are only finitely many homotopy classes $\gam_1,\dots,\gam_j$ of paths on $X$ with endpoints in $Z(\om)$ satisfying $\inf_{\cU} \ell_{\gam_k} \leq |z|$.

Fix $1 \leq k \leq j$. If $\ell_{\gam_k}(X,\wt{\om}) > |z|$, then there is a small open neighborhood $\cU_k \subset \cU$ of $(X,\wt{\om})$ such that $\inf_{\cU_k} \ell_{\gam_k}(X,\wt{\om}) > |z|$. If $\ell_{\gam_k}(X,\wt{\om}) \leq |z|$, then the geodesic representative of $\gam_k$ on $(X,\wt{\om})$ must be a union of saddle connections $\del_1,\dots,\del_r$ that are not parallel to $z$, meaning $\int_{\del_s} \om \notin \R z$ for $1 \leq s \leq r$. If the saddle connections $\del_s$ are pairwise non-parallel, then there is a small open neighborhood $\cU_k \subset \cU$ of $(X,\wt{\om})$ in which the geodesic representative of $\gam_k$ remains the union of $\del_1,\dots,\del_r$, and these saddle connections remain pairwise non-parallel and also not parallel to $z$. Lastly, suppose some of the $\del_s$ are parallel. On nearby holomorphic $1$-forms $(Y,\wt{\eta})$, for $1 \leq s \leq r$, the geodesic representative of $\gam_k$ either contains $\del_s$ or contains a new saddle connection $\del_{r_1}^\pr$ homotopic to a union of consecutive saddle connections $\del_{r_1},\dots,\del_{r_2}$ with $r_1 \leq s \leq r_2$. On $(Y,\wt{\eta})$, the holonomies of these saddle connections satisfy $\int_{\del_{r_1}^\pr} \eta = \int_{\del_{r_1}} \eta + \cdots + \int_{\del_{r_2}} \eta$, and the saddle connections $\del_{r_1},\dots,\del_{r_2}$ remain nearly parallel. Then on $(Y,\wt{\eta})$, the length of $\del_{r_1}^\pr$ is close to the sum of the lengths of $\del_{r_1},\dots,\del_{r_2}$, and $\del_{r_1}^\pr$ is also nearly parallel to these saddle connections. Thus, there is a small open neighborhood $\cU_k \subset \cU$ of $(X,\wt{\om})$ in which the geodesic representative of $\gam_k$ is a union of saddle connections that are not parallel to $z$.

Let $\cV = \bigcap_{k=1}^j \cU_k$. By definition, for all $(Y,\wt{\eta}) \in \cV$, and for all homotopy classes of paths $\gam$ on $Y$ with endpoints in $Z(\eta)$, either $\ell_\gam(Y,\wt{\eta}) > |z|$ or $\gam$ contains a saddle connection $\del$ with $\int_\del \eta \notin I$. In particular, $(Y,\wt{\eta})$ does not have any saddle connections with holonomy in $I$. Thus, since $p$ is an open map, the image $p(\cV)$ is an open neighborhood of $(X,\om)$ that is disjoint from $\Om\cM_g(\kap;I)$.
\end{proof}

\paragraph{\bf Domains of surgeries.} Below, we will discuss two well-known surgeries that involve slitting and regluing geodesic segments emanating from a zero of a holomorphic $1$-form $(Y,\eta)$. If $Z$ is a zero of $\eta$ and $s$ is a segment emanating from $Z$, then the holonomy $\int_s \eta \in \C^\ast$ does not uniquely determine $s$. Rather, the geodesic segments emanating from $Z$ are parameterized by an open subset of a finite cover of $\C^\ast$. Let
\be
\sig : \wt{\C}^\ast \ra \C^\ast
\ee
be the universal covering group of $\C^\ast$. We have polar coordinates $\C^\ast \cong \R_{>0} \times \R/2\pi\Z$ and $\wt{\C}^\ast \cong \R_{>0} \times \R$ in which the identity elements correspond to $(1,0)$. In polar coordinates, $\sig$ is given by reduction mod $2\pi$ in the angular coordinate. The fundamental group of $\GL^+(2,\R)$ is isomorphic to $\Z$ and generated by the rotation subgroup $\SO(2,\R)$. Let
\be
\zeta : \wt{\GL}^+(2,\R) \ra \GL^+(2,\R)
\ee
be the universal covering group of $\zeta$. The action of $\GL^+(2,\R)$ on $\C^\ast$ determines a continuous action of $\wt{\GL}^+(2,\R)$ on $\wt{\C}^\ast$ that is equivariant in the sense that for $M \in \wt{\GL}^+(2,\R)$ and $z \in \wt{\C}^\ast$, we have $\sig(M z) = \zeta(M) \sig(z)$. The kernel of $\zeta$ is generated by an element $R$ that acts on $\wt{\C}^\ast$ by rotating counterclockwise by $2\pi$.

Now fix $(X,\wt{\om}) \in \wt{\Om}\cM_g(\kap;m)$. There is a degree $m+1$ connected covering of topological groups $\wt{\C}^\ast_{m+1} \ra \C^\ast$ which is unique up to isomorphism. Let
\be
\Del(\wt{\om}) \ra \Del(\om)
\ee
be the degree $m+1$ covering consisting of oriented geodesic segments $\gam$ starting at the distinguished zero $Z \in Z(\om)$ such that $\int_\gam \om \in \Del(\om)$. The choice of prong determines a natural inclusion $\Del(\wt{\om}) \hra \wt{\C}^\ast_{m+1}$ by sending segments along the direction of the prong $\tht(\wt{\om})$ into $\R_{>0} \times \{0\}$ in polar coordinates, and requiring that the image of $\gam \in \Del(\wt{\om})$ projects to $\int_\gam \om \in \Del(\om)$. The action of $\GL^+(2,\R)$ determines a continuous action of $\wt{\GL}^+(2,\R)$ on $\wt{\Om}\cM_g(\kap;m)$ that is equivariant in the sense that for $M \in \wt{\GL}^+(2,\R)$ and $(Y,\wt{\eta}) \in \wt{\Om}\cM_g(\kap;m)$, we have $p(M(Y,\wt{\eta})) = \zeta(M)(Y,\eta)$. The element $R$ acts on $\wt{\Om}\cM_g(\kap;m)$ by rotating the chosen prong clockwise by $2\pi$.

Define $S(\wt{\om}) \subset \Del(\wt{\om}) \times \C^\ast$ to be the subset of pairs $(\gam,w)$ such that the holonomy $z = \int_\gam \om$ satisfies $\im(\ol{z} w) > 0$. We similarly have a natural inclusion $S(\wt{\om}) \hra \wt{\C}^\ast_{m+1} \times \C^\ast$. Let
\be
\cS(\kap;m) \ra \wt{\Om}\cM_g(\kap;m)
\ee
be the bundle of holomorphic $1$-forms equipped with an oriented geodesic segment in $\Del(\wt{\om})$. Elements of $\cS(\kap;m)$ are denoted $(X,\wt{\om},\gam)$, where $(X,\wt{\om}) \in \wt{\Om}\cM_g(\kap;m)$ and $\gam \in \Del(\wt{\om})$. Let
\be
\cT(\kap;m) \ra \wt{\Om}\cM_g(\kap;m)
\ee
be the bundle of holomorphic $1$-forms equipped with a pair in $S(\wt{\om})$. Elements of $\cT(\kap;m)$ are denoted $(X,\wt{\om},T)$, where $(X,\wt{\om}) \in \wt{\Om}\cM_g(\kap;m)$ and $T \in S(\wt{\om})$. The actions of $\wt{\GL}^+(2,\R)$ on $\wt{\Om}\cM_g(\kap;m)$ and $\wt{\C}^\ast$ induce actions on $\cS(\kap;m)$ and $\cT(\kap;m)$.

We will implicitly regard elements of $\Del(\wt{\om})$ and $S(\wt{\om})$ as elements of $\wt{\C}^\ast_{m+1}$ and $\wt{\C}^\ast_{m+1} \times \C^\ast$, respectively, using the inclusions above. We then obtain $\wt{\GL}^+(2,\R)$-equivariant inclusions
\be
\cS(\kap;m) \hra \wt{\Om}\cM_g(\kap;m) \times \wt{\C}^\ast_{m+1}
\ee
and
\be
\cT(\kap;m) \hra \wt{\Om}\cM_g(\kap;m) \times \wt{\C}^\ast_{m+1} \times \C^\ast
\ee
that respect the projections to $\wt{\Om}\cM_g(\kap;m)$. \\

\paragraph{\bf Splitting a zero.} Suppose $m \geq 1$, and fix $1 \leq j \leq m$. Given $(X,\wt{\om},\gam) \in \cS(\kap;m)$, let $I = [0,\int_\gam \om]$ be the oriented segment in $\C$ from $0$ to $\int_\gam \om$, and let
\be
\gam_1,\dots,\gam_{j+1} : I \ra X
\ee
be the isometric embeddings that preserve the direction of $I$, such that for $1 \leq k \leq j+1$, $\gam_k(0)$ is the distinguished zero $Z$ and the counterclockwise angle around $Z$ from $\gam$ to $\gam_k(I)$ is $2\pi(k-1)$. In particular, $\gam_1(I) = \gam$. Since $\int_{\gam_k(I)} \om \in \Del(\om)$, the segments $\gam_k(I)$ are disjoint from $Z(\om)$ and from each other except at their common starting point. Slit $X$ along $\gam_1(I) \cup \cdots \cup \gam_{j+1}(I)$ to obtain a surface with boundary $X_0$, and let $\gam_k^+ : I \ra X_0$ and $\gam_k^- : I \ra X_0$ be the left and right edges of the slit coming from $\gam_k$, respectively. Glue $\gam_k^+(z)$ to $\gam_{k+1}^-(z)$ for $1 \leq k \leq j$, and glue $\gam_{j+1}^+(z)$ to $\gam_1^-(z)$. The complex structure and the holomorphic $1$-form on the interior of $X_0$ extend over the slits to give a holomorphic $1$-form $(X^\pr,\om^\pr)$. If $j < m$, then $|Z(\om^\pr)| = |Z(\om)| + 1$ and the distinguished zero $Z$ is split into two zeros joined by a single saddle connection $\gam^\pr$ such that
\be
\int_{\gam^\pr} \om^\pr = \int_\gam \om .
\ee
The order of $\om^\pr$ at the starting point of $\gam^\pr$ is $m - j$, and the order of $\om^\pr$ at the ending point of $\gam^\pr$ is $j$. Let $\kap^\pr$ be the partition of $2g-2$ given by the orders of the zeros of $\om^\pr$. If $j < m$, then
\be
\kap^\pr = (\kap \sm (m)) \cup (m - j,j) ,
\ee
and if $j = m$, then $\kap^\pr = \kap$. We regard $(X^\pr,\om^\pr)$ as an element of $\Om\cM_g(\kap^\pr)$, and we say that $(X^\pr,\om^\pr)$ arises from $(X,\om)$ by {\em splitting a zero}. See Figure \ref{fig:split} for an example. The above surgery defines a {\em zero splitting map}
\be
\Phi = \Phi(\kap;m,j) : \cS(\kap;m) \ra \Om\cM_g(\kap^\pr)
\ee
that is a local covering of orbifolds and is equivariant for the actions of $\wt{\GL}^+(2,\R)$ and $\GL^+(2,\R)$. In suitable local period coordinates on $\wt{\Om}\cM_g(\kap;m)$ and $\Om\cM_g(\kap^\pr)$, the map $\Phi$ can be viewed as a map defined on an open subset of $\C^{2g+|\kap|-1} \times \C$ that concatenates the two inputs. The zero splitting map preserves the area of the underlying holomorphic $1$-form. \\

\begin{figure}
    \centering
    \includegraphics[width=0.9\textwidth]{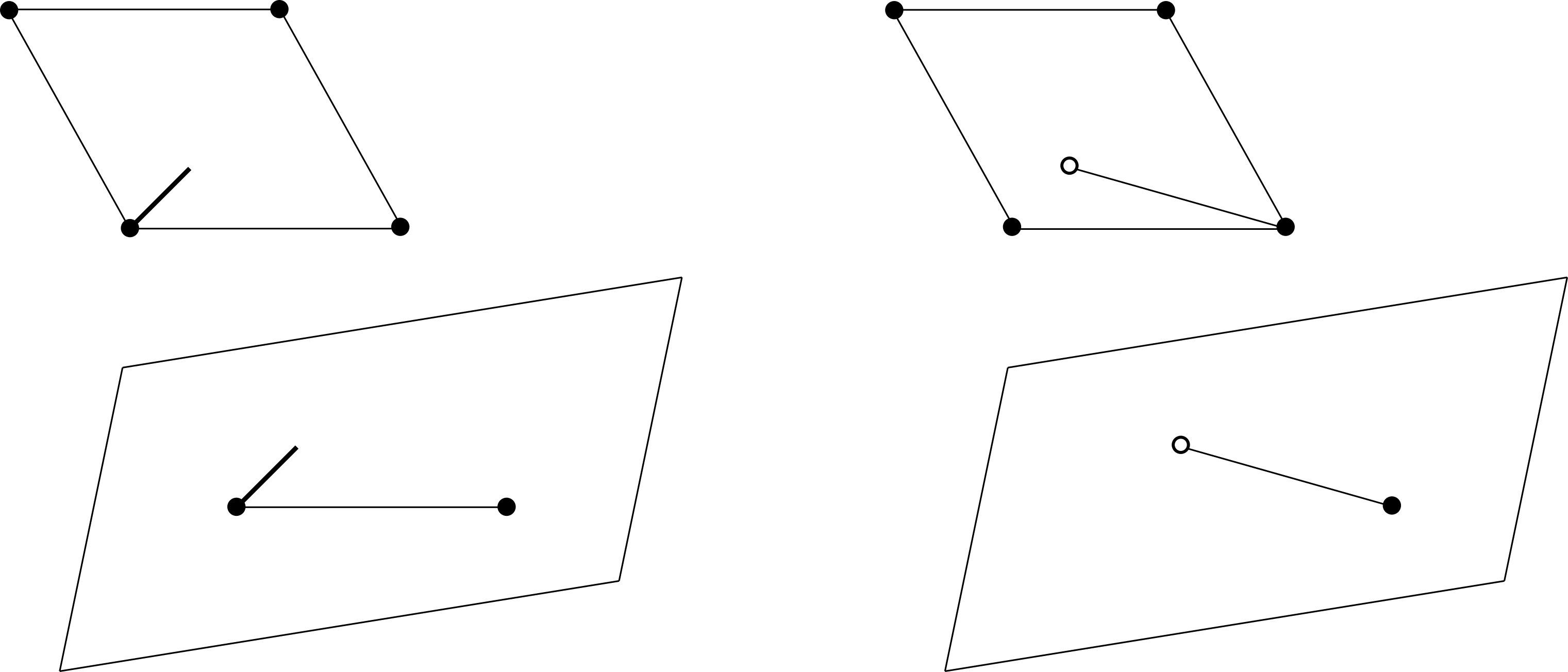}
    \caption{A holomorphic $1$-form in $\Om\cM_2(1,1)$ (right) that arises from a holomorphic $1$-form in $\Om\cM_2(2)$ (left) by splitting a zero. The two segments being slit are shown in bold.}
    \label{fig:split}
\end{figure}

\paragraph{\bf Connected sums with a torus.} Given $(X,\wt{\om},T) \in \cT(\kap;m)$ with $T = (\gam,w)$, let $I = [0,\int_\gam \om]$ be the oriented segment in $\C$ from $0$ to $\int_\gam \om$. The pair $(\gam,w)$ determines a flat torus
\be
T_0 = (\C/(\Z \int_\gam \om + \Z w), dz) .
\ee
Let $\gam_1 : I \ra X$ be the isometric embedding that preserves the direction of $I$ and satisfies $\gam_1(I) = \gam$. Let $\gam_2 : I \ra T_0$ be the projection of $I$, which gives a closed geodesic in $T_0$. Slit $X$ along $\gam_1(I)$ to obtain a surface with boundary $X_0$, and let $\gam_1^+ : I \ra X_0$ and $\gam_1^- : I \ra X_0$ be the left and right edges of the slit coming from $\gam_1$, respectively. Slit $T_0$ along $\gam_2(I)$ to obtain a cylinder with boundary $C_0$, and let $\gam_2^+ : I \ra C_0$ and $\gam_2^- : I \ra C_0$ be the left and right edges of the slit coming from $\gam_2$, respectively. Glue $\gam_1^+(z)$ to $\gam_2^-(z)$, and glue $\gam_2^+(z)$ to $\gam_1^-(z)$. The result is a holomorphic $1$-form $(X^\pr,\om^\pr)$ with a pair of homologous saddle connections $\gam^\pm$ forming a figure-eight on $X^\pr$ and arising from $\gam_1^\pm(I) \subset X_0$. The order of $\om^\pr$ at the zero $Z^\pr$ arising from the distinguished zero $Z$ is $m + 2$. The counterclockwise angle around $Z^\pr$ from the end of $\gam^-$ to the end of $\gam^+$ is $2\pi$. Let
\be
\kap^\pr = (\kap \sm (m)) \cup (m+2)
\ee
be the partition of $2g$ given by the orders of the zeros of $\om^\pr$. We regard $(X^\pr,\om^\pr)$ as an element of $\Om\cM_{g+1}(\kap^\pr)$, and we say that $(X^\pr,\om^\pr)$ arises from $(X,\om)$ by a {\em connected sum with a torus}. A pair of homologous saddle connections that presents $(X^\pr,\om^\pr)$ as a connected sum with a torus is a {\em splitting} of $(X^\pr,\om^\pr)$. See Figure \ref{fig:sum} for an example. The above surgery defines a {\em connected sum map}
\be
\Psi = \Psi(\kap;m) : \cT(\kap;m) \ra \Om\cM_{g+1}(\kap^\pr)
\ee
that is a local covering of orbifolds and is equivariant for the actions of $\wt{\GL}^+(2,\R)$ and $\GL^+(2,\R)$. In suitable local period coordinates, $\Psi$ can be viewed as a map defined on an open subset of $\C^{2g+|\kap|-1} \times \C^2$ that concatenates the two inputs. A connected sum of a holomorphic $1$-form of area $A_1 > 0$ with a flat torus of area $A_2 > 0$ has area $A_1 + A_2$. \\

\begin{figure}
    \centering
    \includegraphics[width=0.5\textwidth]{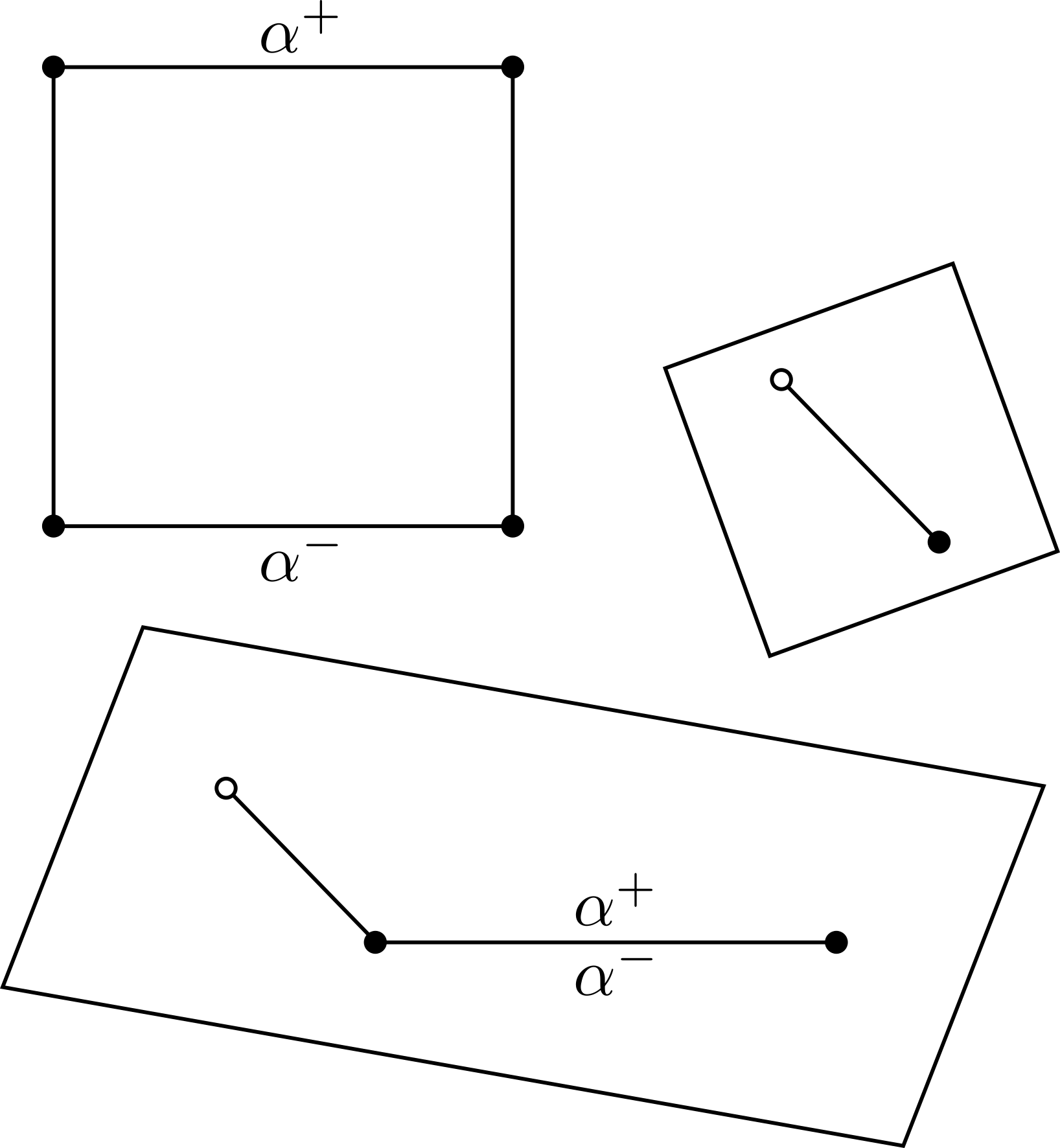}
    \caption{A holomorphic $1$-form $(X,\om) \in \Om\cM_3(3,1)$ arising from a holomorphic $1$-form in $\Om\cM_2(1,1)$ by a connected sum with a torus. The pair of rightward homologous saddle connections $\al^{\pm}$ is a splitting of $(X,\om)$.}
    \label{fig:sum}
\end{figure}

\paragraph{\bf Surgeries and stratum components.} For our inductive arguments, we will need to understand the relationship between the stratum components containing $(X^\pr,\om^\pr)$ and $(X,\om)$, respectively, when $(X^\pr,\om^\pr)$ arises from $(X,\om)$ by splitting a zero or by a connected sum with a torus. We refer to \cite{EMZ:principal} and \cite{KZ:components} for more general results.

\begin{lem} \label{lem:split} Let $\Om\cM_g(\kap^\pr)$ be a connected stratum with $|\kap^\pr| \geq 3$. There is a connected stratum $\Om\cM_g(\kap)$ with $|\kap| = |\kap^\pr| - 1$ such that $\Om\cM_g(\kap^\pr)$ contains holomorphic $1$-forms that arise from holomorphic $1$-forms in $\Om\cM_g(\kap)$ by splitting a zero.
\end{lem}

\begin{proof}
Note that $g \geq 3$ since $|\kap^\pr| \geq 3$. By Corollary \ref{cor:KZconn}, there is $m^\pr \in \kap^\pr$ such that $m^\pr$ is odd and not equal to $g - 1$. Choose $m_1,m_2 \in \kap^\pr \sm (m^\pr)$, let $m = m_1 + m_2$, and let $\kap = (\kap^\pr \sm (m_1,m_2)) \cup (m)$. We have $|\kap| = |\kap^\pr| - 1$, and $\Om\cM_g(\kap)$ is connected by Corollary \ref{cor:KZconn} since $m^\pr \in \kap$. By splitting a zero of order $m$ into two zeros of orders $m_1$ and $m_2$, respectively, we obtain holomorphic $1$-forms in $\Om\cM_g(\kap^\pr)$ that arise from holomorphic $1$-forms in $\Om\cM_g(\kap)$ by splitting a zero.
\end{proof}

\begin{lem} \label{lem:splitspin}
If $\cC^\pr$ is an even (respectively, odd) component of $\Om\cM_g(\kap^\pr)$ where $|\kap^\pr| \geq 3$, then there is an even (respectively, odd) component $\cC$ of a stratum $\Om\cM_g(\kap)$ with $|\kap| = |\kap^\pr| - 1$ such that $\cC^\pr$ contains holomorphic $1$-forms that arise from holomorphic $1$-forms in $\cC$ by splitting a zero.
\end{lem}

\begin{proof}
By Theorem \ref{thm:KZ}, since $|\kap^\pr| \geq 3$, the parts of $\kap^\pr$ are even and $\Om\cM_g(\kap^\pr)$ has a unique even component and a unique odd component. Choose $m_1,m_2 \in \kap^\pr$, let $m = m_1 + m_2$, and let $\kap = (\kap^\pr \sm (m_1,m_2)) \cup (m)$. We have $|\kap| = |\kap^\pr| - 1$, and by Theorem \ref{thm:KZ}, since the parts of $\kap$ are even, $\Om\cM_g(\kap)$ contains an even component and an odd component. Splitting a zero is a local surgery that only modifies a holomorphic $1$-form $(X,\om) \in \Om\cM_g(\kap)$ in a contractible neighborhood of a zero, so the parity of the spin structure $\phi(\om)$ is preserved. Let $\cC$ be an even or odd component of $\Om\cM_g(\kap)$, according to whether $\cC^\pr$ is an even or odd component of $\Om\cM_g(\kap^\pr)$. By splitting a zero of order $m$ into two zeros of orders $m_1$ and $m_2$, respectively, we obtain holomorphic $1$-forms in $\cC^\pr$ that arise from holomorphic $1$-forms in $\cC$ by splitting a zero.
\end{proof}

A holomorphic $1$-form that has a splitting cannot lie in the hyperelliptic component $\cC$ of $\Om\cM_g(g-1,g-1)$. For $(X,\om) \in \cC$, the hyperelliptic involution exchanges the two zeros of $\om$. Moreover, as shown in Lemma 2.1 in \cite{Lin:hyperelliptic}, since we can increase the height of a cylinder on $(X,\om)$ arbitrarily while remaining in the same stratum component, the hyperelliptic involution preserves every cylinder on $(X,\om)$. Therefore, every cylinder contains both zeros of $\om$ in its boundary. A cylinder arising from a splitting of $(X,\om)$ would be bounded by a pair of saddle connections that form a figure-eight at a zero of $\om$, so such cylinders do not occur in $(X,\om)$.

\begin{lem} \label{lem:sum}
Let $\Om\cM_{g+1}(\kap^\pr)$ be a stratum with $g + 1 \geq 3$ such that $|\kap^\pr| = 2$ and the elements of $\kap^\pr$ are odd. There is a connected stratum $\Om\cM_g(\kap)$ with $|\kap| = 2$ such that the nonhyperelliptic component of $\Om\cM_{g+1}(\kap^\pr)$ contains holomorphic $1$-forms that arise from holomorphic $1$-forms in $\Om\cM_g(\kap)$ by a connected sum with a torus.
\end{lem}

\begin{proof}
By Theorem \ref{thm:KZ}, since the elements of $\kap^\pr$ are odd, $\Om\cM_{g+1}(\kap^\pr)$ has a unique nonhyperelliptic component. If $g+1 \geq 4$, then there is $m^\pr \in \kap^\pr$ such that $m^\pr \geq 3$ and such that the elements of $\kap = (\kap^\pr \sm (m^\pr)) \cup (m^\pr-2)$ are odd and distinct. If $g+1 = 3$, then $\kap^\pr = (3,1)$ and we let $m^\pr = 3$ and $\kap = (1,1)$. In either case, $\Om\cM_g(\kap)$ is connected by Corollary \ref{cor:KZconn}, and by applying the connected sum construction at a zero of order $m^\pr - 2$, we obtain holomorphic $1$-forms in $\Om\cM_{g+1}(\kap^\pr)$ that arise from holomorphic $1$-forms in $\Om\cM_g(\kap)$ by a connected sum with a torus.
\end{proof}

\begin{lem} \label{lem:sumspin}
Let $\Om\cM_{g+1}(\kap^\pr)$ be a stratum with $g + 1 \geq 4$ such that $|\kap^\pr| = 2$ and the elements of $\kap^\pr$ are even. If $g + 1 \geq 5$, then there is a stratum $\Om\cM_g(\kap)$ with $|\kap| = 2$ such that the elements of $\kap$ are even, and such that the even (respectively, odd) nonhyperelliptic component $\cC^\pr$ of $\Om\cM_{g+1}(\kap^\pr)$ contains holomorphic $1$-forms that arise from holomorphic $1$-forms in the even (respectively, odd) nonhyperelliptic component $\cC$ of $\Om\cM_g(\kap)$ by a connected sum with a torus. If $g + 1 = 4$, then the same holds except that $\cC$ is hyperelliptic in the even case.
\end{lem}

\begin{proof}
Since $g +1 \geq 4$, there is $m^\pr \in \kap^\pr$ such that $m^\pr \geq 4$. Let $\kap = (\kap^\pr \sm (m^\pr)) \cup (m^\pr - 2)$. By Theorem \ref{thm:KZ}, $\Om\cM_g(\kap)$ has an even nonhyperelliptic component and an odd nonhyperelliptic component, unless $g + 1 = 4$ in which case the even component is hyperelliptic. By Lemma 11 in \cite{KZ:components}, applying the connected sum construction at a zero of order $m^\pr - 2$ does not change the parity of the associated spin structure. Therefore, $\cC^\pr$ contains holomorphic $1$-forms that arise from holomorphic $1$-forms in $\Om\cM_g(\kap)$ with the same spin parity by a connected sum with a torus.
\end{proof}

\paragraph{\bf Surgeries and constrained absolute periods.} Lastly, we will need to use zero splitting maps and connected sum maps in settings where the absolute periods are restricted to lie in certain closed subgroups of $\C$. Let $\Lam$ be a proper closed subgroup of $\C$ that is not discrete and that contains a lattice in $\C$. There is $M \in \SL(2,\R)$ such that $\Lam = M \cdot (\R + i\Z)$. Let $\Lam_0$ be the identity component of $\Lam$, so $\Lam_0 = M \cdot \R$. Define
\be
\Om^\Lam\cM_g(\kap) = \{(X,\om) \in \Om\cM_g(\kap) : \Per(\om) + \Lam_0 = \Lam \} .
\ee
A holomorphic $1$-form $(X,\om) \in \Om\cM_g(\kap)$ lies in $\Om^\Lam\cM_g(\kap)$ if and only if $\Per(\om) \subset \Lam$ and $\Per(\om)$ intersects every connected component of $\Lam$. We denote by $\Om_1^\Lam\cM_g(\kap)$ the area-$1$ locus in $\Om^\Lam\cM_g(\kap)$. For $\cC$ a connected component of $\Om\cM_g(\kap)$, we similarly define $\cC^\Lam$ and $\cC_1^\Lam$. The preimage in $\cS(\kap;m)$ of $\Om^\Lam\cM_g(\kap^\pr)$ under the zero splitting map is given by
\be
\cS^\Lam(\kap;m) = \{(X,\wt{\om},\gam) \in \cS(\kap;m) : \Per(\om) + \Lam_0 = \Lam\} .
\ee
The preimage in $\cT(\kap;m)$ of $\Om^\Lam\cM_g(\kap^\pr)$ under the connected sum map is slightly more complicated, since we are adding new absolute periods in the process of forming a connected sum with a torus, and is given by
\be
\cT^\Lam(\kap;m) = \{(X,\wt{\om},(\gam,w)) \in \cT(\kap;m) : \Per(\om) + \Lam_0 + \Z \int_\gam \om + \Z w = \Lam \} .
\ee
For $(X,\wt{\om},(\gam,w)) \in \cT^\Lam(\kap;m)$, we have $\Per(\om) + \Lam_0 = n_1 \Lam$ and $\Lam_0 + \Z\int_\gam \om + \Z w = n_2 \Lam$ for some $n_1,n_2 \in \Z_{>0}$ with $\gcd(n_1,n_2) = 1$.


\section{The absolute period foliation and surgeries} \label{sec:leaves}

We review the absolute period foliation of a stratum of holomorphic $1$-forms. We then study how leaves of these foliations behave under passing to the finite covers of strata from Section \ref{sec:surgery}, and we study the behavior of leaves when applying the surgeries from Section \ref{sec:surgery}. In the literature, the absolute period foliation is also referred to as the isoperiodic foliation, the Rel foliation, and the kernel foliation. For related discussions and further background, we refer to \cite{BSW:horocycle}, \cite{CDF:transfer}, \cite{McM:navigating}, \cite{McM:isoperiodic}, \cite{Zor:survey}. \\

\paragraph{\bf The period map.} For $X \in \cM_g$, a {\em marking} of $H^1(X;\C)$ is a symplectic isomorphism $m : H^1(S_g;\C) \ra H^1(X;\C)$ that sends $H^1(S_g;\Z)$ to $H^1(X;\Z)$. Let $\cS_g \ra \cM_g$ be the {\em Torelli cover} of moduli space, whose points are closed Riemann surfaces equipped with a marking of their cohomology. Let $\Om\cS_g \ra \cS_g$ be the associated bundle of nonzero holomorphic $1$-forms. The space $\Om\cS_g$ decomposes into strata $\Om\cS_g(\kap)$ indexed by partitions $\kap = (m_1,\dots,m_n)$ of $2g-2$. The {\em period map}
\be
\Per_g : \Om\cS_g \ra H^1(S_g;\C), \quad (X,\om,m) \mapsto m^{-1}([\om])
\ee
sends a holomorphic $1$-form to the associated cohomology class on $S_g$. The period map is a holomorphic submersion, and the connected components of nonempty fibers of $\Per_g$ are leaves of a holomorphic foliation of $\Om\cS_g$. This foliation descends to a holomorphic foliation $\cA$ of $\Om\cM_g$, called the {\em absolute period foliation} of $\Om\cM_g$. The restriction of $\Per_g$ to a stratum $\Om\cS_g(\kap)$ is also a holomorphic submersion, and we similarly obtain a holomorphic foliation $\cA(\kap)$ of $\Om\cM_g(\kap)$, called the {\em absolute period foliation} of $\Om\cM_g(\kap)$. Leaves of $\cA(\kap)$ are immersed complex suborbifolds of dimension $|\kap| - 1$. \\

\paragraph{\bf Geometry of leaves.} Let $\Om\cM_g(\kap)$ be a stratum with $|\kap| > 1$. Fix $(X_0,\om_0) \in \Om\cM_g(\kap)$, and let $L$ be the leaf of $\cA(\kap)$ through $(X_0,\om_0)$. We will sometimes denote this leaf by $L(\om_0)$. Let $v = (1,\dots,1) \in \C^{|\kap|}$, let $\X = \C^{|\kap|}/\C v$, and let $G = (\C^{|\kap|}/\C v) \rtimes {\rm Sym}(|\kap|)$, where the symmetric group ${\rm Sym}(|\kap|)$ acts on vectors by permuting their components. Choose an open disk $U \subset L$ containing $(X_0,\om_0)$, a labelling $Z_1,\dots,Z_{|\kap|}$ of $Z(\om)$, a point $x \in X_0$, and paths $\gam_j$ from $x$ to $Z_j$. The {\em relative period map}
\begin{equation} \label{eq:relative}
\rho : U \ra \X, \quad (X,\om) \mapsto \left(\int_{\gam_1}\om,\dots,\int_{\gam_{|\kap|}}\om\right)
\end{equation}
provides local coordinates on $U$. The map $\rho$ is independent of the choice of $x$, in the sense that choosing a path $\gam_0$ from $x^\pr$ to $x$ and replacing each $\gam_j$ with $\gam_0 \cup \gam_j$ does not change $\rho$. Different choices of labellings and paths may permute the components of $\rho$ and may translate the components of $\rho$ by absolute periods, which are constant on $L$. Thus, $L$ has a $(G,\mathbb{X})$-structure, and in particular a canonical locally Euclidean metric. In general, this metric is incomplete, since the holonomy of a saddle connection with distinct endpoints may approach $0$ along a path in $L$ of finite length. For all $M \in \GL^+(2,\R)$, the action of $\GL^+(2,\R)$ on $\Om\cM_g(\kap)$ induces a homeomorphism $L \ra M \cdot L$ to another leaf of $\cA(\kap)$, and this homeomorphism is affine in the coordinates provided by relative period maps. \\

\paragraph{\bf Finite covers and surgeries.} Choose $m \in \kap$, and let $p : \wt{\Om}\cM_g(\kap;m) \ra \Om\cM_g(\kap)$ be the stratum cover by prong-marked holomorphic $1$-forms from (\ref{eq:prong}). The foliation $\cA(\kap)$ lifts to a foliation $\cA(\kap;m)$ of $\wt{\Om}\cM_g(\kap;m)$, the {\em absolute period foliation} of $\wt{\Om}\cM_g(\kap;m)$. The action of $\wt{\GL}^+(2,\R)$ on $\wt{\Om}\cM_g(\kap;m)$ induces affine homeomorphisms between leaves of $\cA(\kap;m)$. We will sometimes denote the leaf of $\cA(\kap;m)$ through $(X,\wt{\om})$ by $L(\wt{\om})$.

Next, the foliation $\cA(\kap;m)$ lifts to a foliation $\cF_\cS$ of $\cS(\kap;m)$. The leaf of $\cF_\cS$ through $(X,\wt{\om},\gam)$ consists of the elements of $\cS(\kap;m)$ that can be reached from $(X,\wt{\om},\gam)$ by a path in $\cS(\kap;m)$ along which the absolute periods are constant. The segment $\gam$ may vary along the leaf. The foliation $\cA(\kap;m)$ also lifts to a foliation $\cF_\cT$ of $\cT(\kap;m)$. The leaf of $\cF_\cT$ through $(X,\wt{\om},T)$ consists of the elements of $\cT(\kap;m)$ that can be reached from $(X,\wt{\om},T)$ by a path in $\cT(\kap;m)$ along which the absolute periods and $T$ are constant.

Recall that our goal is to analyze the foliations $\cA(\kap)$ inductively, by using the surgeries from Section \ref{sec:surgery} to map a leaf $L$ of $\cA(\kap)$ into a leaf $L^\pr$ of $\cA(\kap^\pr)$. More precisely, we first lift $L$ to $\wt{\Om}\cM_g(\kap;m)$ by taking its preimage $L_1 = p^{-1}(L)$. In the case of splitting a zero, we consider the preimage $L_2$ of $L_1$ under the projection $\cS(\kap;m) \ra \wt{\Om}\cM_g(\kap;m)$, and then use a zero splitting map to map $L_2$ to a subset $L_3$ of a new stratum $\Om\cM_g(\kap^\pr)$ with $|\kap^\pr| = |\kap| + 1$. In the case of forming a connected sum with a torus, we fix $(\gam,w) \in \wt{\C}^\ast_{m+1} \times \C^\ast$ and take the preimage $L_2$ of $L_1$ in $\cT(\kap;m) \cap \left(\wt{\Om}\cM_g(\kap;m) \times \{(\gam,w)\}\right)$ under the projection to $\wt{\Om}\cM_g(\kap;m)$, and then use a connected sum map to map $L_2$ to a subset $L_3$ of a new stratum $\Om\cM_{g+1}(\kap^\pr)$ with $|\kap^\pr| = |\kap|$. Our hope is that $L_3$ is contained in a leaf of $\cA(\kap^\pr)$. However, there are two difficulties in pursuing this approach, and the purpose of the rest of this section is to address these difficulties. First, we need to show that $L_1$ is a single leaf of $\cA(\kap;m)$, as opposed to a disjoint union of leaves. We will show this is typically the case by showing that it holds for a dense open $\GL^+(2,\R)$-invariant set of leaves in $\Om\cM_g(\kap)$, using small loops around points in the metric completion of a leaf of $\cA(\kap;m)$ and some numerology. Second, in the case of connected sums with a torus, $L_2$ might not be a single leaf of $\cF_\cT$. When passing from $L_1$ to $L_2$, we remove some of the holomorphic $1$-forms in $L_1$ with a saddle connection parallel to $\gam$, and this can result in $L_2$ being highly disconnected. This issue appears to be much more subtle and is the main reason for the assumptions imposed on the absolute periods in our main theorems. \\

\paragraph{\bf Lifting leaves to finite covers.} We first address the question of when a leaf of $\cA(\kap)$ lifts to a single leaf of $\cA(\kap;m)$. We are assuming $|\kap| > 1$. Choose $\ell \in \kap \sm (m)$, $1 \leq j \leq \min(\ell+1,m+1)$, an ordered partition $\kap_1 = (a_1,\dots,a_j)$ of $m+1$ with $j$ parts, and an ordered partition $\kap_2 = (b_1,\dots,b_j)$ of $\ell+1$ with $j$ parts. We define $\wt{A}(\kap,\kap_1,\kap_2)$ to be the set of $(X,\wt{\om}) \in \wt{\Om}\cM_g(\kap;m)$ with a collection of $j$ homologous saddle connections $\gam_1,\dots,\gam_j$ from the distinguished zero $Z$ to a different zero $Z^\pr$ of order $\ell$, cyclically ordered in counterclockwise order around $Z$, and with the following properties.
\begin{enumerate}
    \item If $\gam_k$ has length $\eps > 0$ for $1 \leq k \leq j$, then every other saddle connection on $(X,\wt{\om})$ has length at least $3\eps$.
    \item Let $X_1,\dots,X_j$ be the connected components of $X \sm (\gam_1 \cup \cdots \cup \gam_j)$, where $X_k$ is bounded by $\gam_k \cup \gam_{k+1}$, indices taken modulo $j$. The cone angle around $Z$ inside $X_k$ is $2\pi a_k$, and the cone angle around $Z^\pr$ inside $X_k$ is $2\pi b_k$.
\end{enumerate}
Note that homologous saddle connections have the same holonomy, and thus the same length. We also define $A(\kap,\kap_1,\kap_2) = p(\wt{A}(\kap,\kap_1,\kap_2))$. A collection of saddle connections as above persists on an open neighborhood, so $\wt{A}(\kap,\kap_1,\kap_2)$ and $A(\kap,\kap_1,\kap_2)$ are open subsets of $\wt{\Om}\cM_g(\kap;m)$ and $\Om\cM_g(\kap)$, respectively.

The question of which configurations of homologous saddle connections can occur on a holomorphic $1$-form in a given connected component of $\Om\cM_g(\kap)$ was studied in detail in \cite{EMZ:principal}. As a consequence of some special cases of their results, we have the following.

\begin{lem} \label{lem:nonempty} Let $\Om\cM_g(\kap)$ be a stratum with $|\kap| > 1$, and fix $m \in \kap$.
\begin{enumerate}
\item For all $\ell \in \kap \sm (m)$, $A(\kap,(m+1),(\ell+1))$ intersects each component of $\Om\cM_g(\kap)$.
\item If some $m_j \in \kap$ is odd and $g \geq 3$, then for all $\ell \in \kap \sm (m)$, $A(\kap,(m,1),(\ell,1))$ intersects the nonhyperelliptic component of $\Om\cM_g(\kap)$.
\item If all $m_j \in \kap$ are even and $g \geq 5$, then for all $\ell \in \kap \sm (m)$, $A(\kap,(m-1,1,1),(\ell-1,1,1))$ intersects both nonhyperelliptic components of $\Om\cM_g(\kap)$.
\end{enumerate}
\end{lem}

\begin{proof}
Each of statements (1), (2), and (3) in Lemma \ref{lem:nonempty} follows from Lemmas 9.1, 10.2, and 10.3 in \cite{EMZ:principal}. Statement (1) is part of the case of these lemmas where $p = 1$ in the notation of \cite{EMZ:principal}. Statement (2) is part of the case where $p = 2$. Statement (3) is part of the case where $p = 3$. We will sketch how to construct holomorphic $1$-forms in the intersection of $A(\kap,\kap_1,\kap_2)$ with a component of $\Om\cM_g(\kap)$ from each statement.

For statement (1), let $\kap^\pr = (\kap \sm (\ell,m)) \cup (\ell+m)$. By \cite{KZ:components}, each component of $\Om\cM_g(\kap)$ can be accessed from a component of $\Om\cM_g(\kap^\pr)$ by splitting a zero of order $\ell+m$ into two zeros of orders $\ell$ and $m$. By making the resulting saddle connection sufficiently short, we obtain holomorphic $1$-forms in $A(\kap,(m+1),(\ell+1))$ in each component of $\Om\cM_g(\kap)$.

For statement (2), we consider several cases. If $\ell > 1$ and $m > 1$, let $\kap^\pr = (\kap \sm (\ell,m)) \cup (\ell + m - 2)$. Since some $m_j \in \kap$ is odd, we have $g \geq 4$, so by Theorem \ref{thm:KZ} we can choose a holomorphic $1$-form $(X,\om)$ in a nonhyperelliptic component of $\Om\cM_{g-1}(\kap^\pr)$. Split a zero of $\om$ of order $\ell + m - 2$ into two zeros of orders $\ell - 1$ and $m - 1$, then slit along the resulting saddle connection and glue in a flat torus using a parallel slit of the same length to get a holomorphic $1$-form $(X^\pr,\om^\pr) \in \Om\cM_g(\kap)$. If exactly one of $\ell = 1$ or $m = 1$, let $\kap^\pr = (\kap \sm (\ell,m)) \cup (\ell + m - 2)$. If $g \geq 4$, choose $(X,\om)$ in a nonhyperelliptic component of $\Om\cM_{g-1}(\kap^\pr)$, and if $g = 3$, choose $(X,\om)$ in $\Om\cM_2(\kap^\pr)$. Slit $(X,\om)$ along a segment from a zero of order $\ell + m - 2$ to a regular point, and glue in a flat torus using a parallel slit of the same length to get $(X^\pr,\om^\pr) \in \Om\cM_g(\kap)$. Lastly, if $\ell = m = 1$, then $|\kap| \geq 3$ since $g \geq 3$, so let $\kap^\pr = \kap \sm (1,1)$. Choose $(X,\om) \in \Om\cM_{g-1}(\kap^\pr)$, slit along a segment joining two regular points, and glue in a flat torus using a parallel slit of the same length to get $(X^\pr,\om^\pr) \in \Om\cM_g(\kap)$. In all cases, $(X^\pr,\om^\pr)$ lies in the nonhyperelliptic component of $\Om\cM_g(\kap)$, and the two slits become a pair of homologous saddle connections on $(X^\pr,\om^\pr)$ that certify membership in $A(\kap,(m,1),(\ell,1))$.

For statement (3), we also consider several cases. If $\ell > 2$ and $m > 2$, let $\kap^\pr = (\kap \sm (\ell,m)) \cup (\ell + m - 4)$, and choose a holomorphic $1$-form $(X,\om) \in \Om\cM_{g-2}(\kap^\pr)$. Split a zero of $\om$ of order $\ell + m - 4$ into two zeros of orders $\ell - 2$ and $m - 2$, slit along the resulting saddle connection $\gam_1$, and let $\gam_1^\pm$ be the left and right sides of this slit. Slit two flat tori $T_2,T_3$ along slits $\gam_2,\gam_3$ that are parallel to $\gam_1$ and of the same length as $\gam_1$, and let $\gam_2^\pm,\gam_3^\pm$ be the resulting left and right sides. Then glue $\gam_j^+$ to $\gam_{j+1}^-$, indices taken modulo $3$, to get a holomorphic $1$-form $(X^\pr,\om^\pr) \in \Om\cM_g(\kap)$. If exactly one of $\ell = 2$ or $m = 2$, let $\kap^\pr = (\kap \sm (\ell,m)) \cup (\ell + m - 4)$, choose $(X,\om) \in \Om\cM_{g-2}(\kap^\pr)$, slit along a segment from a zero of order $\ell + m - 4$ to a regular point, and glue in two flat tori using parallel slits of the same length as before to get $(X^\pr,\om^\pr) \in \Om\cM_g(\kap)$. If $\ell = m = 2$, then $|\kap| \geq 3$ since $g \geq 5$, so let $\kap^\pr = \kap \sm (2,2)$. Choose $(X,\om) \in \Om\cM_{g-2}(\kap^\pr)$, slit along a segment joining two regular points, and glue in two flat tori using parallel slits of the same length as before to get $(X^\pr,\om^\pr) \in \Om\cM_g(\kap)$. In all cases, $(X^\pr,\om^\pr)$ lies in a nonhyperelliptic component of $\Om\cM_g(\kap)$ with the same spin parity as the component containing $(X,\om) \in \Om\cM_{g-2}(\kap^\pr)$, and the three slits become a triple of homologous saddle connections on $(X^\pr,\om^\pr)$ that certify membership in $A(\kap,(m-1,1,1),(\ell-1,1,1))$. Since $g - 2 \geq 3$, by Theorem \ref{thm:KZ} the stratum $\Om\cM_{g-2}(\kap^\pr)$ contains both an odd component and an even component. Thus, we have produced elements in $A(\kap,(m-1,1,1),(\ell-1,1,1))$ in both nonhyperelliptic components of $\Om\cM_g(\kap)$.
\end{proof}

See Figure \ref{fig:split} (right) for an illustration of Case 1 in the stratum $\Om\cM_2(1,1)$, where the saddle connection arises from the slits on the left. See Figure \ref{fig:case2} for an illustration of Case 2 in the stratum $\Om\cM_4(3,3)$.

\begin{figure}
    \centering
    \includegraphics[width=0.6\textwidth]{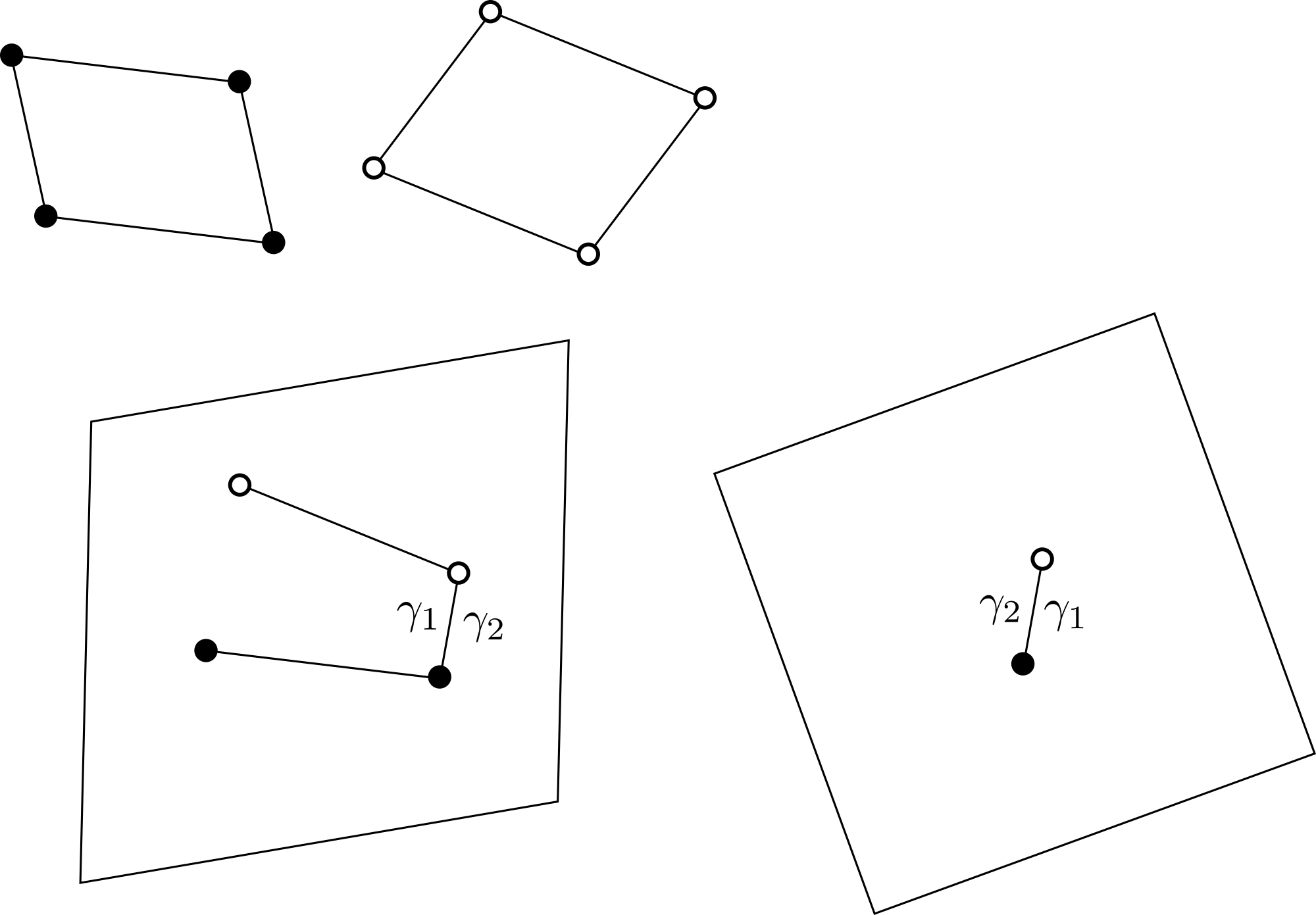}
    \caption{A holomorphic $1$-form in the intersection of $A((3,3),(3,1),(3,1))$ with the nonhyperelliptic component of $\Om\cM_4(3,3)$.}
    \label{fig:case2}
\end{figure}

The next lemma shows that leaves of $\cA(\kap)$ typically lift to leaves of $\cA(\kap;m)$.

\begin{lem} \label{lem:coverleaf}
Let $\Om\cM_g(\kap)$ be a stratum with $|\kap| > 1$. Fix $m \in \kap$, and let $p : \wt{\Om}\cM_g(\kap;m) \ra \Om\cM_g(\kap)$ be the stratum cover in (\ref{eq:prong}). There is an open $\GL^+(2,\R)$-invariant subset $A \subset \Om\cM_g(\kap)$ that intersects each connected component of $\Om\cM_g(\kap)$, such that if $L$ is a leaf of $\cA(\kap)$ that intersects $A$, then $p^{-1}(L)$ is a leaf of $\cA(\kap;m)$.
\end{lem}

\begin{proof}
Fix $\ell \in \kap \sm (m)$, $1 \leq j \leq \min(\ell+1,m+1)$, $\kap_1 = (a_1,\dots,a_j)$ an ordered partition of $m+1$, and $\kap_2 = (b_1,\dots,b_j)$ an ordered partition of $\ell+1$. Suppose that $A(\kap,\kap_1,\kap_2)$ is nonempty. Fix $(X,\om) \in A(\kap,\kap_1,\kap_2)$, fix $(X,\wt{\om}) \in p^{-1}(X,\om)$, and let $\gam_1,\dots,\gam_j$ be homologous saddle connections on $(X,\om)$ as in the definition of $A(\kap,\kap_1,\kap_2)$. Let $L$ be the leaf of $\cA(\kap)$ through $(X,\om)$, and let $\wt{L}$ be the leaf of $\cA(\kap;m)$ through $(X,\wt{\om})$.

By slitting $X$ along $\gam_1 \cup \cdots \cup \gam_j$ and gluing the left side of $\gam_k$ to the right side of $\gam_{k+1}$, indices taken modulo $j$, we obtain a finite collection of holomorphic $1$-forms, $(X_1,\om_1), \dots, (X_j,\om_j)$. Each $(X_k,\om_k)$ has an oriented geodesic segment $\del_k$ from a point $Z_k$ to a point $Z_k^\pr$ coming from gluing the left side of $\gam_k$ to the right side of $\gam_{k+1}$, and satisfying $\int_{\del_k} \om_k = \int_{\gam_k} \om$. The order of $\om_k$ at $Z_k$ is $a_k - 1$, and the order of $\om_k$ at $Z_k^\pr$ is $b_k - 1$. Let $\del_{k,1},\dots,\del_{k,b_k}$ be the oriented geodesic segments on $X_k$ starting at $Z_k^\pr$ such that
\be
\int_{\del_{k,r}} \om_k = -\int_{\del_k} \om_k
\ee
for $1 \leq r \leq b_k$, cyclically ordered in counterclockwise order around $Z_k^\pr$. We may assume that $\del_{k,1}$ is $\del_k$ with the opposite orientation. Slit $X_k$ along $\del_{k,1} \cup \cdots \cup \del_{k,b_k}$ and glue the left side of $\del_{k,r}$ to the right side of $\del_{k,r+1}$, indices taken modulo $b_k$, to obtain a holomorphic $1$-form $(X_k^\pr,\om_k^\pr)$. This surgery combines $Z_k,Z_k^\pr$ into a single zero $Z_k$, and when $a_k > 1$ or $b_k > 1$, $(X_k,\om_k)$ arises from $(X_k^\pr,\om_k^\pr)$ by splitting the zero $Z_k$. The order of $\om_k^\pr$ at $Z_k$ is $a_k + b_k - 2$, and the order of $\om_k^\pr$ at $Z_k^\pr$ is $0$. The union $(X_1^\pr,\om_1^\pr) \cup \cdots \cup (X_j^\pr,\om_j^\pr)$, with the points $Z_1,\dots,Z_j$ identified to a single node, can be viewed as a point in the metric completion of the leaf $L$. Note that in the special case $j = 1$, we are just combining two zeros of $\om$ together by collapsing the saddle connection $\gam_1$.

We can reverse the process above to recover $(X,\om)$ from the $(X_k^\pr,\om_k^\pr)$. More generally, for $1 \leq k \leq j$, choose a collection of oriented geodesic segments $\del_{k,1}^\pr,\dots,\del_{k,b_k}^\pr$ on $X_k^\pr$ starting at $Z_k$ such that $\del_{k,r}^\pr$ has length $\eps$ and the counterclockwise angle around $Z_k$ from $\del_{k,1}^\pr$ to $\del_{k,r}^\pr$ is $2\pi(r-1)$ for $1 \leq r \leq b_k$. Slit $X_k^\pr$ along $\del_{k,1}^\pr \cup \cdots \cup \del_{k,b_k}^\pr$ and glue the left side of $\del_{k,r}^\pr$ to the right side of $\del_{k,r+1}^\pr$, indices taken modulo $b_k$. The resulting holomorphic $1$-form has a distinguished oriented geodesic segment $\del_k$ coming from gluing the left side of $\del_{k,b_k}^\pr$ to the right side of $\del_{k,1}^\pr$, and $\del_k$ is a saddle connection when $a_k > 1$ and $b_k > 1$. Next, slit along the segments $\del_k$, $1 \leq k \leq j$, and glue the left side of $\del_k$ to the right side of $\del_{k+1}$, indices taken modulo $j$, to obtain a holomorphic $1$-form in $A(\kap,\kap_1,\kap_2)$.

The oriented geodesic segments of length $\eps$ on $(X_k^\pr,\om_k^\pr)$ starting at $Z_k$ are parameterized by $\R/2\pi(a_k + b_k - 1)\Z$. By rotating the chosen segments $\del_{k,1}^\pr,\dots,\del_{k,b_k}^\pr$ counterclockwise around $Z_k$ in the construction above, we obtain a family of holomorphic $1$-forms $s_k(t) = (X_{k,t},\om_{k,t})$ such that $s_k(0) = (X_k,\om_k)$ and
\be
\int_{\del_k} \om_{k,t} = e^{it} \int_{\del_k} \om_k .
\ee
Moreover, since $s_k(t)$ is obtained from $s_k(0)$ by only modifying a contractible neighborhood of $Z_k$, the absolute periods do not change. Thus, by slitting $s_k(t)$ along $\del_k$ for $1 \leq k \leq j$ and gluing the left side of $\del_k$ to the right side of $\del_{k+1}$, indices taken modulo $j$, we obtain a path
\be
s : \R \ra L
\ee
such that $s(0) = (X,\om)$, and such that for $t \in \R$ and $1 \leq k \leq j$, the holonomy of $\gam_k$ on $s(t)$ is given by $e^{it}\int_{\gam_k} \om$. Informally, $s(t)$ is obtained from $s(0)$ by rotating each saddle connection $\gam_k$ around its starting point $Z$ counterclockwise through an angle $t$. The image of $s$ is a small loop around a point in the metric completion of $L$. The choice of $(X,\wt{\om}) \in p^{-1}(X,\om)$ then determines a lift
\be
\wt{s} : \R \ra \wt{L}
\ee
such that $\wt{s}(0) = (X,\wt{\om})$ and $p(\wt{s}(t)) = s(t)$ for all $t \in \R$. See Figure \ref{fig:rotate31} for an example in $\wt{A}((3,1),(3,1),(1,1))$. 

\begin{figure}
    \centering
    \includegraphics[width=\textwidth]{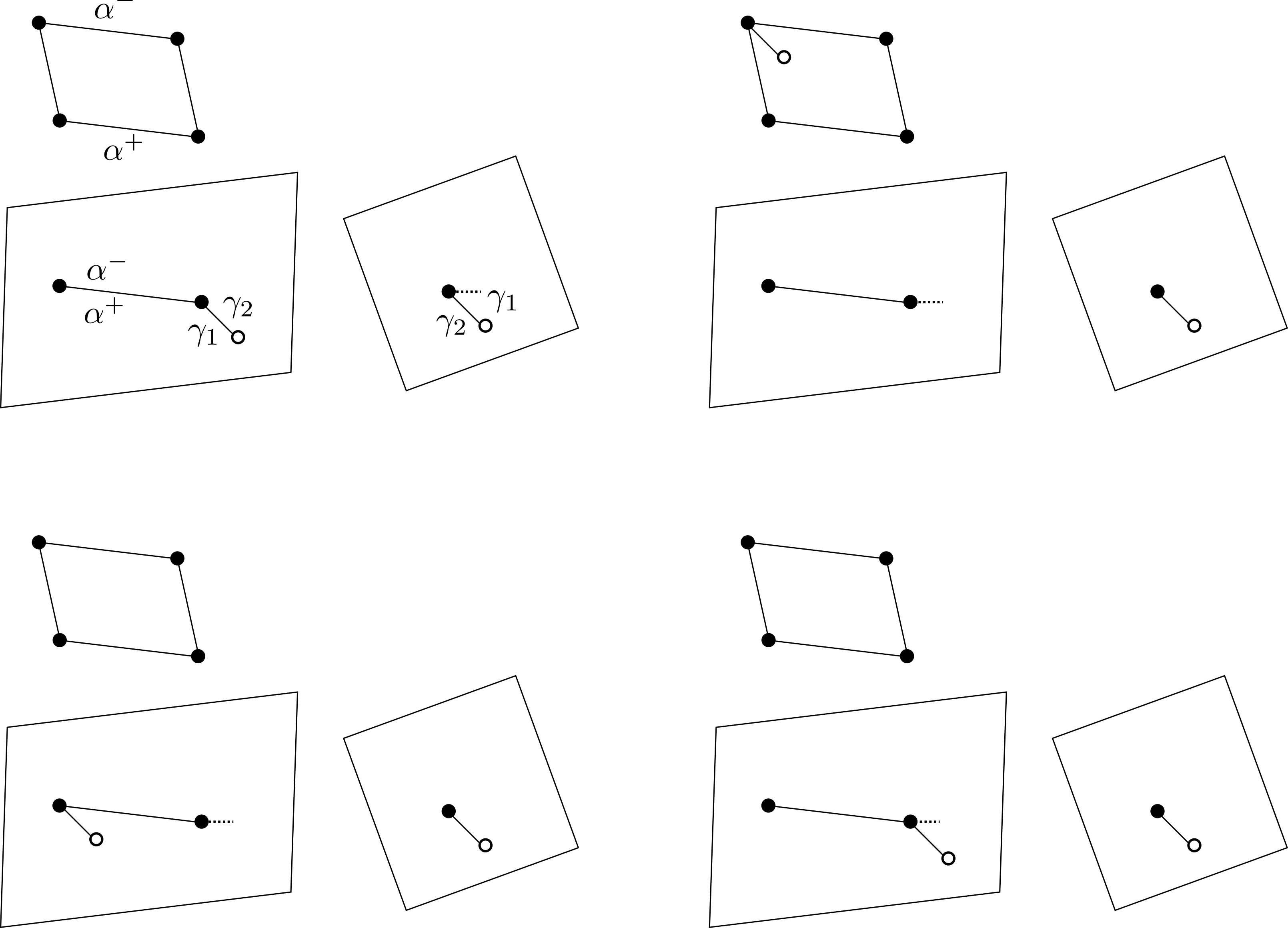}
    \caption{A loop $\wt{s} : \R \ra \wt{L}$ around a point in the metric completion of a leaf $\wt{L}$ of $\cA(\kap;m)$, where $\kap = (3,1)$, $m = 3$, and $\wt{s}(\R) \subset \wt{A}((3,1),(3,1),(1,1))$. Along $\wt{s}(\R)$, the saddle connections $\gam_1,\gam_2$ are rotated counterclockwise around the zero of order $3$. The $4$ images show $\wt{s}(0)$ (top-left), $\wt{s}(2\pi)$ (top-right), $\wt{s}(4\pi)$ (bottom-left), and $\wt{s}(6\pi)$ (bottom-right). We have $s(0) = s(6 \pi)$, but the prongs (shown with dashes) are different. In this case, the $4$ possible choices of prongs are realized in $p^{-1}(s(0)) = \{\wt{s}(0),\wt{s}(6\pi),\wt{s}(12\pi),\wt{s}(18\pi)\}$, and $\wt{s}(24\pi) = \wt{s}(0)$.}
    \label{fig:rotate31}
\end{figure}

Rotating these saddle connections counterclockwise through an angle $2\pi(a_k + b_k - 1)$ does not change $(X_k,\om_k)$, that is, $s_k(t + 2\pi(a_k + b_k - 1)) = s_k(t)$ for $t \in \R$. Therefore, letting
\be
N(\kap_1,\kap_2) = \lcm_{1 \leq k \leq j}(a_k + b_k - 1) ,
\ee
we have
\be
s(t) = s(t + 2\pi N(\kap_1,\kap_2))
\ee
for $t \in \R$. Letting $c(t)$ be the counterclockwise angle around $Z$ from the prong on $s(t)$ to the saddle connection $\gam_1$ on $s(t)$, we have $c(t) = c(0) + t$ for $t \in \R$. For $n \in \Z$, let $\tht_n$ be the prong on $(X,\wt{\om})$ such that the counterclockwise angle from $\tht(\wt{\om})$ to $\tht_n$ is $2\pi n$. Let $(X,\wt{\om}_n)$ be the element of $p^{-1}(X,\om)$ with $\tht(\wt{\om}_n) = \tht_n$. Then
\be
\wt{s}(2\pi N(\kap_1,\kap_2)) = (X,\wt{\om}_{-N(\kap_1,\kap_2)}) \in \wt{L} .
\ee
The cone angle around $Z$ is $2\pi(m+1)$, meaning $(X,\wt{\om}_{m+1}) = (X,\wt{\om})$.

Since the action of $\wt{\GL}^+(2,\R)$ on $\wt{\Om}\cM_g(\kap;m)$ respects leaves of $\cA(\kap;m)$, and since $p$ is equivariant for the actions of $\wt{\GL}^+(2,\R)$ and $\GL^+(2,\R)$, for any $(Y,\wt{\eta}) \in \wt{\Om}\cM_g(\kap;m)$ such that $L(\eta)$ intersects $\GL^+(2,\R) \cdot A(\kap,\kap_1,\kap_2)$, we similarly have $(Y,\wt{\eta}_{-N(\kap_1,\kap_2)}) \in L(\wt{\eta})$. Therefore, letting $N_1(\kap_1,\kap_2) = \gcd(m+1,N(\kap_1,\kap_2))$, we have
\be
(Y,\wt{\eta}_{n N_1(\kap_1,\kap_2)}) \in L(\wt{\eta})
\ee
for all $n \in \Z$ whenever $L(\eta)$ intersects $\GL^+(2,\R) \cdot A(\kap,\kap_1,\kap_2)$.

Now let $\cC$ be a connected component of $\Om\cM_g(\kap)$. Since $A(\kap,\kap_1,\kap_2) \cap \cC$ is open, it has positive measure whenever it is nonempty. By Lemma \ref{lem:nonempty}, there are ordered partitions $\kap_1,\kap_2$ as above such that $A(\kap,\kap_1,\kap_2) \cap \cC$ is nonempty. Let $A_\cC$ be the intersection of the finitely many nonempty subsets of the form $\GL^+(2,\R) \cdot (A(\kap,\kap_1,\kap_2) \cap \cC)$. By ergodicity of the $\GL^+(2,\R)$-action on $\cC$, we have that $A_\cC$ is nonempty. Moreover, $A_\cC$ is open and $\GL^+(2,\R)$-invariant. We claim that
\begin{equation} \label{eq:gcd1}
\gcd \left(\left\{m+1\right\} \cup \left\{N(\kap_1,\kap_2) : A_\cC \subset A(\kap,\kap_1,\kap_2)\right\}\right) = 1 .
\end{equation}
We verify this claim in 3 cases. \\

\paragraph{\bf Case 1:} Suppose that $\ell = m$. By Lemma \ref{lem:nonempty},  we have
\be
A_\cC \subset A(\kap,(m+1),(m+1)) .
\ee
Since $N((m+1),(m+1)) = 2m+1$, the $\gcd$ in (\ref{eq:gcd1}) divides $\gcd(m+1,2m+1) = 1$. \\

\paragraph{\bf Case 2:} Some part of $\kap$ is odd. By Case 1, we may assume that $\cC$ is nonhyperelliptic, which implies $g \geq 3$. Note that $\kap$ contains at least two odd parts, so we may assume that $\ell$ is odd. By Lemma \ref{lem:nonempty},
\be
A_\cC \subset A(\kap,(m+1),(\ell+1)), \quad A_\cC \subset A(\kap,(m,1),(\ell,1)) .
\ee
Since $\ell$ is odd and
\be
N((m+1),(\ell+1)) = m + \ell + 1, \quad N((m,1),(\ell,1)) = m + \ell - 1,
\ee
the $\gcd$ in (\ref{eq:gcd1}) divides
\be
\gcd(m+1,m+\ell+1,m+\ell-1) = \gcd(m+1,\ell,2) = 1.
\ee

\paragraph{\bf Case 3:} All parts of $\kap$ are even. By Case 1, we may assume that $\cC$ is nonhyperelliptic. If $g \geq 5$, then by Lemma \ref{lem:nonempty},
\be
A_\cC \subset A(\kap,(m+1),(\ell+1)), \quad A_\cC \subset A(\kap,(m-1,1,1),(\ell-1,1,1)) .
\ee
Since $m+1$ is odd and
\be
N((m-1,1,1),(\ell-1,1,1)) = m + \ell - 3 ,
\ee
the $\gcd$ in (\ref{eq:gcd1}) divides
\be
\gcd(m+1,m+\ell+1,m+\ell-3) = \gcd(m+1,\ell,4) = 1 .
\ee
If $g \leq 4$, then $\kap$ is one of $(2,2)$, $(2,2,2)$, $(4,2)$. The cases $(2,2)$ and $(2,2,2)$ are already covered by Case 1. For $\kap = (4,2)$, Lemma \ref{lem:nonempty} implies $A_\cC \subset A(\kap,(5),(3))$. Since $N((5),(3)) = 7$ and $\gcd(5,7) = \gcd(3,7) = 1$, the $\gcd$ in (\ref{eq:gcd1}) is $1$ in this case. \\

It remains to show that in the case where $\ell = m$, the leaf $\wt{L}$ also contains elements of $p^{-1}(X,\om)$ for which the prong is at a different zero of order $m$. In this case, $(X,\om) \in A(\kap,(m+1),(m+1))$ arises from a holomorphic $1$-form in $\Om\cM_g(\kap^\pr)$ by splitting a zero, where $\kap^\pr = (\kap \sm (m,m)) \cup (2m)$. Let $\gam_1$ be the resulting saddle connection. If we apply the zero splitting map again to $(X,\om)$, by slitting along the $m+1$ segments emanating from the starting point of $\gam_1$ with holonomy $-\int_{\gam_1} \om$, then $\gam_1$ is preserved while the starting point of $\gam_1$ is moved, and on the resulting holomorphic $1$-form $(X^\pr,\om^\pr)$, the holonomy of $\gam_1$ is $2 \int_{\gam_1} \om$. Here, $(X^\pr,\om^\pr) \in \Om\cM_g(\kap)$ as well. Rotating the choice of segments for both zero splitting operations simultaneously counterclockwise then gives us a small loop $\wt{s}_1 : \R \ra \wt{L}$, and we have $p(\wt{s}_1((m+1)\pi)) = p(\wt{s}_1(0))$. However, on $\wt{s}_1((m+1)\pi)$, the prong is on the other zero of order $m+1$. See Section 8.1 of \cite{EMZ:principal} for a similar discussion in strata with labelled singularities.

To conclude, let $A = \bigcup_\cC A_\cC$ where $\cC$ ranges over the connected components of $\Om\cM_g(\kap)$. Then $A$ is open, $\GL^+(2,\R)$-invariant, and intersects every connected component of $\Om\cM_g(\kap)$. Let $L$ be a leaf of $\cA(\kap)$ that intersects $A$, and fix $(X,\om) \in L \cap A$. By the claim in (\ref{eq:gcd1}) and the surgery in the previous paragraph, for any $(X,\wt{\om}) \in p^{-1}(X,\om)$, the leaf of $\cA(\kap;m)$ through $(X,\wt{\om})$ contains $p^{-1}(X,\om)$. Thus, $p^{-1}(L)$ is a leaf of $\cA(\kap;m)$.
\end{proof}

When $|\kap| > 1$, Lemma \ref{lem:coverleaf} implies that the preimage under $p$ of a connected component of $\Om\cM_g(\kap)$ is a connected component of $\wt{\Om}\cM_g(\kap;m)$. The same holds when $|\kap| = 1$, since in that case the orbit of $(X,\wt{\om})$ under the rotation subgroup of $\wt{\GL}^+(2,\R)$ contains $p^{-1}(X,\om)$. \\

\paragraph{\bf Splitting zeros along leaves.} Here, we show that leaves of $\cF_\cS$ admit a simple global description in terms of leaves of $\cA(\kap;m)$.

\begin{lem} \label{lem:splitleaf}
Let $L_\cS$ be the leaf of $\cF_\cS$ through $(X,\wt{\om},\gam)$. Then $(X^\pr,\wt{\om}^\pr,\gam^\pr) \in L_\cS$ if and only if $(X^\pr,\wt{\om}^\pr)$ is in the leaf of $\cA(\kap;m)$ through $(X,\wt{\om})$ and $\gam^\pr \in \Del(\wt{\om}^\pr)$.
\end{lem}

\begin{proof}
Let $\wt{L}$ be the leaf of $\cA(\kap;m)$ through $(X,\wt{\om})$, fix $(X^\pr,\wt{\om}^\pr) \in \wt{L}$, and fix $\gam^\pr \in \Del(\wt{\om}^\pr)$. Let $s : [0,1] \ra \wt{L}$ be a path such that $s(0) = (X,\wt{\om})$ and $s(1) = (X^\pr,\wt{\om}^\pr)$. Let $(X_t,\wt{\om}_t) = s(t)$. By compactness, there is $\eps > 0$ such that for all $t \in [0,1]$, every saddle connection on $s(t)$ has length at least $\eps$. Since $\Del(\wt{\om})$ is path-connected, there is a path $s_1 : [0,1] \ra L_\cS$ such that $s_1(0) = (X,\wt{\om},\gam)$ and $s_1(1) = (X,\wt{\om},\gam_1)$, where $\gam_1$ has length less than $\eps$. Using the natural inclusions $\Del(\wt{\om}_t) \hra \wt{\C}^\ast_{m+1}$, we obtain a well-defined path $\wt{s} : [0,1] \ra L_\cS$ given by $\wt{s}(t) = (s(t),\gam_1)$. Then since $\Del(\wt{\om}^\pr)$ is path-connected, there is a path $s_2 : [0,1] \ra L_\cS$ such that $s_2(0) = (X^\pr,\wt{\om}^\pr,\gam_1)$ and $s_2(1) = (X^\pr,\wt{\om}^\pr,\gam^\pr)$. By concatenating $s_1,\wt{s},s_2$, we see that $(X^\pr,\wt{\om}^\pr,\gam^\pr) \in L_\cS$. The other containment is clear by definition of $\cF_\cS$.
\end{proof}

Fix $1 \leq j < m$, let $\kap^\pr = (\kap \sm (m)) \cup (m-j,j)$, and consider the associated zero splitting map $\Phi : \cS(\kap;m) \ra \Om\cM_g(\kap^\pr)$. Splitting a zero is a local surgery that only modifies a holomorphic $1$-form in a contractible neighborhood of one of its zeros, so it does not change the absolute periods. Therefore, $\Phi$ sends leaves of $\cF_\cS$ into leaves of $\cA(\kap^\pr)$. \\

\paragraph{\bf Geodesics on leaves.} We will address the question of when leaves of $\cA(\kap;m)$ lift to leaves of $\cF_\cT$ in the case $|\kap| = 2$, which will be sufficient for our purposes. Before doing this, we study the geometry of leaves of $\cA(\kap)$ in the case $|\kap| = 2$ in greater detail. In this case, a leaf $L$ of $\cA(\kap)$ is a Riemann surface equipped with a canonical quadratic differential $q$. To describe $q$, fix $(X_0,\om_0) \in L$ and let $\gam$ be a saddle connection on $(X_0,\om_0)$ with distinct endpoints. Let $Z_1$ and $Z_2$ be the starting point of $\gam$ and the ending point of $\gam$, respectively. The map
\be
r : (X,\om) \mapsto \int_\gam \om \in \C
\ee
provides a local coordinate on $L$ near $(X_0,\om_0)$, and we have $q = dr^2$. For any $z \in \C^\ast$, there is a locally defined geodesic with respect to $|q|$ through $(X_0,\om_0)$,
\be
s : (-\eps,\eps) \ra L, \quad s(t) = (X_t,\om_t),
\ee
such that $\frac{d}{dt}\int_\gam \om_t = z$. The maximal domain of definition of $s$ is not necessarily $\R$. However, the only obstruction is the existence of a saddle connection on $(X_0,\om_0)$ with distinct endpoints and with holonomy in $\R z$.

\begin{cor} \label{cor:geodesic} (\cite{BSW:horocycle}, Corollary 6.2)
The maximal domain of definition of $s$ contains $t_0 \in \R$ if and only if $(X_0,\om_0)$ does not have a saddle connection from $Z_2$ to $Z_1$ with holonomy in $\{t t_0 z : 0 \leq t \leq 1\}$.
\end{cor}

A more general version of Corollary \ref{cor:geodesic} is proven in \cite{BSW:horocycle}, which applies to any stratum $\Om\cM_g(\kap)$ with $|\kap| > 1$. Note that \cite{BSW:horocycle} work with strata with labelled singularities. See also \cite{McM:navigating}, \cite{MW:cohomology}.

Fix $z \in \C^\ast$. Let $s : (a,b) \ra L$ be a geodesic with respect to $|q|$, and suppose that $(a,b)$ is the maximal domain of definition of $s$ and that $-\infty < a < b < +\infty$. Choose a square root $\sqrt{q}$ along $s(a,b)$. Corollary \ref{cor:geodesic} implies that $\int_{s(a,b)}\sqrt{q}$ is constrained by the absolute periods of the holomorphic $1$-forms in $L$.

\begin{lem} \label{lem:scleaf}
Fix $(X,\om) \in L$. If $s : (a,b) \ra L$ is a geodesic with respect to $|q|$ such that $(a,b)$ is the maximal domain of definition of $s$, and $-\infty < a < b < +\infty$, then
\be
\int_{s(a,b)} \sqrt{q} \in \Per(\om) .
\ee
\end{lem}

\begin{proof}
For $t \in (a,b)$, let $(X_t,\om_t) = s(t)$. By Corollary \ref{cor:geodesic}, there is $z \in \C^\ast$ and a consistent labelling $Z_1,Z_2$ of the zeros of $(X_t,\om_t)$, such that $(X_t,\om_t)$ has a saddle connection $\gam_1$ from $Z_1$ to $Z_2$ and another saddle connection $\gam_2$ from $Z_2$ to $Z_1$ with holonomies
\be
\int_{\gam_1} \om_t = (t-a)z, \quad \int_{\gam_2} \om_t = (b-t)z .
\ee
Then $\gam = \gam_1 \cup \gam_2$ is an oriented loop in $X_t$, and there is a choice of $\sqrt{q}$ along $s(a,b)$ for which
\be
\int_{s(a,b)}\sqrt{q} = (b-a)z = \int_\gam \om_t \in \Per(\om_t) = \Per(\om) .
\ee
\end{proof}

Segments $s$ as in Lemma \ref{lem:scleaf} can disconnect $L$ when removed. Lemma \ref{lem:scleaf} is not used in the rest of the paper, but is included to help motivate the following crucial lemma.

Fix $(X,\om) \in \Om\cM_g(\kap)$. When forming a connected sum with a torus, we are adding new absolute periods, one of which arises from the segment $\del$ being slit on $(X,\om)$. This makes the question of whether $L(\om)$ lifts to a leaf of $\cF_\cT$ subtle, due to the possible presence of a saddle connection $\gam$ on $(X,\om)$ that is parallel to and shorter than the slit $\del$. When $\gam$ is a closed loop, the connected sum construction may fail to be defined on large portions of $L(\om)$, but the holonomy of $\gam$ must be an absolute period of $\om$. Thus, if $\del$ is not parallel to an absolute period of $\om$, then $\gam$ must have distinct endpoints. In this case, the holonomy of $\gam$ changes as we move along $L(\om)$, and $\gam$ remains parallel to and shorter than $\del$ only on a geodesic segment in $L(\om)$. Unlike the setting of Lemma \ref{lem:scleaf}, this segment is not parallel to an absolute period of $\om$, and removing this segment from $L(\om)$ does not disconnect $L(\om)$.

The following lemma shows that leaves of $\cA(\kap)$ do not become disconnected after removing all of the holomorphic $1$-forms with a saddle connection parallel to and shorter than a given $z \in \C^\ast$, provided that $z$ is not parallel to any of the associated absolute periods. This provides a sufficient condition for lifting leaves of $\cA(\kap)$ to leaves of $\cF_\cT$.

\begin{lem} \label{lem:LIconn}
Let $\Om\cM_g(\kap)$ be a stratum with $|\kap| = 2$. Fix $(X,\om) \in \Om\cM_g(\kap)$, let $L$ be the leaf of $\cA(\kap)$ through $(X,\om)$, and let $q$ be the canonical quadratic differential on $L$. Fix $z \in \C^\ast$ such that
\be
z \notin \bigcup_{z_0 \in \Per(\om)} \R z_0 .
\ee
Let $I = \{tz : 0 \leq t \leq 1\}$, and let
\be
L(I) = \left\{(Y,\eta) \in L : \Gam(\eta) \cap I \neq \emptyset \right\} .
\ee
The subspace $L(I) \subset L$ is closed, and is a countable union of embedded isolated parallel line segments in the metric $|q|$. The complement $L \sm L(I)$ is path-connected.
\end{lem}

\begin{proof}
By Lemma \ref{lem:Iclosed}, the subset $\Om\cM_g(\kap;I)$ of holomorphic $1$-forms in $\Om\cM_g(\kap)$ with a saddle connection holonomy in $I$ is closed. By definition,
\be
L(I) = L \cap \Om\cM_g(\kap;I),
\ee
so $L(I)$ is closed in the subspace topology on $L$. Then since $L$ is immersed, $L(I)$ is closed in the intrinsic topology on $L$.

Fix $(Y,\eta) \in L(I)$, and let $\gam$ be a saddle connection on $(Y,\eta)$ with holonomy in $I$. Since $(Y,\eta) \in L$, we have $\Per(\eta) = \Per(\om)$, and by assumption,
\be
\R z \cap \left(\bigcup_{z_0 \in \Per(\eta)} \R z_0\right) = \{0\} .
\ee
Then since $I \subset \R z$, the holonomy of $\gam$ is not an absolute period of $\eta$, so $\gam$ has distinct endpoints. Let $Z$ be the starting point of $\gam$, and let $Z^\pr$ be the ending point of $\gam$. If $\gam^\pr$ is another saddle connection on $(Y,\eta)$ with holonomy in $\R z$, then since $|\kap| = 2$, possibly after reversing orientation, $\gam^\pr$ starts at $Z$ and ends at $Z^\pr$. Concatenating $\gam$ with the reverse of $\gam^\pr$ gives a closed loop whose associated absolute period is $0$ since it lies in $\R z$, so $\gam$ and $\gam^\pr$ have the same holonomy. Thus, there are only finitely many saddle connections $\gam_1,\dots,\gam_m$ on $(Y,\eta)$ with holonomy in $\R_{>0} z$, and they all start at $Z$ and end at $Z^\pr$. By Corollary \ref{cor:geodesic}, there is a geodesic ray in $L$ in the metric $|q|$ through $(Y,\eta)$ given by
\be
s : \R_{>0} \ra L, \quad s(t) = (Y_t,\eta_t),
\ee
such that for all $t > 0$ and $1 \leq k \leq m$,
\be
\int_{\gam_k} \eta_t = t z .
\ee
In particular, $s$ is injective and $s^{-1}(L(I)) = (0,1]$. The period coordinates of $s(1)$ lie in the $\Q$-span of $\Per(\om)$ and $z$, so there are only countably many possibilities for $s(1)$. Thus, $L(I)$ is a countable union of embedded parallel line segments in the metric $|q|$.

Now, $\ell = s((0,1])$ is one of the countably many maximal line segments in $L(I)$. Fix $0 < \eps < 1$, and let $\ell_\eps = s([\eps,1])$. Since $\ell_\eps$ is a compact line segment in $L$, there is $\eps_1 > 0$ such that the $\eps_1$-neighborhood $U$ of $\ell_\eps$ in the metric $|q|$ is an embedded disk in $L$ and has compact closure. By Lemma \ref{lem:Ufin}, there are only finitely many homotopy classes of paths $\del_1,\dots,\del_n$ on $s(1)$ from $Z$ to $Z^\pr$ for which the geodesic representative on some holomorphic $1$-form in $U$ has length less than $2|z|$. Fix $0 < \eps_2 < |z|$. Shrinking $\eps_1$ if necessary, we may assume that along any straight path in $U$ of length at most $\eps_1$, the length of the geodesic representative of each $\del_k$ changes by at most $\eps_2$. Suppose that some $(Y^\pr,\eta^\pr) \in U \sm \ell$ has a saddle connection $\gam^\pr$ with holonomy in $I$. There is a straight path $\varphi_0$ in $L$ of length less than $\eps_1$ from $(Y^\pr,\eta^\pr)$ to $s(t_0) = (Y_{t_0},\eta_{t_0})$ for some $t_0 \in [\eps,1]$. Parallel transport of the homotopy class of $\gam^\pr$ along $\varphi_0$ gives a homotopy class of paths on $(Y_{t_0},\eta_{t_0})$ such that the length of the geodesic representative on $(Y_{t_0},\eta_{t_0})$ is less than $|z|+\eps_2$. Since $\eps_2 < |z|$, the homotopy class of $\gam^\pr$ is $\del_j$ for some $1 \leq j \leq n$. Since $\varphi_0$ is not parallel to $\ell$, we have $\int_{\del_j} \eta_{t_0} \notin \R z$. Since $\varphi_0$ has length less than $\eps_1$ in the metric $|q|$, the change in $\int_{\del_j} \eta^\pr$ along $\varphi_0$ has absolute value less than $\eps_1$. Thus, the Euclidean distance in $\C$ from $\int_{\del_j} \eta_{t_0}$ to the line $\R z$ is less than $\eps_1$. For $1 \leq k \leq n$, let $z_k = \int_{\del_k} \eta_1$. Then
\be
\int_{\del_k} \eta_t = z_k + (t - 1) z \in z_k + \R z
\ee
so the distance $D_k$ from $\int_{\del_k} \eta_t$ to $\R z$ is constant. Let $S$ be the subset of $k \in \{1,\dots,n\}$ such that $D_k > 0$, and let $D = \min_{k \in S} D_k > 0$. By shrinking $\eps_1$ further so that $\eps_1 < D$, we get a contradiction. Thus, the $\eps_1$-neighborhood of $\ell_\eps$ is disjoint from $L(I) \sm \ell$. Letting $\eps \ra 0$ and taking a union of these neighborhoods of $\ell_\eps$, we get a neighborhood of $\ell$ that is disjoint from $L(I) \sm \ell$. Thus, the maximal line segments in $L(I)$ are isolated.

Choose a path $\varphi : [0,1] \ra L$ such that $\varphi(0) \notin L(I)$ and $\varphi(1) \notin L(I)$. Applying a homotopy to $\varphi$, keeping the endpoints fixed, we may assume that $\varphi$ is piecewise-geodesic with finitely many pieces in the metric $|q|$, that the endpoints of each piece do not lie in $L(I)$, and that each piece is not parallel to the line segments in $L(I)$. By compactness, there is $0 < \eps < |z|$ such that for all $t \in [0,1]$, the length of the shortest saddle connection on $\varphi(t)$ is at least $\eps$. Since $L(I) \subset L$ is closed, $\varphi([0,1]) \cap L(I)$ is compact, and since the maximal line segments in $L(I)$ are isolated, $\varphi([0,1]) \cap L(I)$ is discrete. Therefore, $\varphi([0,1]) \cap L(I)$ is finite, so let $0 < t_1 < \cdots < t_n < 1$ be the finitely many times such that $\varphi(t_j) \in L(I)$. For $1 \leq j \leq n$, let $s_j : \R_{>0} \ra L$ be the geodesic ray through $\varphi(t_j)$ such that $\ell_j = s_j((0,1])$ is a maximal line segment in $L(I)$, and let $\ell_{j,\eps} = s_j([\eps,1])$. Fix $\eps_1 > 0$ such that for $1 \leq j \leq n$, the $\eps_1$-neighborhood of $\ell_{j,\eps}$ is an embedded disk disjoint from $L(I) \sm \ell$. Fix $\eps_2 > 0$ such that for $1 \leq j \leq n$ and $t \in (t_j - \eps_2, t_j + \eps_2)$, the distance from $\varphi(t)$ to $\ell_j$ is less than $\eps_1$ in the metric $|q|$. Then for $1 \leq j \leq n$, we can apply a homotopy to the restriction $\varphi\hmid_{[t_j-\eps_2,t_j+\eps_2]}$, keeping the endpoints $\varphi(t_j-\eps_2)$ and $\varphi(t_j+\eps_2)$ fixed, to arrange that the image of $\varphi\hmid_{[t_j-\eps_2,t_j+\eps_2]}$ is contained in the $\eps_1$-neighborhood of $\ell_{j,\eps}$ and disjoint from $\ell_j$. In other words, instead of crossing $\ell_j$ at time $t_j$, we go around $\ell_j$ while staying close to $\ell_j$. This gives us a path $[0,1] \ra L \sm L(I)$ with the same starting point $\varphi(0)$ and the same ending point $\varphi(1)$. Thus, $L \sm L(I)$ is path-connected.
\end{proof}

\begin{rmk} \label{rmk:parallel}
The hypothesis in Lemma \ref{lem:LIconn} that $z$ is not parallel to an absolute period of $\om$ can be weakened slightly to include the case where the group $\Per(\om) \cap \R z$ is cyclic with a generator $w$ such that $|w| > |z|$.
\end{rmk}

\begin{rmk} \label{rmk:sqtiled}
In the special case where $\Per(\om)$ is the lattice $\Z + i\Z$, the leaf $L = L(\om)$ is tiled by finitely many unit squares for the metric $|q|$. See \cite{Dur:square} for illustrations of these leaves in the stratum $\Om\cM_2(1,1)$. For $z \in \C^\ast$ not parallel to any element of $\Z + i\Z$, the subset $L(I)$ is a union of finitely many embedded parallel line segments with irrational slope, and in this case it is easy to see that $L \sm L(I)$ is path-connected.
\end{rmk}

\paragraph{\bf Connected sums along leaves.} Lemma \ref{lem:LIconn} also holds with $\wt{\Om}\cM_g(\kap;m)$ and $\cA(\kap;m)$ in place of $\Om\cM_g(\kap)$ and $\cA(\kap)$, and the proof is the same. For $\wt{L}$ a leaf of $\cA(\kap;m)$, define $\wt{L}(I) = \{(Y,\wt{\eta}) \in \wt{L} : \Gam(\eta) \cap I \neq \emptyset\}$ as in Lemma \ref{lem:LIconn}. Consider the full measure subset of $\cT(\kap;m)$ given by
\be
\cT_{\rm conn}(\kap;m) = \left\{(X,\wt{\om},(\gam,w)) \in \cT(\kap;m) : \int_\gam \om \notin \bigcup_{z \in \Per(\om)} \R z\right\} .
\ee
Note that $\cT_{\rm conn}(\kap;m)$ is a union of leaves of $\cF_\cT$.

\begin{cor} \label{cor:sumleaf}
Suppose $|\kap| = 2$, fix $(X,\wt{\om},T) \in \cT_{\rm conn}(\kap;m)$, and write $T = (\gam,w)$. Let $\wt{L}$ be the leaf of $\cA(\kap;m)$ through $(X,\wt{\om})$, let $L_\cT$ be the leaf of $\cF_\cT$ through $(X,\wt{\om},T)$, and let $I = \left\{t \int_\gam \om : 0 \leq t \leq 1\right\}$. Then we have $L_\cT = (\wt{L} \sm \wt{L}(I)) \times \{T\}$.
\end{cor}

Lastly, let $\kap^\pr = (\kap \sm (m)) \cup (m+2)$, and consider the associated connected sum map $\Psi : \cT(\kap;m) \ra \Om\cM_{g+1}(\kap^\pr)$. Then $\Psi$ sends leaves of $\cF_\cT$ into leaves of $\cA(\kap^\pr)$.


\section{Pairs of splittings} \label{sec:pair}

This section focuses on holomorphic $1$-forms in a stratum component that can presented as a connected sum with a torus. Our goal is to show that any two such holomorphic $1$-forms with the same area are related by a finite sequence of moves as follows. In each move, we choose a presentation as a connected sum, keep the torus fixed, and change the complementary holomorphic $1$-form while preserving its area. Our main tool will be a criterion for presenting a holomorphic $1$-form as a connected sum in two different ways. In a later section, we will prove our main theorems using stronger versions of this goal obtained by combining the results of Sections \ref{sec:leaves} and \ref{sec:pair}. \\

\paragraph{\bf Splittings.} Recall from Section \ref{sec:surgery} that a splitting of $(X,\om)$ is a pair of homologous saddle connections $\al^\pm$ on $(X,\om)$ that form a figure-eight at a zero $Z$ of $\om$, such that
\begin{enumerate}
    \item the counterclockwise angle around $Z$ from the end of $\al^-$ to the end of $\al^+$ is $2\pi$;
    \item one of the connected components of $X \sm (\al^+ \cup \al^-)$ is a cylinder $C$.
\end{enumerate}
We will refer to $C$ as the {\em associated cylinder} of $\al^\pm$. The homology class in $H_1(X;\Z)$ represented by $\al^\pm$ is denoted $[\al^\pm]$. Slitting $(X,\om)$ along $\al^{\pm}$ and regluing the sides of the slits gives a holomorphic $1$-form of genus $g-1$ and a flat torus. For $z,w \in \C$, the signed area of the parallelogram spanned by $z$ and $w$ is given by $\im(\ol{z}w)$.

\begin{lem} \label{lem:newsum}
Let $\al^\pm$ be a splitting of $(X,\om)$ with associated cylinder $C$, and choose a saddle connection $\bet \subset C \cup Z(\om)$. Suppose that $(X,\om)$ has an embedded open parallelogram $P$ with one pair of parallel sides given by $\al^\pm$ and the other pair given by a pair of homologous saddle connections $\gam_0^\pm$. Let
\be
z = \int_{\al^\pm} \om, \quad w = \int_\bet \om, \quad z_1 = \int_{\gam_0^\pm} \om,
\ee
and suppose that
\begin{equation} \label{eq:direction}
\im(\ol{z} w) > 0, \quad \im(\ol{z} z_1) > 0, \quad \im(\ol{z}_1 w) > 0 , \quad \im\left(\ol{(z + w)} (z_1 + w)\right) > 0 .
\end{equation}
Then $P \cup \ol{C}$ contains another splitting $\gam^\pm$ of $(X,\om)$ with associated cylinder $C^\pr$ and with the same starting point and ending point as $\al^\pm$, and there is a saddle connection $\del \subset C^\pr \cup Z(\om)$ such that
\be
[\gam^\pm] = -[\gam_0^\pm] - [\bet], \quad [\del] = [\al^\pm] + [\bet]
\ee
in $H_1(X;\Z)$.
\end{lem}

\begin{proof}
For $M \in \GL^+(2,\R)$, $\im(\ol{z}w)$ and $\im(\ol{Mz} \; Mw)$ have the same sign. There is an affine homeomorphism $(X,\om) \ra M(X,\om)$ that sends zeros to zeros, and sends a saddle connection on $(X,\om)$ with holonomy $z_0$ to a saddle connection on $M(X,\om)$ with holonomy $Mz_0$. A pair of homologous saddle connections is a splitting of $(X,\om)$ if and only if the corresponding pair on $M(X,\om)$ is a splitting of $M(X,\om)$. Thus, it is enough to show that Lemma \ref{lem:newsum} holds for $M(X,\om)$. Since $\im(\ol{z}w) > 0$, by applying an appropriate element of $\GL^+(2,\R)$ to $(X,\om)$, we may assume that $z = 1$ and $w = i$.

We regard $C$ as a unit square with its vertical sides glued together to form $\bet$. The bottom side of $C$ is $\al^-$, and the top side of $C$ is $\al^+$. Since $z = 1$ and $w = i$, the inequalities in (\ref{eq:direction}) reduce to
\be
\im(z_1) > 0, \quad \re(z_1) > 0, \quad \im(z_1) - \re(z_1) + 1 > 0 .
\ee
These inequalities imply $0 < \re(z_1 + i) < \im(z_1 + i)$, so there is a geodesic segment in the direction of $z_1 + i$ from the bottom-left corner of $C$ to a point in the interior of $\al^+$ that is disjoint from the interior of the vertical saddle connection $\bet$. These inequalities also imply $\im(\ol{z}_1 (z_1 + i)) > 0$ and $\im\left(\ol{(z_1 + i)}(z_1 - w)\right) > 0$, so there is a geodesic segment in the direction of $-z_1-i$ from the top-left corner of $P$ to a point in the interior of $\al^+$ that is disjoint from the interiors of the saddle connections $\gam_0^\pm$. Both of these geodesic segments end at the same point on $\al^+$, and this point is a distance $\re(z_1)/(\im(z_1) + 1)$ from the left end of $\al^+$. Thus, there is a saddle connection $\gam^- \subset P \cup \ol{C}$ from the top-left corner of $P$ to the bottom-left corner of $C$ that crosses $\al^+$. Similarly, there is a saddle connection $\gam^+ \subset P \cup \ol{C}$ from the top-right corner of $C$ to the bottom-right corner of $P$ that crosses $\al^-$.

The saddle connections $\gam^\pm$ form the boundary of a cylinder $C^\pr \subset P \cup \ol{C}$, so they are homologous. In particular, the cylinder $C^\pr$ is one of the connected components of $X \sm (\gam^+ \cup \gam^-)$. The saddle connections $\gam^\pm$ have the same starting point and ending point $Z$ as $\al^\pm$. Since the counterclockwise angle around $Z$ from the end of $\al^-$ to the end of $\al^+$ is $2\pi$, the counterclockwise angle around $Z$ from the end of $\gam^-$ to the end of $\gam^+$ is also $2\pi$. Thus, $\gam^\pm$ is another splitting of $(X,\om)$. Let $\del$ be the geodesic segment from the bottom-left corner of $C$ to the top-right corner of $C$. Then $\del$ is a saddle connection contained in $C^\pr \cup Z(\om)$. All of the saddle connections $\al^\pm,\gam_0^\pm,\gam^\pm,\bet,\del$ have the same starting and ending point, and therefore represent elements of $H_1(X;\Z)$. The relations $[\gam^\pm] = -[\gam_0^\pm] - [\bet]$ and $[\del] = [\al^\pm] + [\bet]$ are clear.
\end{proof}

See Figure \ref{fig:sum2} for an illustration of Lemma \ref{lem:newsum}. \\

\begin{figure}
    \centering
    \includegraphics[width=0.9\textwidth]{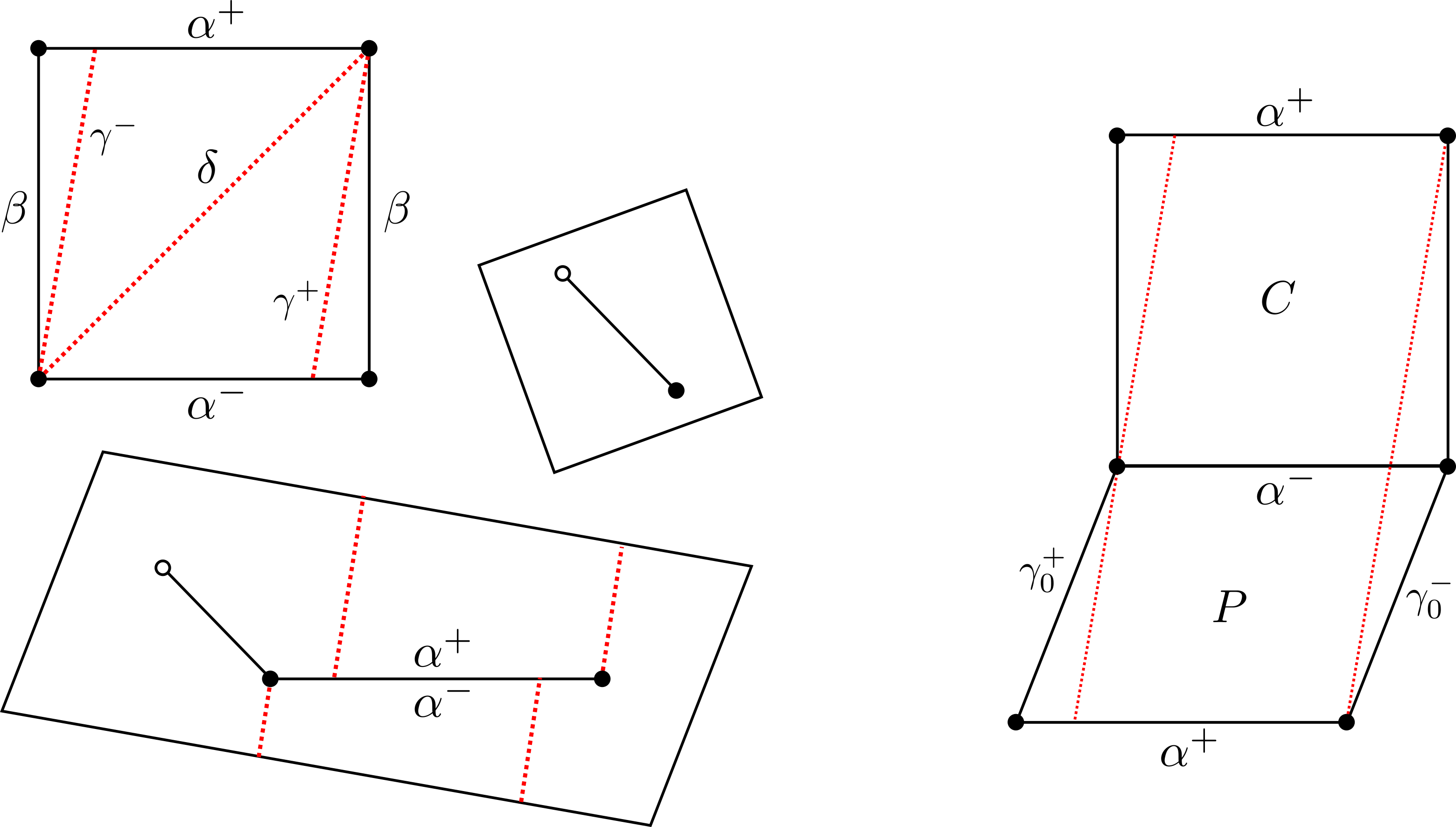}
    \caption{Left: a holomorphic $1$-form in $\Om\cM_3(3,1)$ with a splitting $\al^\pm$ and a second splitting $\gam^\pm$ as in Lemma \ref{lem:newsum}. Right: the parallelogram $P$ and the cylinder $C$, with the parallelogram sides $\al^\pm$ and $\gam_0^\pm$ labelled.}
    \label{fig:sum2}
\end{figure}

\paragraph{\bf Related splittings.} To any splitting $\al^\pm$ of a holomorphic $1$-form $(X,\om)$ with associated cylinder $C$, we can also associate a pair of absolute periods of $\om$ as follows. Choose a saddle connection $\bet \subset C \cup Z(\om)$. Let $z = \int_{\al^\pm} \om$ and $w = \int_\bet \om$. Reversing the orientation of $\bet$ if necessary, we may require that $\im(\ol{z} w) > 0$. We say that the pair $(z,w)$ are the {\em associated periods} of the splitting $\al^\pm$. Note that $\im(\ol{z} w)$ is the area of $C$ in the metric $|\om|$. Also, different choices of $\bet$ will yield absolute periods $nz+w$ in place of $w$ for all $n \in \Z$. Thus, while the period $z$ is uniquely determined by the splitting $\al^\pm$, the period $w$ is only determined up to addition by an integer multiple of $z$.

For $\cC$ a stratum component, define
\be
\cC_1(z,w) = \left\{(X,\om) \in \cC_1 : (X,\om) \text{ has a splitting with associated periods } (z,w) \right\} .
\ee
The following lemma provides a criterion for two subsets of the form $\cC_1(z,w)$ and $\cC_1(z^\pr,w^\pr)$ to intersect. In other words, we will construct holomorphic $1$-forms in $\cC_1$ with two splittings whose associated pairs of periods are $(z,w)$ and $(z^\pr,w^\pr)$, respectively.

\begin{lem} \label{lem:sum2}
Let $\cC$ be the nonhyperelliptic component of a stratum $\Om\cM_g(m_1,m_2)$ with $m_1 \geq 3$ odd. Fix $z,w,z_1 \in \C$ such that $0 < \im(\ol{z} w) < 1$ and such that
\begin{equation} \label{eq:gen}
0 < \im(\ol{z} z_1) < 1 - \im(\ol{z} w), \quad 0 < \im(\ol{z}_1 w) < \im(\ol{z} (z_1 + w)) .
\end{equation}
There exists $(X,\om) \in \cC_1$ with a pair of splittings $\al^\pm$ and $\gam^\pm$ with associated cylinders $C$ and $C^\pr$, such that $\al^\pm$ and $\gam^\pm$ all start and end at the same zero of order $m_1$, and there are saddle connections $\bet \subset C \cup Z(\om)$ and $\del \subset C^\pr \cup Z(\om)$, such that
\be
z = \int_{\al^\pm} \om, \quad w = \int_\bet \om, \quad -z_1-w = \int_{\gam^\pm} \om, \quad z+w = \int_{\del}  \om.
\ee
\end{lem}

\begin{proof}
Let $T_0$ be the flat torus $(\C/(\Z z + \Z w), dz)$. By assumption $\im(\ol{z}z_1) > 0$, and in particular $z \notin \R z_1$. If $w_1 \in \C$ is such that $z = t (mz_1 + nw_1)$ with $t \in \R$ and $m,n \in \Z$, then $t \neq 0$ and $n \neq 0$, therefore $w_1$ lies in $\R z + \Q z_1$, a countable union of parallel lines. The complement $\C \sm (\R z + \Q z_1)$ is dense in $\C$. Thus, we can choose $w_1 \in \C$ such that $z \notin \bigcup_{z_0 \in \Z z_1 + \Z w_1} \R z_0$ and such that
\begin{equation} \label{eq:area}
0 < \im(\ol{z} z_1) < \im(\ol{z}_1 w_1) < 1 - \im(\ol{z} w) .
\end{equation}
Let $T_1$ be the flat torus $(\C / (\Z z_1 + \Z w_1), dz)$. Let $T_2$ be a flat torus with area less than $1 - \im(\ol{z} w) - \im(\ol{z}_1 w_1)$. The segment $[0,z] \subset \C$ projects to a closed geodesic $\al_0 \subset T_0$, and projects to an embedded geodesic segment $\al \subset T_1$.

The segments $[0,z_1],[z,z+z_1] \subset \C$ project to a pair of closed geodesics $\gam_0^\pm \subset T_1$ passing through the endpoints of $\al$ and otherwise disjoint from $\al$. The inequalities $0 < \im(\ol{z} z_1) < \im(\ol{z}_1 w_1)$ in (\ref{eq:area}) imply that $\gam_0^\pm$ and the two sides of $\al$ bound an embedded open parallelogram $P \subset T_1$. For $j = 1,2$, choose short embedded geodesic segments $s_j \subset T_j$ with the same length and in the same direction, such that $s_1$ starts at the starting point of $\al$ and is otherwise disjoint from $\ol{P}$. Slit $T_j$ along $s_j$ and reglue opposite sides to get a holomorphic $1$-form $(X_0,\om_0) \in \Om\cM_2(1,1)$ given by a connected sum of two flat tori along a pair of homologous saddle connections $s^\pm$. On $(X_0,\om_0)$, the starting point of $\al$ is a zero of $\om_0$.

Let $\al_1,\dots,\al_{(m_1-3)/2}$ be a collection of short embedded geodesic segments in $T_2$, starting at the zero that is the starting point of $\al$ and otherwise disjoint from each other and from $s^+ \cup s^-$. Let $\al_1^\pr,\dots,\al_{(m_2-1)/2}^\pr$ be a collection of short embedded geodesic segments in $T_2$, starting at the other zero of $\om_0$ and otherwise disjoint from each other and from $s^+ \cup s^-$. Additionally, we require that $\al_j$ and $\al_k^\pr$ are disjoint for all $j,k$. Slit $(X_0,\om_0)$ along $\al$, slit $T_0$ along $\al_0$, and reglue opposite sides to get a holomorphic $1$-form in $\Om\cM_3(3,1)$ with a splitting $\al^\pm$ with associated cylinder $C$. Then, iterate this procedure using the segments $\al_1,\dots,\al_{(m_1-3)/2}$ and $\al_1^\pr,\dots,\al_{(m_2-1)/2}^\pr$ and using flat tori with appropriate areas to get a holomorphic $1$-form $(X,\om) \in \Om_1\cM_g(m_1,m_2)$. As discussed before Lemma \ref{lem:sum}, since $(X,\om)$ has a splitting it cannot lie in a hyperelliptic component, therefore $(X,\om) \in \cC_1$. On $(X,\om)$, we have $\int_{\al^\pm} \om = z$. Let $\bet \subset C \cup Z(\om)$ be a saddle connection such that $\int_\bet \om = w$. By construction, the saddle connections $\al^\pm$ have holonomy $z$, and the saddle connections $\gam_0^\pm$ have holonomy $z_1$. As part of the inequalities in (\ref{eq:gen}), we have $\im(\ol{z} z_1) > 0$ and $\im(\ol{z}_1 w) > 0$. Additionally, the inequality $\im(\ol{z}_1 w) < \im(\ol{z}(z_1 + w))$ is equivalent to $\im\left(\ol{(z+w)}(z_1 + w)\right) > 0$. Thus, the saddle connections $\al^\pm,\gam_0^\pm$ and the parallelogram $P$ on $(X,\om)$ satisfy the hypotheses of Lemma \ref{lem:newsum}. Letting $\gam^\pm$ be a splitting of $(X,\om)$ with associated cylinder $C^\pr \subset P \cup \ol{C}$ and $\del \subset C^\pr \cup Z(\om)$ a saddle connection as in Lemma \ref{lem:newsum}, we are done.
\end{proof}

Now define
\be
\cT_{(0,1)} = \left\{(z,w) \in \C^2 : 0 < \im(\ol{z} w) < 1\right\} .
\ee
By definition, for all $(z,w) \in \cT_{(0,1)}$ and all $n \in \Z$, we have
\be
\cC_1(z,w) = \cC_1(z,nz + w) .
\ee
Additionally, by Lemma \ref{lem:sum2}, for all $(z,w) \in \cT_{(0,1)}$ and all $z_1 \in \C$ satisfying the inequalities in (\ref{eq:gen}), the intersection
\be
\cC_1(z,w) \cap \cC_1(-z_1 - w, z + w)
\ee
is nonempty. With the above two properties of the sets $\cC_1(z,w)$ in mind, let $\sim$ be an equivalence relation on $\cT_{(0,1)}$ satisfying the following. We suppose that
\begin{equation} \label{eq:gen1}
(z,w) \sim (z,nz + w)
\end{equation}
for all $(z,w) \in \cT_{(0,1)}$ and all $n \in \Z$. We also suppose that
\begin{equation} \label{eq:gen2}
(z,w) \sim (-z_1-w,z+w)
\end{equation}
for all $(z,w) \in \cT_{(0,1)}$ and all $z_1 \in \C$ satisfying the inequalities in (\ref{eq:gen}).

Let $\cC$ be a stratum component as in Lemma \ref{lem:sum2}. Fix $(X,\om) \in \cC_1$ with a splitting $\al^\pm$ with associated cylinder $C$ and associated periods $(z,w) \in \cT_{(0,1)}$. By definition, $(X,\om) \in \cC_1(z,w)$. If we change the complement of $\ol{C}$ in $(X,\om)$, without changing $\ol{C}$ or the area of its complement, then we get another holomorphic $1$-form $(X_1,\om_1) \in \cC_1(z,w)$. Next, choose a new splitting $\al_1^\pm$ of $(X_1,\om_1)$, and repeat. By definition of $\sim$, for any $(Y,\eta) \in \cC_1$ with a splitting with associated pair $(z^\pr,w^\pr) \in \cT_{(0,1)}$ such that $(z,w) \sim (z^\pr,w^\pr)$, we can iterate this procedure finitely many times to go from $(X,\om)$ to $(Y,\eta)$. Our goal is to show that $(Y,\eta)$ can be any holomorphic $1$-form in $\cC_1$ with a splitting, and Lemma \ref{lem:sum2} reduces this to showing that any two elements of $\cT_{(0,1)}$ are equivalent with respect to $\sim$.

\begin{lem} \label{lem:reln}
For all $(z,w) \in \cT_{(0,1)}$ and $(z^\pr,w^\pr) \in \cT_{(0,1)}$, we have $(z,w) \sim (z^\pr,w^\pr)$.
\end{lem}

\begin{proof}
Observe that since $\cT_{(0,1)}$ is connected, it is enough to show that every equivalence class for $\sim$ is open. Fix $(z,w) \in \cT_{(0,1)}$. Since $\im(\ol{z} w) > 0$, there is $z_1 \in \C$ such that $\im(\ol{z} z_1) > 0$ and $\im(\ol{z}_1 w) > 0$. By multiplying $z_1$ by small positive real numbers, we see that $z_1$ can be arbitrarily small. For such $z_1$ sufficiently small, we have
\be
0 < \im(\ol{z} z_1) < 1 - \im(\ol{z} w), \quad 0 < \im(\ol{z}_1 w) < \im(\ol{z} (z_1+w)) .
\ee
Let $u_1 = -z_1-w$ and $v_1 = z+w$. Since $(z,w)$ and $z_1$ satisfy (\ref{eq:gen}), we have $(z,w) \sim (u_1,v_1)$. Similarly, since $\im(\ol{u_1} v_1) > 0$, there is $z_2 \in \C$ such that
\be
0 < \im(\ol{u}_1 z_2) < 1 - \im(\ol{u}_1 v_1), \quad 0 < \im(\ol{z}_2 v_1) < \im(\ol{u}_1 (z_2+v_1)) .
\ee
Let $u_2 = -z_2-v_1$ and $v_2 = u_1+v_1$. Since $(u_1,v_1)$ and $z_2$ satisfy (\ref{eq:gen}), we have $(u_1,v_1) \sim (u_2,v_2)$. The inequalities in (\ref{eq:gen}) are open conditions, so for $s,t \in \C$ sufficiently small, by replacing $(z,w)$ with $(z + s, w+ t)$, $z_1$ with $z_1 + s$, and $z_2$ with $z_2 - s - t$, we get
\begin{align*}
(z + s, w + t) &\sim (-(z_1 + s) - (w + t), (z + s) + (w + t)) = (u_1 - s - t, v_1 + s + t) \\
&\sim (-(z_2-s-t)-(v_1+s+t), (u_1-s-t) + (v_1 + s + t)) = (u_2,v_2) .
\end{align*}
We have shown that $(z,w) \sim (u_2,v_2) \sim (z+s,w+t)$ for all $s,t \in \C$ sufficiently small. Thus, the equivalence class of $(z,w)$ contains an open neighborhood of $(z,w)$.
\end{proof}

To conclude this section, we adapt our connected sum constructions to subspaces of strata defined by constraining the absolute periods to lie in certain closed subgroups. Let $\Lam$ be a proper closed subgroup of $\C$ that is not discrete and that contains a lattice in $\C$. Recall that $\Lam = M \cdot (\R + i\Z)$ for some $M \in \SL(2,\R)$. Let $\Lam_0 = M \cdot \R$ be the identity component of $\Lam$. For $z \in \Lam$, define
\begin{equation} \label{eq:index}
\cI_\Lam(z) = |\im(M^{-1} z)| \in \Z_{\geq 0} .
\end{equation}
If $z \notin \Lam_0$, then $\cI_\Lam(z)$ is the index $[\Lam : \Lam_0 + \Z z]$. If $z \in \Lam_0$, then $\cI_\Lam(z) = 0$. Recall that if $\cC$ is a stratum component, then $\cC^\Lam$ denotes the set of $(X,\om) \in \cC$ such that $\Per(\om) + \Lam_0 = \Lam$. Associated to a splitting of $(X,\om) \in \cC^\Lam_1$ is a pair of periods $(z,w) \in \cT_{(0,1)}$ and a holomorphic $1$-form $(X_1,\om_1)$ of genus $g-1$ satisfying $\Per(\om_1) + \Lam_0 = n\Lam$ for some $n \in \Z_{>0}$. Since
\be
\Lam = \Per(\om) + \Lam_0 = \Z z + \Z w + \Per(\om_1) + \Lam_0 = \Z z + \Z w + n\Lam ,
\ee
we have $\gcd(\cI_\Lam(z),\cI_\Lam(w),n) = 1$. With Lemma \ref{lem:LIconn} and Corollary \ref{cor:sumleaf} in mind, we will focus on splittings that are not parallel to $\Lam_0$. Define
\be
\cT^\Lam_{(0,1)} = \{(z,w,n) \in (\Lam \sm \Lam_0) \times \Lam \times \Z_{>0} : 0 < \im(\ol{z} w) < 1, \; \gcd(\cI_\Lam(z),\cI_\Lam(w),n) = 1\} .
\ee
Then any splitting of $(X,\om) \in \cC^\Lam_1$ that is not parallel to $\Lam_0$ determines an element of $\cT_{(0,1)}^\Lam$ as above. Let $\sim_\Lam$ be an equivalence relation on $\cT^\Lam_{(0,1)}$ satisfying the following. We suppose that
\begin{equation} \label{eq:gen1R+iZ}
(z,w,n) \sim_\Lam (z,kz+w,n)
\end{equation}
for all $(z,w,n) \in \cT_{(0,1)}^\Lam$ and all $k \in \Z$. We suppose that
\begin{equation} \label{eq:gen2R+iZ}
(z,w,n) \sim_\Lam (-z_1-w,z+w,\gcd(\cI_\Lam(z+w),n))
\end{equation}
for all $(z,w,n) \in \cT_{(0,1)}^\Lam$ and all $z_1 \in n\Lam \sm \Lam_0$ such that $\cI_\Lam(-z_1 - w) \neq 0$, and such that $z,w,z_1$ satisfy the inequalities in (\ref{eq:gen}). Lastly, we suppose that
\begin{equation} \label{eq:gen3R+iZ}
(z,w,n) \sim_\Lam (-z_1-w,z+w,\gcd(\cI_\Lam(z+w),n))
\end{equation}
for all $(z,w,n) \in \cT_{(0,1)}^\Lam$ with $w \in \Lam \sm \Lam_0$ and all $z_1 \in n\Lam_0 = \Lam_0$ such that $z,w,z_1$ satisfy the inequalities in (\ref{eq:gen}), and such that there is $w_1 \in n \Lam$ satisfying $\im(\ol{z}_1 w_1) > 0$ and
\begin{equation} \label{eq:gen4R+iZ}
0 < \im(\ol{z}z_1) < \im(\ol{z}_1 w_1) < 1 - \im(\ol{z}w) .
\end{equation}
The following lemma is an analogue of Lemma \ref{lem:sum2} for $\cC^\Lam$.

\begin{lem} \label{lem:sum2R+iZ}
Let $\cC$ be the nonhyperelliptic component of a stratum $\Om\cM_g(m_1,m_2)$ with $m_1 \geq 3$ odd. Let $\Lam = M \cdot (\R + i\Z)$ with $M \in \SL(2,\R)$, and let $\Lam_0 = M \cdot \R$. Fix $(z,w,n) \in \cT_{(0,1)}^\Lam$ and $z_1 \in \C$ satisfying (\ref{eq:gen}), and also satisfying (\ref{eq:gen4R+iZ}) if $z_1 \in \Lam_0$. There exists $(X,\om) \in \cC_1^\Lam$ with a pair of splittings $\al^\pm$ and $\gam^\pm$ with associated cylinders $C$ and $C^\pr$, such that $\al^\pm$ and $\gam^\pm$ all start and end at the same zero of order $m_1$, and there are saddle connections $\beta \subset C \cup Z(\om)$ and $\del \subset C^\pr \cup Z(\om)$, such that
\be
z = \int_{\al^\pm} \om, \quad w = \int_\bet \om, \quad -z_1 - w = \int_{\gam^\pm} \om, \quad z + w = \int_\del \om .
\ee
Moreover, the holomorphic $1$-forms $(X_1,\om_1),(X_2,\om_2) \in \Om\cM_{g-1}(m_1-2,m_2)$ obtained from $(X,\om)$ by slitting and regluing $\al^\pm,\gam^\pm$, respectively, satisfy $\Per(\om_1) + \Lam_0 = n\Lam$ and $\Per(\om_2) + \Lam_0 = n_1 \Lam$ where $n_1 = \gcd(\cI_\Lam(z+w),n)$.
\end{lem}

\begin{proof}
The proof is similar to that of Lemma \ref{lem:sum2}. We may assume $\Lam = \R + i\Z$, so $\Lam_0 = \R$. Let $T_0 = (\C / (\Z z + \Z w), dz)$. Since $z \notin \R$, $\R z + \Q z_1$ is a countable union of parallel lines that are not parallel to $\R$, so the intersection $(\R z + \Q z_1) \cap (\R + i\Z)$ is countable. First, suppose $z_1 \in (\R + in\Z) \sm \R$. Then for any connected component $C_a = \R + i a$ of $\R + i\Z$, $a \in \Z$, we have $\{\im(\ol{z}_1 w_0) : w_0 \in C_a\} = \R$. Thus, we can choose $w_1 \in (\R + in\Z) \sm (\R z + \Q z_1)$ satisfying the inequalities in (\ref{eq:area}), and we have $z \notin \bigcup_{z_0 \in \Z z_1 + \Z w_1} \R z_0$. Next, suppose $z_1 \in \R$. In this case, for each $a \in \Z$, the set $\{\im(\ol{z}_1 w_0) : w_0 \in C_a\}$ is a singleton, and (\ref{eq:gen4R+iZ}) ensures that we can choose $w_1 \in (\R + in\Z) \sm (\R z + \Q z_1)$ satisfying the inequalities in (\ref{eq:area}), and again we have $z \notin \bigcup_{z_0 \in \Z z_1 + \Z w_1} \R z_0$.

With $w_1$ chosen as in the previous paragraph, let $T_1 = (\C / (\Z z_1 + \Z w_1), dz)$. Let $T_2 = (\C / (\Z z_2 + \Z w_2), dz)$ be a flat torus with $z_2,w_2 \in \R + in\Z$ such that $T_2$ has area less than $1 - \im(\ol{z}w) - \im(\ol{z}_1 w_1)$, and such that $\gcd(\cI_\Lam(z_2),\cI_\Lam(w_2)) = n$. The rest of the construction is the same as in the proof of Lemma \ref{lem:sum2}, except as follows. Recall that $\al \subset T_1$ is an embedded geodesic segment, and that $\al$ is not horizontal since $z \notin \R$. For $\al_1,\dots,\al_{(m_1-3)/2}$, we choose short embedded horizontal segments in $T_2$, such that $\al_1$ starts at the zero that is the starting point of $\alpha$ and emanates rightward, and such that $\al_{j+1}$ starts at the ending point of $\al_j$ and emanates rightward for $j = 1,\dots,(m_1-5)/2$. For $\al_1^\pr,\dots,\al_{(m_1-1)/2}^\pr$, we choose short embedded horizontal segments in $T_2$, such that $\al_1^\pr$ starts at the other zero of $\om_0$ and emanates rightward, and such that $\al_{j+1}^\pr$ starts at the ending point of $\al_j^\pr$ and emanates rightward for $j = 1,\dots,(m_2-3)/2$. The associated flat tori $T_3,\dots,T_{g-1}$ all have the form $T_j = (\C/(\Z z_j + \Z w_j), dz)$ with $z_j \in \R$ and $w_j \in (\R + in\Z) \sm \R$. Let $(X,\om) \in \cC_1^\Lam$ be the resulting holomorphic $1$-form with splittings $\al^\pm$ and $\gam^\pm$ as in the proof of Lemma \ref{lem:sum2}. Let $(X_1,\om_1),(X_2,\om_2) \in \Om\cM_{g-1}(m_1-2,m_2)$ be the holomorphic $1$-forms obtained from $(X,\om)$ by slitting and regluing $\al^\pm,\gam^\pm$, respectively. Since $\gcd(\cI_\Lam(z_2),\cI_\Lam(w_2)) = n$ and $z_1,w_1,z_3,w_3,\dots,z_{g-1},w_{g-1} \in \R + in\Z$, we have $\Per(\om_1) + \R = \R + in\Z$. This means the tuple in $\cT_{(0,1)}^\Lam$ associated to $\al^\pm$ is $(z,w,n)$. The tuple associated to $\gam^\pm$ has the form $(-z_1-w,z+w,n_1)$ for some $n_1 \in \Z_{>0}$. We have
\be
\Per(\om_2) + \Lam_0 = \Z z_1 + \Z (z + w + w_1) + n \Lam = \Z (z + w) + n\Lam .
\ee
Thus, $n_1 = \gcd(\cI_\Lam(z+w),n)$ as desired.
\end{proof}

Lastly, we prove an analogue of Lemma \ref{lem:reln} in the setting of absolute periods constrained to lie in certain closed subgroups of $\C$.

\begin{lem} \label{lem:relnR+iZ}
Let $\Lam = M \cdot (\R + i\Z)$ with $M \in \SL(2,\R)$. For all $(z,w,n) \in \cT^\Lam_{(0,1)}$ and $(z^\pr,w^\pr,n^\pr) \in \cT^\Lam_{(0,1)}$, we have $(z,w,n) \sim_\Lam (z^\pr,w^\pr,n^\pr)$.
\end{lem}

\begin{proof}
We may assume $\Lam = \R + i\Z$. The connected components of $\cT^\Lam_{(0,1)}$ are given by
\begin{align*}
C_{a,b,c} &= \cT^\Lam_{(0,1)} \cap \left((\R + ia) \times (\R + ib) \times \{c\}\right) \\
&= \{(r_1 + ia, r_2 + ib, c) : r_1,r_2 \in \R, \; 0 < r_1 b - r_2 a < 1\}
\end{align*}
with $a \in \Z \sm \{0\}$, $b \in \Z$, $c \in \Z_{>0}$ satisfying $\gcd(a,b,c) = 1$. For any $(z,w,n) \in \cT_{(0,1)}^\Lam$, since $\im(z) \neq 0$, there is $k \in \Z$ such that $\im(z)$ and $\im(kz + w)$ have opposite signs. Following the proof of Lemma \ref{lem:reln}, there are arbitrarily small $z_1 \in \R$ such that $(z,kz+w,n)$ and $z_1$ satisfy the inequalities in (\ref{eq:gen}) and (\ref{eq:gen4R+iZ}). By definition of $\sim_\Lam$, we then have
\be
(z,w,n) \sim_\Lam (z, kz + w, n) \sim_\Lam (-z_1 - kz - w, (k + 1)z + w, n^\pr)
\ee
where $n^\pr = \gcd(\cI_\Lam((k+1)z+w),n)$. Similarly, letting $(z^\pr,w^\pr,n^\pr) = (-z_1 - kz - w, (k + 1)z + w, n^\pr)$, there is $\ell \in \Z$ such that $\im(z^\pr)$ and $\im(\ell z^\pr + w^\pr)$ have opposite signs, and then there are arbitrarily small $z_2 \in \R$ such that
\be
(z^\pr,w^\pr,n^\pr) \sim_\Lam (z^\pr, \ell z^\pr + w^\pr, n^\pr) \sim_\Lam (-z_2 - \ell z^\pr - w^\pr, (\ell + 1)z^\pr + w^\pr, n^{\pr\pr}) ,
\ee
where $n^{\pr\pr} = \gcd(\cI_\Lam((\ell+1)z^\pr + w^\pr),n^\pr)$. The inequalities in (\ref{eq:gen}) and (\ref{eq:gen4R+iZ}) are open conditions, and as in the proof of Lemma \ref{lem:reln}, we see that the equivalence class of $(z,w,n)$ with respect to $\sim_\Lam$ contains an open neighborhood of $(z,w,n)$. Then each equivalence class for $\sim_\Lam$ is open, and therefore a union of connected components of $\cT_{(0,1)}^\Lam$. Thus, it is enough to show that each equivalence class for $\sim_\Lam$ intersects every connected component of $\cT^\Lam_{(0,1)}$. We will carry this out in three steps as follows. Let $E(a,b,c)$ denote the equivalence class containing the connected component $C_{a,b,c}$. \\

\paragraph{\bf Step 1:} Fix $(z,w,n) \in \cT^\Lam_{(0,1)}$. We first show that the equivalence class of $(z,w,n)$ contains a connected component of the form $C_{a,b,1}$. Since this equivalence class contains the connected component of $(z,w,n)$ in $\cT^\Lam_{(0,1)}$, we may assume that $\re(z) = 0$. Since the action of $-I \in \SL(2,\R)$ on $\C$ respects $\sim_\Lam$, we may assume that $\im(z) > 0$. Additionally, by (\ref{eq:gen1R+iZ}), we may assume that $0 \leq \im(w) < \im(z)$.

Write $z = im_1$ and $w = t + in_1$ with $m_1,n_1 \in \Z$, $0 \leq n_1 < m_1$, and $0 < -t < 1/m_1$. Fix $z_1 = s + im_2 \in \R + in\Z$ and $k \in \Z$ with $m_2 \neq 0$ and $m_2 \neq -km_1 - n_1$. We have
\be
\im(\ol{z}z_1) = -sm_1, \quad 1 - \im(\ol{z}(kz+w)) = 1 + tm_1,
\ee
\be
\im(\ol{z}_1 (kz+w)) = s(km_1 + n_1) - tm_2, \quad \im(\ol{z}(z_1+kz+w)) = -(s+t)m_1 .
\ee
Then by definition of $\sim_\Lam$, if
\begin{equation} \label{eq:ineq1}
0 < -s m_1 < 1 + t m_1, \quad 0 < s (k m_1 + n_1) - t m_2 < -(s + t) m_1 ,
\end{equation}
then $(z,w,n) \sim_\Lam (-z_1-kz-w,(k+1)z+w,\gcd(\cI_\Lam((k+1)z+w),n))$. By dividing the first group of inequalities in (\ref{eq:ineq1}) by $m_1 > 0$ and the second group of inequalities by $-tm_1 > 0$, we see that (\ref{eq:ineq1}) is equivalent to
\begin{equation} \label{eq:ineq2}
0 < -s < \frac{1}{m_1} + t, \quad 0 < -\frac{s}{t}\left(k + \frac{n_1}{m_1}\right) + \frac{m_2}{m_1} < \frac{s}{t} + 1 .
\end{equation}
Note that $s/t + 1 > 1$ and $0 \leq n_1/m_1 < 1$.

Fix $m_2 \in n\Z_{< 0}$, and let $k \in \Z_{<0}$ be large. Here, $k$ may depend on $m_1,n_1,m_2,t$. As $-s$ ranges from $0$ to $1/m_1 + t$, the middle expression in the second group of inequalities in (\ref{eq:ineq2}) ranges from $m_2/m_1 < 0$ to $(1/tm_1 + 1)(k + n_1/m_1) + m_2/m_1 > 0$, where positivity is ensured since $1/tm_1 + 1 < 0$ and since $k + n_1/m_1 < 0$ is large. Thus, there is $s \in (-1/m_1-t,0)$ satisfying the inequalities in (\ref{eq:ineq2}), and then $z_1 = s + im_2$ satisfies
\be
(z,w,n) \sim_\Lam (-z_1-kz-w, (k+1)z+w, n^\pr) ,
\ee
where $n^\pr = \gcd(\cI_\Lam((k+1)z+w),n)$. Similarly, there is $s^\pr \in (-1/m_1-t,0)$ such that $z_1^\pr = s^\pr + im_2$ satisfies
\be
(z,w,n) \sim_\Lam (-z_1^\pr-(k+1)z-w, (k+2)z+w, n^{\pr\pr}) ,
\ee
where $n^{\pr\pr} = \gcd(\cI_\Lam((k+2)z+w),n)$. Both $n^\pr$ and $n^{\pr\pr}$ divide $n$. If $n^\pr = n$ and $n^{\pr\pr} = n$, then $(k+1)z + w \in n\Lam$ and $(k+2)z + w \in n\Lam$, which implies $z,w \in n\Lam$, a contradiction since $\gcd(\cI_\Lam(z),\cI_\Lam(w),n) = 1$. Therefore, $n^\pr < n$ or $n^{\pr\pr} < n$, and by iterating this procedure, we see that $E(m_1,n_1,n)$ contains a connected component of the form $C_{a,b,1}$. \\

\paragraph{\bf Step 2:} Fix $(z,w,1) \in \cT_{(0,1)}^\Lam$, and let $n_2 = \gcd(\cI_\Lam(z),\cI_\Lam(w)) > 0$. We next show that the equivalence class of $(z,w,1)$ contains $C_{n_2,0,1}$ if $\im(z) > 0$, and contains $C_{-n_2,0,1}$ if $\im(z) < 0$. As in Step 1, we may assume $\re(z) = 0$ and $\im(z) > 0$, so that $z = im_1$ and $w = t + in_1$ with $m_1,n_1 \in \Z$, $0 \leq n_1 < m_1$, and $0 < -t < 1/m_1$. If $n_1 = 0$, we are done, so assume $n_1 > 0$ as well.

Let $m_2 = 0$ and $k = -1$. Since $0 < n_1/m_1 < 1$, we have $k + n_1/m_1 < 0$. As $-s$ ranges from $0$ to $1/m_1 + t$, the middle expression in the second group of inequalities in (\ref{eq:ineq2}) ranges from $0$ to $(1/tm_1 + 1)(k + n_1/m_1) > 0$. Letting $z_1 = s < 0$ be sufficiently small, we see that $s$ satisfies the inequalities in (\ref{eq:ineq1}), and $(z,w,1)$ and $z_1$ satisfy (\ref{eq:gen4R+iZ}), therefore
\be
(z,w,1) \sim_\Lam (-z_1-kz-w,(k+1)z+w,1) = (-s+z-w, w, 1) .
\ee
This shows that $E(m_1,n_1,1) = E(m_1-n_1,n_1,1)$. By (\ref{eq:gen1R+iZ}), we also have $E(m_1,n_1,1) = E(m_1,d m_1 + n_1,1)$ for all $d \in \Z$. Thus, by iterating this procedure and running the Euclidean algorithm on the pair of integers $(m_1,n_1)$, we see that $E(m_1,n_1,1) = E(n_2,0,1)$. \\

\paragraph{\bf Step 3:} We conclude by showing that all of the connected components $C_{n_2,0,1}$, $n_2 \in \Z \sm \{0\}$, lie in the same equivalence class. By Step 2, we have $E(n_2,0,1) = E(n_2,n_2,1)$. Suppose $n_2 > 1$. Let $z = in_2$, and let $w = t + in_2$ with $0 < -t < 1/n_2$. Fix $m_2 \in \Z$ with $0 < m_2 < n_2$, and let $k = 0$. As $-s$ ranges from $0$ to $1/n_2 + t$, the middle expression in the second group of inequalities in (\ref{eq:ineq2}) ranges from $m_2/n_2 > 0$ to $1/tn_2 + 1 + m_2/n_2$ (in this case, $m_1 = n_1 = n_2$). Since $0 < m_2/n_2 < 1$, for sufficiently small $s < 0$, letting $z_1 = s + i m_2$, we have
\be
(z,w,1) \sim_\Lam (-z_1-w,z+w,1) = (-(s+t)-(n_2+m_2)i, t + 2 n_2 i, 1) .
\ee
Then $E(n_2,0,1) = E(-n_2-m_2,2n_2,1)$ for all $0 < m_2 < n_2$. Taking $m_2 = n_2 - 1$, by Step 2 we have $E(n_2,0,1) = E(-2n_2 + 1, 2n_2, 1) = E(-1,0,1)$. If $n_2 > 2$, taking $m_2 = n_2 - 2$, by Step 2 we also have $E(n_2,0,1) = E(-2n_2 + 2, 2n_2, 1) = E(-2,0,1)$. Similarly, for all $n_2 \in \Z$ with $n_2 < -1$, we have $E(n_2,0,1) = E(1,0,1)$, and if $n_2 < -2$, we also have $E(n_2,0,1) = E(2,0,1)$. We have shown that for all integers $n_2 > 2$,
\be
E(n_2,0,1) = E(-1,0,1), \quad E(n_2,0,1) = E(-2,0,1) = E(1,0,1) ,
\ee
and similarly, for all integers $n_2 < - 2$,
\be
E(n_2,0,1) = E(1,0,1), \quad E(n_2,0,1) = E(2,0,1) = E(-1,0,1) .
\ee
Thus, for all $n_2 \in \Z \sm \{0\}$, we have $E(n_2,0,1) = E(1,0,1)$ and we are done.
\end{proof}


\section{Periods under connected sums} \label{sec:induct}

This brief section contains additional preparatory lemmas for our proofs of Theorems \ref{thm:conn} and \ref{thm:connR+iZ}, on the behavior of the group of absolute periods under connected sums. Roughly speaking, the absolute periods of the holomorphic $1$-forms considered in Theorem \ref{thm:conn} do not satisfy any atypical $\Q$-linear relations and do not have any atypical pairs of parallel absolute periods. We first show that removing a torus from a connected sum presentation preserves these properties, and we use these properties to give an explicit criterion for a leaf of $\cA(\kap)$ to contain holomorphic $1$-forms with a dense $\GL^+(2,\R)$-orbit.

\begin{lem} \label{lem:preserved}
Suppose $(X^\pr,\om^\pr) \in \Om\cM_{g+1}(\kap^\pr)$ arises from $(X,\om) \in \Om\cM_g(\kap)$ by a connected sum with a torus, and let $\gam$ be the associated geodesic segment in $(X,\om)$. If $\Per(\om^\pr)$ satisfies
\be
\Per(\om^\pr) \cong \Z^{2(g+1)}, \quad \Per(\om^\pr) \cap \R z \subset \Q z \quad \text{for all } z \in \Per(\om^\pr) ,
\ee
then $\Per(\om)$ satisfies
\be
\Per(\om) \cong \Z^{2g}, \quad \Per(\om) \cap \R z \subset \Q z \quad \text{for all } z \in \Per(\om) .
\ee
Moreover, $\int_\gam \om \notin \bigcup_{z \in \Per(\om)} \R z$.
\end{lem}

\begin{proof}
The inclusion $X \sm \gam \hra X^\pr$ induces an injection $f : H_1(X;\Z) \hra H_1(X^\pr;\Z)$ such that $\int_c \om = \int_{f(c)} \om^\pr$ for all $c \in H_1(X;\Z)$. Since $\Per(\om^\pr) \cong \Z^{2(g+1)}$, the homomorphism $H_1(X^\pr;\Z) \ra \C$, $c^\pr \mapsto \int_{c^\pr} \om^\pr$, is injective. This implies the homomorphism $H_1(X;\Z) \ra \C$, $d \mapsto \int_d \om$ is also injective, and thus $\Per(\om) \cong \Z^{2g}$. Next, fix $z \in \Per(\om)$. Since $\int_c \om = \int_{f(c)} \om^\pr$ for all $c \in H_1(X;\Z)$, we have $\Per(\om) \subset \Per(\om^\pr)$. Then since $\Per(\om^\pr) \cap \R z \subset \Q z$, we also have $\Per(\om) \cap \R z \subset \Q z$.

Lastly, let $\gam^\pm$ be the given splitting of $(X^\pr,\om^\pr)$. In other words, $(X,\om)$ is the holomorphic $1$-form of genus $g$ obtained from $(X^\pr,\om^\pr)$ by slitting and regluing $\gam^\pm$ and taking the connected component of genus $g$. Let $c^\pr = [\gam^\pm] \in H_1(X^\pr;\Z)$, and let $C$ be the associated cylinder on $(X^\pr,\om^\pr)$. Then $c^\pr$ is represented by a closed geodesic $\al$ contained in $C$. Let $\bet \subset C \cup Z(\om^\pr)$ be a saddle connection crossing $C$. Then $\bet$ is a closed loop intersecting $\al$ exactly once. On the other hand, the image $f(H_1(X;\Z))$ is generated by homology classes represented by simple closed curves in $X^\pr \sm \ol{C}$, which are in particular disjoint from $\al \cup \bet$. Thus, $c^\pr \notin f(H_1(X;\Z))$. Then since $\Per(\om^\pr) \cong \Z^{2(g+1)}$ and $\Per(\om^\pr) \cap \R z \subset \Q z$ for all $z \in \Per(\om^\pr)$, we have $\int_{c^\pr} \om^\pr \notin \R \int_{f(c)} \om$ for all $c \in H_1(X;\Z)$. This means $\int_\gam \om \notin \bigcup_{z \in \Per(\om)} \R z$.
\end{proof}

By \cite{EMM:closures} and \cite{Wri:field}, a $\GL^+(2,\R)$-orbit closure in a stratum is defined in local period coordinates by homogeneous $\R$-linear equations with coefficients in a number field. The smallest such number field is the {\em affine field of definition} of the orbit closure. By \cite{Wri:cylinder}, when the affine field of definition is not $\Q$, every holomorphic $1$-form in the orbit closure has a pair of parallel cylinders whose circumferences have an irrational ratio. These results yield a simple criterion for a leaf of $\cA(\kap)$ to contain holomorphic $1$-forms with a dense $\GL^+(2,\R)$-orbit. Recall that for $(X,\om) \in \Om\cM_g(\kap)$, the leaf of $\cA(\kap)$ through $(X,\om)$ is denoted $L(\om)$.

\begin{lem} \label{lem:generic}
Fix $(X_0,\om_0) \in \Om\cM_g(\kap)$ such that $\Per(\om_0) \cong \Z^{2g}$ and $\Per(\om_0) \cap \R z \subset \Q z$ for all $z \in \Per(\om_0)$. For a dense subset of $(X,\om) \in L(\om_0)$, the $\GL^+(2,\R)$-orbit of $(X,\om)$ is dense in its connected component in $\Om\cM_g(\kap)$.
\end{lem}

\begin{proof}
Fix $(X,\om) \in L(\om_0)$. Let $\{a_j,b_j\}_{j=1}^g$ be a basis for $H_1(X;\Z)$, and extend to a basis $\{a_j,b_j\}_{j=1}^g \cup \{c_j\}_{j=1}^{n-1}$ for $H_1(X,Z(\om);\Z)$. Suppose the $\GL^+(2,\R)$-orbit of $(X,\om)$ is not dense in its stratum component. Then the period coordinates $\int_{a_1} \om, \dots, \int_{c_{n-1}} \om$ satisfy a homogeneous linear equation with coefficients in a number field $K \subset \R$. Since there are only countably many possible equations, after moving a small distance in $L(\om_0)$ we can ensure that the coefficients of $\int_{c_j} \om$ in such an equation must all be $0$. This means
\be
\sum_{j=1}^g \left(s_j \int_{a_j} \om + t_j \int_{b_j} \om\right) = 0
\ee
for some $s_j,t_j \in K$. Since $\Per(\om) \cong \Z^{2g}$, at least one of the coefficients $s_j,t_j$ is not in $\Q$. Therefore, the affine field of definition of the $\GL^+(2,\R)$-orbit closure of $(X,\om)$ is not $\Q$. Then $(X,\om)$ has parallel cylinders $C_1,C_2$, and closed geodesics $\al_j \subset C_j$, such that the absolute periods $\int_{\al_1} \om$ and $\int_{\al_2} \om$ satisfy $\int_{\al_1} \om / \hspace{-3pt} \int_{\al_2} \om \in \R \sm \Q$. But then
\be
\Z \int_{\al_1} \om + \Z \int_{\al_2} \om \subset \Per(\om) \cap \R \int_{\al_1} \om ,
\ee
a contradiction since $\Per(\om) = \Per(\om_0)$.
\end{proof}

Next, we prove versions of Lemmas \ref{lem:preserved} and \ref{lem:generic} for the holomorphic $1$-forms considered in Theorem \ref{thm:connR+iZ}. The absolute periods of these holomorphic $1$-forms do not satisfy any atypical $\Q$-linear relations, and they are not dense in $\C$, which forces them to be dense in a closed subgroup of the form $M \cdot (\R + i\Z)$ for some $M \in \SL(2,\R)$.

\begin{lem} \label{lem:preservedR+iZ}
Suppose $(X^\pr,\om^\pr) \in \Om\cM_{g+1}(\kap^\pr)$ arises from $(X,\om) \in \Om\cM_g(\kap)$ by a connected sum with a torus, and let $\gam$ be the associated geodesic segment in $(X,\om)$. If $\Per(\om^\pr)$ satisfies
\be
\Per(\om^\pr) \cong \Z^{2(g+1)}, \text{ and } \Per(\om^\pr) \text{ is not dense in } \C ,
\ee
then $\Per(\om)$ satisfies
\be
\Per(\om) \cong \Z^{2g}, \text{ and } \Per(\om) \text{ is not dense in } \C .
\ee
Moreover, letting $M \in \SL(2,\R)$ be such that $\Per(\om)$ is dense in $M \cdot (\R + i\Z)$, if $\int_\gam \om \notin M \cdot \R$, then $\int_\gam \om \notin \bigcup_{z \in \Per(\om)} \R z$.
\end{lem}

\begin{proof}
Similar to the proof of Lemma \ref{lem:preserved}.
\end{proof}

\begin{lem} \label{lem:genericR+iZ}
Fix $(X_0,\om_0) \in \Om\cM_g(\kap)$ such that $\Per(\om_0) \cong \Z^{2g}$ is not dense in $\C$. For a dense subset of $(X,\om) \in L(\om_0)$, the $\GL^+(2,\R)$-orbit of $(X,\om)$ is dense in its stratum component.
\end{lem}

\begin{proof}
By applying an element of $\SL(2,\R)$, we may assume that $\Per(\om_0)$ is dense in $\R + i\Z$. As in the proof of Lemma \ref{lem:generic}, for a dense subset of $(X,\om) \in L(\om_0)$, if the $\GL^+(2,\R)$-orbit closure of $(X,\om)$ is not dense in its stratum component, then the affine field of definition of the orbit closure is not $\Q$. By Theorem 2 in \cite{Mas:geodesic}, the set of directions of cylinders on $(X,\om)$ is dense in $S^1$, so there are non-horizontal parallel cylinders $C_1,C_2$ on $(X,\om)$ and closed geodesics $\al_j \subset C_j$ satisfying $\int_{\al_1} \om / \hspace{-3pt} \int_{\al_2} \om \in \R \sm \Q$. However, since $\Per(\om) = \Per(\om_0) \subset \R + i\Z$ and $\Per(\om) \cong \Z^{2g}$, we have $\Per(\om) \cap \R z \subset \Q z$ for all $z \in \Per(\om) \sm \R$, a contradiction.
\end{proof}


\section{Connectivity of spaces of isoperiodic forms} \label{sec:conn}

In this section, we prove that spaces of isoperiodic forms in most stratum components are typically connected, and we describe a monodromy obstruction to the connectivity of these spaces in the remaining stratum components. We also deduce the ergodicity of the absolute period foliation and provide explicit full measure sets of dense leaves in most strata. \\

\paragraph{\bf Spaces of isoperiodic forms.} Fix $g \geq 2$. Recall that $S_g$ denotes the connected closed oriented surface of genus $g$, and the cohomology group $H^1(S_g;\C)$ is identified with the group of homomorphisms $H_1(S_g;\Z) \ra \C$. Let $\langle \al, \bet \rangle = \frac{i}{2} \int_{S_g} \al \wedge \ol{\bet}$ be the intersection form on $H^1(S_g;\C)$. For $\phi \in H^1(S_g;\C)$, define $\Per(\phi)$ to be the image of the associated homomorphism $H_1(S_g;\Z) \ra \C$. The self-intersection $\lr{\phi,\phi}$ is given by
\be
\lr{\phi,\phi} = \sum_{j=1}^g \im\left(\ol{\phi(a_j)}\phi(b_j)\right)
\ee
where $\{a_j,b_j\}_{j=1}^g$ is a symplectic basis for $H_1(S_g;\Z)$ with respect to the algebraic intersection form on $H_1(S_g;\Z)$. One can verify that this formula is independent of the choice of symplectic basis, for instance using one of the generating sets for the integral symplectic group $\Sp(2g,\Z)$ found in Section 6.1 of \cite{FM:primer}. We say that $\phi$ is {\em positive} if $\langle \phi, \phi \rangle > 0$. When $\phi$ is positive, we say that $\phi$ is {\em elliptic of degree $d > 0$} if $\Per(\phi)$ is a lattice in $\C$ and the naturally associated homotopy class of maps $S_g \ra \C/\Per(\phi)$ has degree $d$.

Let ${\rm Homeo}^+(S_g,X)$ be the set of orientation-preserving homeomorphisms $S_g \ra X$. The {\em space of isoperiodic forms} representing $\phi$ is defined by
\be
\cM(\phi) = \left\{(X,\om) \in \Om\cM_g : f^\ast([\om]) = \phi \text{ for some } f \in {\rm Homeo}^+(S_g,X) \right\} .
\ee
Here, $[\om] \in H^1(X;\C)$ denotes the cohomology class represented by $\om$. Any two holomorphic $1$-forms $(X,\om),(Y,\eta) \in \cM(\phi)$ are {\em isoperiodic}, in the sense that there is a symplectic isomorphism $m : H_1(X;\Z) \ra H_1(Y;\Z)$ such that $\int_c \omega = \int_{m(c)} \eta$ for all $c \in H_1(X;\Z)$. The area of any $(X,\om) \in \cM(\phi)$ is given by $\Area(X,\om) = \lr{\phi,\phi}$. For $\cC$ a stratum component, we define
\be
\cC(\phi) = \cC \cap \cM(\phi) .
\ee
The symplectic automorphism group $\Aut(H_1(S_g;\Z)) \cong \Sp(2g,\Z)$ acts on the group of homomorphisms $H_1(S_g;\Z) \ra \C$, and $\cC(\phi)$ only depends on the orbit of $\phi$ under this action.

Haupt's theorem \cite{Hau:periods} says that $\cM(\phi)$ is nonempty if and only if $\phi$ is positive and not elliptic of degree $1$. Haupt's theorem was rediscovered in \cite{Kap:periods} using tools from homogeneous dynamics, and another proof is given in \cite{CDF:transfer}. A generalization of Haupt's theorem to stratum components was proven in \cite{BJJP:Haupt}, and an independent proof for strata (without considering their connected components) is given in \cite{Fil:Haupt}. In particular, if $\phi$ is positive and $\Per(\phi)$ is not a lattice in $\C$, then $\cC(\phi)$ is nonempty for all stratum components $\cC$.

The proof of Haupt's theorem in \cite{Kap:periods} relies on classifying the $\Sp(2g,\Z)$-orbit closures in the set of positive cohomology classes in $H^1(S_g;\C)$. Roughly speaking, when $g \geq 3$ these orbit closures are determined by the closure of the associated group of absolute periods. For $a > 0$, let $E_a$ be the set of cohomology classes $\phi \in H^1(S_g;\C)$ with $\lr{\phi,\phi} = a$. For $\Lam \subset \C$ a closed subgroup of the form $\Lam = M \cdot (\R + i\Z)$ with $M \in \SL(2,\R)$ and identity component $\Lam_0 = M \cdot \R$, let $E_a^\Lam$ be the set of $\phi \in E_a$ such that $\Per(\phi) + \Lam_0 = \Lam$.

\begin{lem} \label{lem:Kapovich} (\cite{Kap:periods}, see Proposition 3.10 in \cite{CDF:transfer}).
Suppose $g \geq 3$. Fix $\phi \in H^1(S_g;\C)$ with $\lr{\phi,\phi} = a > 0$, and let $\Lam$ be the closure of $\Per(\phi)$ in $\C$.
\begin{enumerate}
    \item If $\Lam = \C$, then the closure of $\Sp(2g,\Z) \cdot \phi$ is $E_a$.
    \item If $\Lam = M \cdot (\R + i\Z)$ with $M \in \SL(2,\R)$, then the closure of $\Sp(2g,\Z) \cdot \phi$ is $E_a^\Lam$.
    \item If $\Lam = \Per(\phi)$, then $\Sp(2g,\Z) \cdot \phi$ is closed.
\end{enumerate}
If $g = 2$, there is one additional case when $\Lam = \C$ and $\phi$ arises from an eigenform for real multiplication. In that case, the closure of $\Sp(2g,\Z) \cdot \phi$ is given by $\SL(2,\R) \cdot \left(\Sp(2g,\Z) \cdot \phi\right)$.
\end{lem}

For more details about eigenforms in genus $2$, we refer to Section 5 in \cite{McM:SL2R}. For our purposes, we will only need to know that when $g = 2$, $\Lam = \C$, and $\phi$ arises from an eigenform for real multiplication, we have $\Per(\phi) \cap \R z \cong \Z^2$ for all nonzero $z \in \Per(\phi)$. In particular, there are atypical pairs of parallel absolute periods in this case. \\

\paragraph{\bf Setup for proving Theorem \ref{thm:conn}.} Let $\cC$ be a stratum component and let $\phi \in H^1(S_g;\C)$ be a positive cohomology class, such that $\cC$ and $\phi$ satisfy the hypotheses of Theorem \ref{thm:conn}. We now begin our proof that $\cC(\phi)$ is connected. It will be easy to reduce to the case of strata with two zeros by splitting zeros, so we will focus on this case here.

A connected component of $\cC(\phi)$ is a leaf of the absolute period foliation of $\cC$. Lemma \ref{lem:sum2} constructs holomorphic $1$-forms in $\cC$ with a splitting. Any splitting persists on an open neighborhood in a stratum, and splittings are preserved by the action of $\GL^+(2,\R)$, so the set of holomorphic $1$-forms in $\cC$ with a splitting is nonempty, open, and $\GL^+(2,\R)$-invariant. Lemma \ref{lem:generic} then ensures that every connected component of $\cC(\phi)$ contains a holomorphic $1$-form $(X,\om)$ with a splitting $\al^\pm$. Let $C$ be the associated cylinder of $\al^\pm$, and let $(z,w)$ be the associated periods of $\al^\pm$. Since $\Per(\phi) \cong \Z^{2g}$, there is a unique pair of homology classes $a,b \in H_1(X;\Z)$ such that $\int_a \om = z$ and $\int_b \om = z$. These homology classes have algebraic intersection $a \cdot b = 1$, and their periods satisfy an area constraint $0 < \im(\ol{z}w) < \lr{\phi,\phi}$.

Slitting and regluing $\al^\pm$ yields a flat torus and a holomorphic $1$-form $(X^\pr,\om^\pr)$ of genus $g - 1$ with a distinguished geodesic segment $s$ coming from $\al^\pm$. The homology group $H_1(X^\pr;\Z)$ is identified with the symplectic orthogonal $\{a,b\}^\perp \subset H_1(X;\Z)$ in a way that preserves absolute periods. Letting $\cC^\pr$ be the stratum component containing $(X^\pr,\om^\pr)$, there is a cohomology class $\phi^\pr \in H^1(S_{g-1};\C)$ such that $(X^\pr,\om^\pr) \in \cC^\pr(\phi^\pr)$, and $\phi^\pr$ can be thought of as the restriction of $\phi$ to the homology of a subsurface of genus $g - 1$.

The results in Section \ref{sec:leaves}, along with an inductive hypothesis, will imply that the component of $\cC(\phi)$ containing $(X,\om)$ contains all holomorphic $1$-forms $(Y,\eta)$ with a splitting whose associated periods are $(z,w)$. Thus, we have a well-defined surjective function from certain pairs of periods to connected components of $\cC(\phi)$. We need to show that this function is constant. The results in Section \ref{sec:pair} provide a criterion for when two pairs of periods determine the same component of $\cC(\phi)$. This criterion will reduce our connectivity problem to a delicate algebraic problem, which we will then solve. \\

\paragraph{\bf Spaces of isoperiodic forms with a fixed splitting.} We now formalize the previous discussion. Fix $g \geq 3$. Let $\phi \in H^1(S_g;\C)$ be a positive cohomology class, and let $\cC$ be a connected component of a stratum $\Om\cM_g(m_1,m_2)$ with $m_1 \geq m_2$. Throughout this subsection, we make the following additional assumptions.
\begin{itemize}
    \item We assume that $\cC$ is a nonhyperelliptic component and that $m_1,m_2$ are odd.
    \item We assume that $\Per(\phi) \cong \Z^{2g}$ and that $\Per(\phi) \cap \R z \subset \Q z$ for all $z \in \Per(\phi)$.
\end{itemize}
Fix $z,w \in \Per(\phi)$, and let $a,b \in H_1(S_g;\Z)$ be the unique homology classes such that $\phi(a) = z$ and $\phi(b) = w$. Suppose that $a \cdot b = 1$ and that $0 < \im(\ol{z} w) < \lr{\phi,\phi}$. If $m_2 \geq 3$, let $m_\cC = m_2$, and otherwise, let $m_\cC = m_1$. Then we define
\be
\cC(\phi,z,w) = \left\{ (X,\om) \in \cC(\phi) : \begin{matrix} (X,\om) \text{ has a splitting at a zero of $\om$ of order $m_\cC$} \\ \text{with associated periods } (z,w) \end{matrix} \right\} .
\ee
If $m_\cC = m_2$, let $\cC^\pr = \Om\cM_{g-1}(m_1,m_2 - 2)$, and if $m_\cC = m_1$, let $\cC^\pr = \Om\cM_{g-1}(m_1-2,m_2)$. Note that $\cC^\pr$ is connected by Corollary \ref{cor:KZconn}. Choose a symplectic isomorphism $f_1 : H_1(S_{g-1};\Z) \ra \{a,b\}^\perp$. This choice determines a homomorphism $H_1(S_{g-1};\Z) \ra \C$, $c \mapsto \phi(f_1(c))$, and thus a cohomology class $\phi^\pr \in H^1(S_{g-1};\C)$. Different choices yield cohomology classes in the same orbit under $\Aut(H_1(S_{g-1};\Z))$, so the space $\cC^\pr(\phi^\pr)$ is independent of the choice of $f_1$.

\begin{lem} \label{lem:Czwconn}
Suppose that $\cC^\pr(\phi^\pr)$ is connected. Then $\cC(\phi,z,w)$ is connected.
\end{lem}

\begin{proof}
Fix $(X,\om) \in \cC(\phi,z,w)$. Since $(X,\om) \in \cC(\phi)$, there is a symplectic isomorphism $f : H_1(S_g;\Z) \ra H_1(X;\Z)$ such that $\phi(c) = \int_{f(c)} \om$ for all $c \in H_1(S_g;\Z)$. Let $\al^\pm$ be a splitting of $(X,\om)$ at a zero of $\om$ of order $m_\cC$ with associated periods $(z,w)$, and let $a_1,b_1 \in H_1(X;\Z)$ be the unique homology classes such that $\int_{a_1} \om = z$ and $\int_{b_1} \om = w$. Let $(X^\pr,\om^\pr) \in \cC^\pr$ be the holomorphic $1$-form in genus $g-1$ obtained by slitting and regluing $\al^\pm$, so there is a symplectic isomorphism $f_2 : H_1(X^\pr;\Z) \ra \{a_1,b_1\}^\perp$ such that $\int_c \om^\pr = \int_{f_2(c)} \om$ for all $c \in H_1(X^\pr;\Z)$. Since $\phi(a) = \int_{a_1} \om$ and $\phi(b) = \int_{b_1} \om$, we must have $f(a) = a_1$ and $f(b) = b_1$, so $f$ restricts to an isomorphism $\{a,b\}^\perp \cong \{a_1,b_1\}^\perp$. Then the composition $f_2^{-1} \circ f \circ f_1 : H_1(S_{g-1};\Z) \ra H_1(X^\pr;\Z)$ satisfies $\phi^\pr(c) = \int_{f_2^{-1}(f(f_1(c)))} \om^\pr$ for all $c \in H_1(S_{g-1};\Z)$, which means $(X^\pr,\om^\pr) \in \cC^\pr(\phi^\pr)$.

The splitting $\al^\pm$ determines a geodesic segment $s_1$ on $(X^\pr,\om^\pr)$ emanating from a zero of $\om^\pr$ of order $m_\cC - 2$. By Lemma \ref{lem:preserved}, the segment $s_1$ is not parallel to any absolute period of $(X^\pr,\om^\pr)$. Let $I = \{tz : 0 \leq t \leq 1\}$. Then any saddle connection $\gam^\pr$ on $(X^\pr,\om^\pr)$ with holonomy in $I$ must have distinct endpoints, and the interior of $\gam^\pr$ must be disjoint from $s_1$, so there is a saddle connection $\gam$ on $(X,\om)$ with the same holonomy as $\gam^\pr$. Thus, after replacing $(X,\om)$ with a nearby holomorphic $1$-form in the leaf $L(\om)$, we may assume that $\Gam(\om^\pr)$ is disjoint from $I$.

Now fix another holomorphic $1$-form $(Y,\eta) \in \cC(\phi,z,w)$. Let $\bet^\pm$ be a splitting of $(Y,\eta)$ at a zero of order $m_\cC$ with associated periods $(z,w)$, and let $(Y^\pr,\eta^\pr) \in \cC^\pr$ be the holomorphic $1$-form obtained by slitting and regluing $\bet^\pm$. As above, we have $(Y^\pr,\eta^\pr) \in \cC^\pr(\phi^\pr)$, and after replacing $(Y,\eta)$ with a nearby holomorphic $1$-form in $L(\eta)$, we may assume that $\Gam(\eta^\pr)$ is disjoint from $I$.

By assumption, $\cC^\pr(\phi^\pr)$ is connected, and is therefore a leaf of the absolute period foliation of $\cC^\pr$. We denote this leaf by $L$. Since $\Gam(\om^\pr)$ and $\Gam(\eta^\pr)$ are disjoint from $I$, we have $(X^\pr,\om^\pr), (Y^\pr,\eta^\pr) \in L \sm L(I)$, and Lemma \ref{lem:LIconn} tells us that there is a path $\varphi : [0,1] \ra L \sm L(I)$ from $(X^\pr,\om^\pr)$ to $(Y^\pr,\eta^\pr)$. Let $Z_1$ be the zero of $\om^\pr$ disjoint from $s_1$, and let $Z_2$ be the other zero of $\om^\pr$. As we travel along $\varphi([0,1])$, the zero $Z_1$ moves relative to $Z_2$ while remaining disjoint from $s_1$. The path $\varphi$ then determines a path in $L(\om)$ along which the splitting $\al^\pm$ persists, from $(X,\om)$ to a holomorphic $1$-form that arises from $(Y^\pr,\eta^\pr)$ via a connected sum with a torus where the splitting has associated periods $(z,w)$. Unfortunately, this is not enough to ensure that we have a path in $L(\om)$ from $(X,\om)$ to $(Y,\eta)$, since our connected sum construction also depends on a choice of prong-marking. However, Lemma \ref{lem:coverleaf} ensures that the connected sums from all possible prong-markings of $(Y^\pr,\eta^\pr)$ lie on the same leaf.

To see this, let $\Om\cM_{g-1}(\kap^\pr)$ be the stratum containing $\cC^\pr$, and consider the stratum cover $p : \wt{\Om}\cM_{g-1}(\kap^\pr; m_\cC - 2) \ra \Om\cM_{g-1}(\kap^\pr)$ by prong-marked holomorphic $1$-forms from (\ref{eq:prong}). By Lemmas \ref{lem:preserved} and \ref{lem:generic}, the leaf $L$ contains a holomorphic $1$-form whose $\GL^+(2,\R)$-orbit is dense in $\cC^\pr$. Then by Lemma \ref{lem:coverleaf}, the preimage $\wt{L} = p^{-1}(L)$ is a leaf of $\cA(\kap^\pr;m_\cC - 2)$, and in particular is path-connected. The preimages $p^{-1}(X^\pr,\om^\pr)$ and $p^{-1}(Y^\pr,\eta^\pr)$ are contained in $\wt{L} \sm \wt{L}(I)$, since saddle connection holonomies do not depend on the choice of prong-marking, and $\wt{L} \sm \wt{L}(I)$ is also path-connected as indicated below Lemma \ref{lem:LIconn}. Thus, there is a path in $\wt{L} \sm \wt{L}(I)$ from any element of $p^{-1}(X^\pr,\om^\pr)$ to any element of $p^{-1}(Y^\pr,\eta^\pr)$.

Now let $\kap = (m_1,m_2)$, and consider the connected sum map $\Psi : \cT(\kap^\pr; m_\cC - 2) \ra \Om\cM_g(\kap)$. There is $(X^\pr,\wt{\om^\pr}) \in p^{-1}(X^\pr,\om^\pr)$ and $(X^\pr,\wt{\om^\pr},T_1) \in \Psi^{-1}(X,\om)$, where $T_1 = (\gam_1,w)$ and $\gam_1 \in \wt{\C}^\ast_{m_\cC-1}$ is a segment with holonomy $z$. Similarly, there is $(Y^\pr,\wt{\eta^\pr}) \in p^{-1}(Y^\pr,\eta^\pr)$ and $(Y^\pr,\wt{\eta^\pr},T_2) \in \Psi^{-1}(Y,\eta)$, where $T_2 = (\gam_2,w)$ and $\gam_2 \in \wt{\C}^\ast_{m_\cC-1}$ is a segment with holonomy $z$. By Corollary \ref{cor:sumleaf}, the leaf of $\cF_\cT$ through $(X^\pr,\wt{\om^\pr},T_1)$ is given by $(\wt{L} \sm \wt{L}(I)) \times \{T_1\}$, and the leaf of $\cF_\cT$ through $(Y^\pr,\wt{\eta^\pr},T_2)$ is given by $(\wt{L} \sm \wt{L}(I)) \times \{T_2\}$. The path $\varphi : [0,1] \ra L \sm L(I)$ lifts to a path in $(\wt{L} \sm \wt{L}(I)) \times \{T_1\}$ from $(X^\pr,\wt{\om^\pr},T_1)$ to $(Y^\pr,\wt{\eta^\pr},T_1)$. Now, by the previous paragraph, we can replace $(Y^\pr,\wt{\eta}^\pr)$ with any other element of $p^{-1}(Y^\pr,\eta^\pr)$. This amounts to changing the choice of prong $\tht(\wt{\eta^\pr})$, and the angle between different prongs is an integer multiple of $2\pi$. Also, the angle between $\gam_1$ and $\gam_2$ is an integer multiple of $2\pi$. If we rotate the chosen prong on $(Y^\pr,\wt{\eta^\pr})$ by $2\pi$ clockwise, and simultaneously rotate the segment $\gam_1$ by $2\pi$ counterclockwise, then the result of the connected sum construction is unchanged since the same segment is being slit on the underlying holomorphic $1$-form $(Y^\pr,\eta^\pr)$. Thus, after replacing $(Y^\pr,\wt{\eta}^\pr)$ with another element of $p^{-1}(Y^\pr,\eta^\pr)$, we have $\Psi(Y^\pr,\wt{\eta^\pr},T_1) = (Y,\eta)$. Then by applying $\Psi$ to a path in $(\wt{L} \sm \wt{L}(I)) \times \{T_1\}$ from $(X^\pr,\wt{\om^\pr},T_1)$ to $(Y^\pr,\wt{\eta^\pr},T_1)$, we obtain a path in $\cC(\phi,z,w)$ from $(X,\om)$ to $(Y,\eta)$. Thus, $\cC(\phi,z,w)$ is connected.
\end{proof}

\begin{lem} \label{lem:Czwunion}
Suppose that Theorem \ref{thm:conn} is true for $\cC^\pr$. Then every component of $\cC(\phi)$ contains $\cC(\phi,z^\pr,w^\pr)$ for some $z^\pr,w^\pr \in \C$.
\end{lem}

\begin{proof}
Fix $(X,\om) \in \cC(\phi)$. The connected component of $\cC(\phi)$ containing $(X,\om)$ is a leaf $L(\om)$ of the absolute period foliation of $\cC$. By Lemma \ref{lem:generic}, after replacing $(X,\om)$ with a nearby holomorphic $1$-form on $L(\om)$, we may assume that the $\GL^+(2,\R)$-orbit of $(X,\om)$ is dense in $\cC$. By Lemma \ref{lem:sum2}, the set of holomorphic $1$-forms in $\cC$ with a splitting is nonempty, so since this set is open and $\GL^+(2,\R)$-invariant, $(X,\om)$ has a splitting $\al^\pm$. Let $(z^\pr,w^\pr)$ be the associated periods of $\al^\pm$, and let $(X^\pr,\om^\pr) \in \cC^\pr$ be the holomorphic $1$-form in genus $g - 1$ obtained from $(X,\om)$ by slitting and regluing $\al^\pm$. By Lemma \ref{lem:preserved}, we have $(X^\pr,\om^\pr) \in \cC^\pr(\phi^\pr)$ for some positive $\phi^\pr \in H^1(S_{g-1};\C)$ satisfying the hypotheses of Theorem \ref{thm:conn}. Then by Lemma \ref{lem:Czwconn} and our inductive hypothesis, $\cC(\phi,z^\pr,w^\pr)$ is connected, so $\cC(\phi,z^\pr,w^\pr)$ is contained in $L(\om)$.
\end{proof}

\begin{lem} \label{lem:Czwequal1}
For all $n \in \Z$, we have $\cC(\phi,z,w) = \cC(\phi,z, nz + w)$.
\end{lem}

\begin{proof}
This is immediate from the definitions. A splitting with associated periods $(z,w)$ and a splitting with associated periods $(z,nz + w)$ are the same thing, since the period $z$ is uniquely determined by the splitting, while the period $w$ depends on the choice of saddle connection crossing the associated cylinder. Different choices of this saddle connection yield $nz + w$ in place of $w$ for any $n \in \Z$.
\end{proof}

\begin{lem} \label{lem:Czwequal2}
Let $a_1 \in \{a,b\}^\perp$ be a primitive homology class, and let $z_1 = \phi(a_1)$. If $z,w,z_1$ satisfy the inequalities in (\ref{eq:gen}), then $\cC(\phi,z,w) \cap \cC(\phi,-z_1-w,z+w)$ is nonempty.
\end{lem}

\begin{proof}
We need to show there exists $(X,\om) \in \cC(\phi)$ with two splittings whose associated periods are $(z,w)$ and $(-z_1-w,z+w)$, respectively. The construction in the proof of Lemma \ref{lem:sum2} provides such a holomorphic $1$-form, as follows.

Recall that in the construction in the proof of Lemma \ref{lem:sum2}, we first glue a pair of flat tori $T_1 = (\C/(\Z z_1 + \Z w_1), dz)$ and $T_2 = (\C/(\Z z_2 + \Z w_2), dz)$ along a pair of homologous saddle connections $s^\pm$ with holonomy $u$. We then iteratively form connected sums, using the flat torus $T_0 = (\C / (\Z z + \Z w), dz)$ and a segment on $T_1$ with holonomy $z$, and then using flat tori $T_j = (\C / (\Z z_j + \Z w_j), dz)$ and short segments on $T_2$ with holonomy $z_j$ for $3 \leq j \leq g - 1$. The resulting holomorphic $1$-form $(X,\om)$ has a splitting on $T_1$ with associated periods $(z,w)$, and has $g - 3$ splittings on $T_2$ with associated periods $(z_3,w_3),\dots,(z_{g-1},w_{g-1})$. Additionally, $(X,\om)$ has a splitting with associated periods $(-z_1-w,z+w)$. The parameters $z,w,z_1,w_1,\dots,z_{g-1},w_{g-1},u$ are local period coordinates, and the parameters $z,w,z_1,w_1,\dots,z_{g-1},w_{g-1}$ are the periods of a symplectic basis for $H_1(X;\Z)$.

In the setting of Lemma \ref{lem:Czwequal2}, the periods $z,w,z_1$ are given to us, and they satisfy the inequalities in (\ref{eq:gen}). In particular, the flat torus $T_0$ is determined. The homology classes $a,b,a_1$ determined by $z = \phi(a)$, $w = \phi(b)$, and $z_1 = \phi(a_1)$ satisfy $a \cdot b = 1$ and $a_1 \in \{a,b\}^\perp$. For the flat torus $T_1$, we need to choose $w_1 \in \Per(\phi)$ such that the homology class $b_1$ determined by $w_1 = \phi(b_1)$ satisfies $b_1 \in \{a,b\}^\perp$, $a_1 \cdot b_1 = 1$, and $0 < \im(\ol{z}z_1) < \im(\ol{z}_1 w_1) < 1 - \im(\ol{z} w)$. Fix $b_1^\pr \in \{a,b\}^\perp$ such that $a_1 \cdot b_1^\pr = 1$. To find the desired $w_1$, it is enough to show the set
\be
V = \left\{ \im(\ol{z}_1 w_1^\pr) : w_1^\pr \in \phi(b_1^\pr) + \phi(\{a,b,a_1,b_1^\pr\}^\perp) \right\}
\ee
is dense in $\R$. Since $V$ is a translate of the abelian group
\be
V_1 = \left\{\im(\ol{z}_1 v) : v \in \phi(\{a,b,a_1,b_1^\pr\}^\perp)\right\} ,
\ee
this is equivalent to showing that $V_1$ is dense in $\R$. If $V_1$ were not dense in $\R$, we would have $V_1 \cong \Z$. Since $\{a,b,a_1,b_1^\pr\}^\perp$ has rank $2g - 4 \geq 2$ and $\Per(\phi) \cong \Z^{2g}$, there would be a nontrivial $\Z$-linear relation $n_1 \im(\ol{z}_1 \phi(c_1)) + n_2 \im(\ol{z}_1 \phi(c_2)) = 0$ with $c_1,c_2 \in \{a,b,a_1,b_1^\pr\}^\perp$ linearly independent. This means $n_1 \phi(c_1) + n_2 \phi(c_2) \in \R z_1$. However, if $n_1 \phi(c_1) + n_2 \phi(c_2) = 0$, then $\Per(\phi)$ would have rank less than $2g$, and if $n_1 \phi(c_1) + n_2 \phi(c_2) \neq 0$, then $\Per(\phi) \cap \R z_1$ would have rank at least $2$, contradicting the assumptions on $\Per(\phi)$. Thus, $V_1$ is dense in $\R$ and the desired $w_1 \in \Per(\phi)$ exists.

Suppose $g = 3$. Once $z,w,z_1,w_1$ are fixed, we can choose $z_2,w_2 \in \Per(\phi)$ to be the periods of any pair of homology classes $a_2,b_2$ such that $\Z a_2 + \Z b_2 = \{a,b,a_1,b_1\}^\perp$ and $a_2 \cdot b_2 = 1$. We have $\im(\ol{z}_2 w_2) = \lr{\phi,\phi} - \im(\ol{z} w) - \im(\ol{z}_1 w_1) > 0$, so we are done in this case.

Suppose $g \geq 4$. Once $z,w,z_1,w_1$ are fixed, we can choose $z_2,w_2,\dots,z_{g-1},w_{g-1} \in \Per(\phi)$ to be the periods of any symplectic basis $a_2,b_2,\dots,a_{g-1},b_{g-1}$ of $\{a,b,a_1,b_1\}^\perp$ such that $\im(\ol{z_j} w_j) > 0$ for $2 \leq j \leq g - 1$, such that $z_3,\dots,z_{g-1}$ are small (depending on $z_2,w_2$), and such that $z_3,\dots,z_{g-1}$ are pairwise non-parallel. The conditions on the complex numbers $z_2,w_2,\dots,z_{g-1},z_{g-1}$ are open conditions. If $m_{g-2} : H_1(S_{g-2};\Z) \ra \{a,b,a_1,b_1\}^\perp$ is a symplectic isomorphism and $\phi_{g-2} \in H^1(S_{g-1};\C)$ satisfies $\phi_{g-2}(c) = \phi(m_{g-2}(c))$ for all $c \in H_1(S_{g-2};\Z)$, then $\phi_{g-2}$ is positive and satisfies $\Per(\phi_{g-2}) \cong \Z^{2(g-2)}$ and $\Per(\phi_{g-2}) \cap \R z_0 \subset \Q z_0$ for all $z_0 \in \Per(\phi_{g-2})$ by Lemma \ref{lem:preserved}. In particular, $\Per(\phi_{g-2})$ is dense in $\C$. Thus, Kapovich's classification of $\Sp(2(g-2),\Z)$-orbit closures of positive cohomology classes from Lemma \ref{lem:Kapovich} ensures that such a choice of periods $z_2,w_2,\dots,z_{g-1},w_{g-1}$ is possible.
\end{proof}

\paragraph{\bf Polarized modules.} Next, we explain why Lemmas \ref{lem:Czwconn}-\ref{lem:Czwequal2} reduce the connectivity of $\cC(\phi)$ to an algebraic problem, and we solve this algebraic problem.

Let $\Lam \subset \C$ be a subgroup isomorphic to $\Z^{2g}$. A unimodular symplectic form on $\Lam$ is a bilinear map $\Lam \times \Lam \ra \Z$, $(a,b) \mapsto a \cdot b$, such that there is a symplectic basis $\{a_j,b_j\}_{j=1}^g$ for $\Lam$, meaning $a_j \cdot b_k = - b_k \cdot a_j = \del_{jk}$ and $a_j \cdot a_k = b_j \cdot b_k = 0$. Following \cite{McM:isoperiodic}, a {\em polarized module} is a subgroup $\Lam \subset \C$ isomorphic to $\Z^{2g}$ equipped with a unimodular symplectic form $\Lam \times \Lam \ra \Z$, $(a,b) \mapsto a \cdot b$, such that $\sum_{j=1}^g \im(\ol{a_j} b_j) > 0$, where $\{a_j,b_j\}_{j=1}^g$ is any symplectic basis for $\Lam$. If a cohomology class $\phi \in H^1(S_g;\C)$ is positive and $\Per(\phi) \cong \Z^{2g}$, then $\Per(\phi)$ is a polarized module, where the symplectic form is inherited from the algebraic intersection form on $H_1(S_g;\Z)$. Any holomorphic $1$-form $(X,\om) \in \cM(\phi)$ has area given by $\Area(X,\om) = \sum_{j=1}^g \im(\ol{a_j} b_j)$. An element $a \in \Lam$ is {\em primitive} if there does not exist $n \in \Z \sm \{\pm 1\}$ and $a^\pr \in \Lam$ such that $a = na^\pr$. Equivalently, there exists $b \in \Lam$ such that $a \cdot b = 1$. A submodule $V \subset \Lam$ is {\em primitive} if there does not exist $v \in \Lam \sm V$ and $n \in \Z \sm \{0\}$ such that $nv \in V$.

Suppose that for any symplectic basis $\{a_j,b_j\}_{j=1}^g$ of $\Lam$, we have $\sum_{j=1}^g \im(\ol{a_j} b_j) = 1$. Additionally, suppose that $\Lam \cap \R z \subset \Q z$ for all $z \in \Lam$. Define
\be
\Lam_{(0,1)} = \left\{(a,b) \in \Lam \times \Lam : a \cdot b = 1, \; 0 < \im(\ol{a} b) < 1 \right\} .
\ee
Let $\phi \in H^1(S_g;\C)$ be a positive cohomology class such that $\Per(\phi) = \Lam$ as a polarized module. Then by Lemma \ref{lem:Czwconn}, each pair $(z,w) \in \Lam_{(0,1)}$ determines a unique connected component of $\cC(\phi)$, namely, the component containing $\cC(\phi,z,w)$. By Lemma \ref{lem:Czwunion}, every component of $\cC(\phi)$ contains $\cC(\phi,z,w)$ for some $(z,w) \in \Lam_{(0,1)}$. Let $\sim_\Lam$ be an equivalence relation on $\Lam_{(0,1)}$ satisfying the following. We suppose that
\begin{equation} \label{eq:gen1alg}
(a,b) \sim_\Lam (a, na+b)
\end{equation}
for all $(a,b) \in \Lam_{(0,1)}$ and all $n \in \Z$. We also suppose that
\begin{equation} \label{eq:gen2alg}
(a,b) \sim_\Lam (-c-b,a+b)
\end{equation}
for all $(a,b) \in \Lam_{(0,1)}$ and all primitive $c \in \{a,b\}^\perp$ such that
\begin{equation} \label{eq:gen3alg}
0 < \im(\ol{a} c) < 1 - \im(\ol{a} b), \quad 0 < \im(\ol{c} b) < \im(\ol{a} (b + c)) .
\end{equation}
The above equivalences for $\sim_\Lam$ are variants of the equivalences for $\sim$ in (\ref{eq:gen1}) and (\ref{eq:gen2}) that take into account the homology classes involved in the constructions in Lemmas \ref{lem:newsum} and \ref{lem:sum2}. In particular, the requirement that $c$ is primitive is imposed because a closed geodesic in a cylinder is a simple closed curve and thus represents a primitive homology class. By Lemmas \ref{lem:Czwequal1} and \ref{lem:Czwequal2}, any two elements of $\Lam_{(0,1)}$ that are equivalent with respect to $\sim_\Lam$ determine the same component of $\cC(\phi)$. We summarize this discussion with the following lemma.

\begin{lem} \label{lem:connalg}
If $(a,b) \sim_\Lam (a^\pr,b^\pr)$ for all $(a,b),(a^\pr,b^\pr) \in \Lam_{(0,1)}$, then $\cC(\phi)$ is connected.
\end{lem}

In order to show that any two elements of $\Lam_{(0,1)}$ are equivalent with respect to $\sim_\Lam$, it will be useful to understand which submodules of $\Lam$ are dense in $\C$, and to know that the primitive elements of $\Lam$ are dense in $\C$ in a strong sense.

\begin{lem} \label{lem:rank3dense}
Fix $g \geq 3$. Let $\Lam \subset \C$ be a polarized module of rank $2g$ such that $\Lam \cap \R z \subset \Q z$ for all $z \in \Lam$. Every submodule of $\Lam$ of rank at least $3$ is dense in $\C$.
\end{lem}

\begin{proof}
It is enough to consider submodules of rank $3$. Let $V \subset \Lam$ be a submodule of rank $3$, and write $V = V_0 + \Z z_0$ where $V_0$ has rank $2$. Since $V \cap \R z \subset \Q z$ for all $z \in V$, the submodule $V_0$ is a lattice in $\C$, and $z_0 \notin \bigcup_{v \in V_0} \R v$. If $V$ is not dense in $\C$, then its closure is given by $M \cdot (\R + i\Z)$ with $M \in \SL(2,\R)$. Let $n \in \Z$ be such that $z_0 \in M \cdot (\R + in)$. Since $V_0 \subset M \cdot (\R + i\Z)$ is a lattice, there is a nonzero $v \in V_0 \cap (M \cdot \R)$, and there is a nonzero $m \in \Z$ such that there exists $v_0 \in V_0 \cap (M \cdot (\R + im))$. Then $mz_0 - nv_0 \in \R v$, but since $V_0 \cap \Z z_0 = \{0\}$, we have $mz_0 - nv_0 \notin \Q v$, a contradiction. Thus, $V$ is dense in $\C$.
\end{proof}

\begin{lem} \label{lem:primdense}
Let $\Lam \subset \C$ be a polarized module of rank $2g$, and let $P \subset \Lam$ be the subset of primitive elements. Let $V \subset \Lam$ be a primitive submodule, and fix $a \in V$. If $V$ is not a discrete subset of $\C$, then $P \cap (P - a) \cap V$ is dense in the closure of $V$ in $\C$.
\end{lem}

\begin{proof}
Since $V$ is not a discrete subset of $\C$, we must have $g \geq 2$. Since $V$ is a primitive submodule of $\Lam$, an element $v \in V$ is a primitive element of $\Lam$ if and only if there does not exist $m \in \Z \sm \{\pm 1\}$ and $v^\pr \in V$ such that $v = mv^\pr$. Since $V$ is not a discrete subset of $\C$, either $V$ is dense in $\C$ or the closure of $V$ has the form $M \cdot \R$ or $M \cdot (\R + i\Z)$ with $M \in \SL(2,\R)$. We will consider cases according to the size of the closure of $V$. We may assume that $a \neq 0$, since that case implies the case where $a = 0$. \\

\paragraph{\bf Case 1:} Suppose the closure of $V$ has the form $M \cdot \R$ with $M \in \SL(2,\R)$. We may assume that $V$ is dense in $\R$. Since $V$ has rank at least $2$, there is $z \in V$ such that $a / z \notin \Q$. Then
\be
W = \left\{v \in V : nv \in \Z a + \Z z \text{ for some nonzero } n \in \Z\right\}
\ee
is a primitive submodule of rank $2$ that contains $a$ and is dense in $\R$. Write $a = kz_1$ with $k \in \Z$ and $z_1 \in W$ primitive, and extend to a basis $z_1,z_2$ for $W$. By rescaling by a nonzero real number, we may assume that $z_1 = 1$. By replacing $z_2$ with an element of $z_2 + \Z$, we may assume that $0 < z_2 < 1$.

Fix $m \in \Z$. For each prime number $p$, let $n_p$ be the unique integer such that $m < -n_p + pz_2 < m+1$. Since $0 < z_2 < 1$, for all but finitely many primes $p$, we have $0 < n_p < p - k$ which implies $\gcd(n_p, p) = 1$ and $\gcd(n_p + k, p) = 1$. Thus, $-n_p + pz_2 \in W$ and $(-n_p + pz_2) + a \in W$ are primitive for all but finitely many primes $p$. Denote by
\be
p_1 = 2, p_2 = 3, p_3 = 5, \dots
\ee
the sequence of prime numbers. A theorem of Vinogradov tells us that for any $\al \in \R \sm \Q$, the sequence $\{p_j \al\}_{j=1}^\infty$ equidistributes in $\R / \Z$ with respect to the Lebesgue measure \cite{Vin:prime}. In particular, these sequences are dense in $\R / \Z$, and therefore
\be
\{-n_p + p z_2 : p \text{ is prime and sufficiently large}\} \subset P \cap (P - a) \cap W
\ee
is a dense subset of the interval $(m,m+1)$. Since $m$ was an arbitrary integer, we conclude that $P \cap (P - a) \cap V$ is dense in $\R$. \\

\paragraph{\bf Case 2:} Suppose the closure of $V$ has the form $M \cdot (\R + i\Z)$ with $M \in \SL(2,\R)$. We may assume that $V$ is dense in $\R + i\Z$. In this case, $V$ has rank at least $3$ and $V \cap \R$ has rank at least $2$. Then since $V \cap (\R + i)$ is dense in $\R + i$, there is $z \in V \cap (\R + i)$ such that $a \notin \Z z$. Choose $\ell \in \Z$ such that $a - \ell z \in \R$, and write $a - \ell z = k_1 w$ with $k_1 \in \Z$ and $w \in V \cap \R$ primitive. Since $V \cap \R$ has rank at least $2$, there is $u \in V \cap \R$ such that $u/w \notin \Q$. Then
\be
W = \left\{v \in V : nv \in \Z z + \Z w + \Z u \text{ for some nonzero } n \in \Z \right\}
\ee
is a primitive submodule of rank $3$ that contains $a$ and is dense in $\R + i\Z$. Now write $a - \ell z = kz_1$ with $k \in \Z$ and $z_1 \in W$ primitive, and extend to a basis $z_1,z_2$ for $W \cap \R$. We can extend $z_1,z_2$ to a basis $z_1,z_2,z_3$ for $W$ by choosing $z_3 \in W \cap (\R + i)$, so let $z_3 = z$. By rescaling real parts by a nonzero real number, we may assume that $z_1 = 1$, while preserving the density of $W$ in $\R + i\Z$. By replacing $z_2$ with an element of $z_2 + \Z$, we may assume that $0 < z_2 < 1$.

Fix $m, n \in \Z$. For each prime number $p$, let $n_p$ be the unique integer such that $m < -n_p + pz_2 < m + 1$. Then $-n_p + pz_2 + n z_3$ lies in the translated interval $(0,1) + m + nz_3$, and since $z_3 \in \R + i$, the union of these translated intervals (over all $m,n \in \Z$) is dense in $\R + i\Z$. Recall that $a = k z_1 + \ell z_3$. Since $0 < z_2 < 1$, for all but finitely many primes $p$, we have $0 < n_p < p - k$ which implies $\gcd(n_p,p,n) = 1$ and $\gcd(n_p + k, p, n + \ell) = 1$. Thus, $-n_p + pz_2 + nz_3 \in W$ and $(-n_p + pz_2 + nz_3) + a \in W$ are primitive for all but finitely many primes $p$. As in Case 1, since $z_2 \notin \Q$, Vinogradov's theorem implies $\{p_j z_2\}_{j=1}^\infty$ is dense in $\R / \Z$, and therefore
\be
\{-n_p + pz_2 + nz_3 : p \text{ is prime and sufficiently large}\} \subset P \cap (P - a) \cap W
\ee
is a dense subset of $(0,1) + m + nz_3$. Thus, since $m$ and $n$ were arbitrary integers, $P \cap (P - a) \cap V$ is dense in $\R + i\Z$. \\

\paragraph{\bf Case 3:} Lastly, suppose $V$ is dense in $\C$. In this case, $V$ has rank at least $3$. The set of $z \in V$ such that $z \notin \R a$ is dense in $\C$. Fix $z \in V$ such that $z \notin \R a$, and fix $w \in V$ such that $\Z a + \Z z + \Z w$ has rank $3$. Then
\be
W = \{v \in V : nv \in \Z a + \Z z + \Z w \text{ for some nonzero } n \in \Z \}
\ee
is a primitive submodule of rank $3$ that contains $a$. Since $z$ is an arbitrary element of $V \sm \R a$, the union of the rank $3$ primitive submodules containing $a$ is dense in $\C$, so it is enough to show that $P \cap (P - a) \cap W$ is dense in the closure of $W$ in $\C$. If $W$ is not dense in $\C$, then since $W$ contains a lattice and is not discrete, $W$ is dense in $M \cdot (\R + i\Z)$ for some $M \in \SL(2,\R)$. Case 2 then tells us that $P \cap (P - a) \cap W$ is dense in $M \cdot (\R + i\Z)$. Thus, we may assume that $W$ is dense in $\C$. Write $a = kz_1$ with $k \in \Z$ and $z_1 \in W$ primitive, and extend to a basis $z_1,z_2,z_3$ for $W$. Since $W$ is dense in $\C$, it must be that $z_2 \notin \R z_1$. Then $W_0 = \Z z_1 + \Z z_2$ is a lattice in $\C$. By applying an element of $\GL(2,\R)$, we may assume that $z_1 = 1$ and $z_2 = i$. Then, by replacing $z_3$ with an element of $z_3 + \Z + \Z i$, we may assume that $z_3$ lies in the open unit square $(0,1) + (0,1)i$. Write $z_3 = x +iy$ with $0 < x,y < 1$.

Fix $m,n \in \Z$. For each prime number $p$, let $m_p,n_p$ be the unique integers such that $-m_p - n_p i + p z_3$ lies in the square $(0,1) + (0,1)i + m + ni$. The union of these squares is dense in $\C$. Since $0 < x,y < 1$, for all but finitely many primes $p$, we have $0 < m_p,n_p < p - k$ which implies $\gcd(m_p,n_p,p) = 1$ and $\gcd(m_p + k,n_p,p) = 1$. Thus, $-m_p - n_p i + pz_3 \in W$ and $(-m_p - n_p i + pz_3) + a \in W$ are primitive for all but finitely many primes $p$. Since $W$ is dense in $\C$, the projection of $\Z z_3$ to the torus $\C / (\Z + \Z i)$ is dense, which is equivalent to $1,x,y$ being $\Q$-linearly independent. The set
\be
\{-m_p - n_p i + p z_3 : p \text{ is prime and sufficiently large}\} \subset P \cap (P - a) \cap W
\ee
is dense in $(0,1) + (0,1)i + m + ni$ if and only if $\{p_j z_3\}_{j=1}^\infty$ projects to a dense subset of $\C / (\Z + \Z i)$.

As in Case 1, the sequence $\{p_j z_3\}_{j=1}^\infty$ equidistributes in $\C / (\Z + \Z i)$ with respect to the Lebesgue measure. To see this, we identify $\C / (\Z + \Z i)$ with $\R^2 / \Z^2$ and recall Weyl's equidistribution criterion, which in our case says that a sequence $\{(s_j,t_j)\}_{j=1}^\infty$ in $\R^2 / \Z^2$ equidistributes if and only if for all nonzero $(k_1,k_2) \in \Z^2$, we have $\frac{1}{N}\sum_{j=1}^N \exp(2\pi i (k_1 s_j + k_2 t_j)) \ra 0$ as $N \ra \infty$. Equivalently, the sequences $\{k_1 s_j + k_2 t_j\}_{j=1}^\infty$ equidistribute in $\R/\Z$. Since $1,x,y$ are $\Q$-linearly independent, Vinogradov's equidistribution theorem implies that for all nonzero $(k_1,k_2) \in \Z^2$, the sequence $\{p_j (k_1 x + k_2 y)\}_{j=1}^\infty$ equidistributes in $\R / \Z$. Thus, $\{p_j z_3\}_{j=1}^\infty$ equidistributes in $\C / (\Z + \Z i)$. Since $m$ and $n$ were arbitrary integers, we are done.
\end{proof}

\begin{lem} \label{lem:relndense}
Fix $g \geq 3$. Let $\Lam \subset \C$ be a polarized module of rank $2g$ such that $\Lam \cap \R z \subset \Q z$ for all $z \in \Lam$. Then every equivalence class for $\sim_\Lam$ is dense in $\cT_{(0,1)}$.
\end{lem}

\begin{proof}
Fix $(z,w) \in \Lam_{(0,1)}$. Recall the equivalence relation $\sim$ on $\cT_{(0,1)}$ defined in Section 4. By Lemma \ref{lem:reln}, for all $(z^\pr,w^\pr) \in \cT_{(0,1)}$, we have $(z,w) \sim (z^\pr,w^\pr)$. This means there is a sequence of pairs
\be
(z,w) = (z_1,w_1), (z_2,w_2), \dots, (z_N, w_N) = (z^\pr,w^\pr) \in \cT_{(0,1)}
\ee
such that for $1 \leq j \leq N-1$, either $(z_{j+1},w_{j+1}) = (z_j, n_j z_j + w_j)$ for some $n_j \in \Z$, or $(z_{j+1},w_{j+1}) = (-z_{1,j} - w_j, z_j + w_j)$ for some $z_{1,j} \in \C$ such that $z_j,w_j,z_{1,j}$ satisfy the inequalities in (\ref{eq:gen}).

If $(z_2,w_2) = (z_1, n_1 z_1 + w_1)$ for some $n_1 \in \Z$, then $(z_2,w_2) \in \Lam_{(0,1)}$ as well. Otherwise, $(z_2,w_2) = (-z_{1,1} - w_1, z_1 + w_1)$ for some $z_{1,1} \in \C$ such that $z_1,w_1,z_{1,1}$ satisfy the inequalities in (\ref{eq:gen}). We have $z_1 + w_1 \in \Lam$, and since the inequalities in (\ref{eq:gen3alg}) are open conditions, we can replace $z_{1,1}$ with any sufficiently close primitive element of $\{z_1,w_1\}^\perp$. Lemmas \ref{lem:rank3dense} and \ref{lem:primdense} ensure that primitive elements in $\{z_1,w_1\}^\perp$ are dense in $\C$, so there exist primitive elements $z_{1,1}^\pr \in \{z_1,w_1\}^\perp$ arbitrarily close to $z_{1,1}$. Letting $(z_2^\pr,w_2^\pr) = (-z_{1,1}^\pr - w_1, z_1 + w_1)$, we have $(z_1,w_1) \sim_\Lam (z_2^\pr,w_2^\pr)$ and $(z_2^\pr,w_2^\pr)$ is close to $(z_2,w_2)$.

We can iterate this argument with $(z_j,w_j)$ for $3 \leq j \leq N$ to obtain a nearby element $(z_j^\pr,w_j^\pr) \in \Lam_{(0,1)}$ such that $(z_{j-1}^\pr,w_{j-1}^\pr) \sim_\Lam (z_j^\pr,w_j^\pr)$. Thus, the equivalence class of $(z,w)$ for $\sim_\Lam$ is dense in $\cT_{(0,1)}$.
\end{proof}

With our various density properties established, we are now ready to prove that $\Lam_{(0,1)}$ consists of a single equivalence class with respect to $\sim_\Lam$.

\begin{lem} \label{lem:relnalg}
Fix $g \geq 3$. Let $\Lam \subset \C$ be a polarized module of rank $2g$ such that $\Lam \cap \R z \subset \Q z$ for all $z \in \Lam$. Then for all $(a,b), (c,d) \in \Lam_{(0,1)}$, we have $(a,b) \sim_\Lam (c,d)$.
\end{lem}

\begin{proof}
By Lemma \ref{lem:rank3dense}, every submodule of $\Lam$ of rank at least $3$ is dense in $\C$. By Lemma \ref{lem:relndense}, every equivalence class for $\sim_\Lam$ is dense in $\cT_{(0,1)}$, so it is enough to show that $(a,b) \sim_\Lam (c,d)$ for all $(a,b) \in \Lam_{(0,1)}$ and $(c,d) \in \Lam_{(0,1)}$ sufficiently close to $\left(1/2,i/2\right)$. Fix $\eps > 0$ small, and fix $(a,b) \in \Lam_{(0,1)}$ such that $|a - 1/2| < \eps$ and $|b - i/2| < \eps$.

Applying the relations in (\ref{eq:gen2alg}) twice, we see that for all primitive $a_1 \in \{a,b\}^\perp$ such that $|a_1 - (a+b)| < 4\eps$ and all primitive $a_2 \in \{-a_1 - b, a + b\}^\perp$ such that $|a_2 + b| < 4\eps$,
\be
(a,b) \sim_\Lam (-a_1 - b, a + b) \sim_\Lam (-a_2 - a - b, -a_1 + a) .
\ee
We will bootstrap from this observation in several steps. \\

\paragraph{\bf Step 1a.} Fix $b_1 \in b^\perp$ such that $a \cdot b_1 = 1$ and $|b_1 - i/2| < \eps$. We will show that $(a,b) \sim_\Lam (a,b_1)$. The submodule $\{a,b,b_1\}^\perp$ is primitive and has rank at least $2g-3 \geq 3$. Then since $|b - b_1| < 2\eps$, by Lemmas \ref{lem:rank3dense} and \ref{lem:primdense} there exists a primitive $a_1 \in \{a,b,b_1\}^\perp$ such that $|a_1 - (a+b)| < 2\eps$ and $|a_1 - (a+b_1)| < 2\eps$. The relation
\begin{equation} \label{eq:relnalg1}
(-a_1 - b) + (a + b) = (-a_1 - b_1) + (a + b_1)
\end{equation}
implies that the primitive submodule $V_1 = \{-a_1 - b, a + b, -a_1 - b_1, a + b_1\}^\perp$ has rank at least $2g - 3 \geq 3$. Since $b_1 - b \in V_1$, by Lemmas \ref{lem:rank3dense} and \ref{lem:primdense} there exists $a_2 \in V_1$ such that $|a_2 + b| < 2\eps$ and such that both $a_2$ and $a_2^\pr = a_2 + b - b_1$ are primitive. We have $a_2^\pr \in V_1$ and $|a_2^\pr + b| < 4\eps$. Thus, since $V_1 = \{-a_1-b,a+b\}^\perp \cap \{-a_1-b_1,a+b_1\}^\perp$, we have
\be
(a,b) \sim_\Lam (-a_2 - a - b, -a_1 + a) = (-a_2^\pr - a - b_1, -a_1 + a) \sim_\Lam (a,b_1) .
\ee

\paragraph{\bf Step 1b.} Fix $b_3 \in \Lam$ such that $a \cdot b_3 = 1$ and $|b_3 - i/2| < \eps$. We will find $b_1,b_2 \in \Lam$ such that $b_1 \in b + \{a,b\}^\perp$, $b_2 \in b_1 + \{a,b_1\}^\perp$, $b_3 \in b_2 + \{a,b_2\}^\perp$, and such that $|b_1 - i/2| < \eps$, $|b_2 - i/2| < \eps$. Step 1a then implies $(a,b) \sim_\Lam (a,b_1) \sim_\Lam (a,b_2) \sim_\Lam (a,b_3)$.

Write $b_3 = k a + b + c_3$ with $k \in \Z$ and $c_3 \in \{a,b\}^\perp$. We may assume $k \neq 0$, since otherwise $b_3 \in b + \{a,b\}^\perp$ and then $(a,b) \sim_\Lam (a,b_3)$ by Step 1a. Since $\{a,b,c_3\}^\perp$ is a primitive submodule of rank at least $2g - 3 \geq 3$, by Lemmas \ref{lem:rank3dense} and \ref{lem:primdense} there is a primitive $c_2 \in \{a,b,c_3\}^\perp$ such that $|ka + b + c_2 - i/2| < \eps$. Letting $b_2 = ka + b + c_2$, we have $|b_2 - i/2| < \eps$, $a \cdot b_2 = 1$, and $b_3 \cdot b_2 = c_3 \cdot c_2 = 0$, so $b_3 \in b_2 + \{a,b_2\}^\perp$. Since $c_2 \in \{a,b\}^\perp$ is primitive, there is $d_2 \in \{a,b\}^\perp$ such that $c_2 \cdot d_2 = 1$. Since $d_2 + \{a,b,c_2\}^\perp$ is a translate of a submodule of rank at least $2g - 3 \geq 3$, by Lemma \ref{lem:rank3dense} there is $e \in \{a,b,c_2\}^\perp$ such that $|b - kd_2 + e - i/2| < \eps$. Let $c_1 = -k d_2 + e$, and let $b_1 = b + c_1$, so $|b_1 - i/2| < \eps$. We have $a \cdot b_1 = 1$ and $c_1 \in \{a,b\}^\perp$, so $b_1 \in b + \{a,b\}^\perp$. Lastly, since $b_2 \cdot b_1 = k + c_2 \cdot c_1 = 0$, we have $b_2 \in b_1 + \{a,b_1\}^\perp$. Thus, by applying Step 1a three times, we get $(a,b) \sim_\Lam (a,b_3)$. \\

\paragraph{\bf Step 2a.} Fix $a_1 \in a^\perp$ such that $a_1 \cdot b = 1$ and $|a_1 - 1/2| < \eps$. An argument similar to Step 1a will show that $(a,b) \sim_\Lam (a_1,b)$. Since the primitive submodule $\{a,b,a_1\}^\perp$ has rank at least $2g - 3 \geq 3$, and since $a_1 - a \in \{a,b,a_1\}^\perp$, by Lemmas \ref{lem:rank3dense} and \ref{lem:primdense} there exists $a_1^\pr \in \{a,b,a_1\}^\perp$ such that $|a_1^\pr - (a+b)| < 2\eps$ and such that both $a_1^\pr$ and $a_1^{\pr\pr} = a_1^\pr + a_1 - a$ are primitive. We have $a_1^{\pr\pr} \in \{a,b,a_1\}^\perp$ and $|a_1^{\pr\pr} - (a+b)| < 4\eps$. The relations
\begin{equation} \label{eq:relnalg2}
(-a_1^\pr-b) + (a+b) = -a_1^\pr + a = -a_1^{\pr\pr} + a_1 = (-a_1^{\pr\pr}-b) + (a_1+b)
\end{equation}
imply the primitive submodule $V_2 = \{-a_1^\pr-b,a+b,-a_1^{\pr\pr}-b,a_1+b\}^\perp$ has rank at least $2g-3 \geq 3$. Since $a_1 - a \in V_2$, by Lemmas \ref{lem:rank3dense} and \ref{lem:primdense} there exists $a_2 \in V_2$ such that $|a_2 + b| < 2\eps$ and such that both $a_2$ and $a_2^\pr = a_2 + a - a_1$ are primitive. We have $a_2^\pr \in V_2$ and $|a_2^\pr + b| < 4\eps$. Thus, since $V_2 = \{-a_1^\pr-b,a+b\}^\perp \cap \{-a_1^{\pr\pr}-b,a_1+b\}^\perp$, we have
\be
(a,b) \sim_\Lam (-a_2 - a - b, -a_1^\pr + a) = (-a_2^\pr - a_1 - b, -a_1^{\pr\pr} + a_1) \sim_\Lam (a_1,b) .
\ee

\paragraph{\bf Step 2b.} Fix $a_3 \in \Lam$ such that $a_3 \cdot b = 1$ and $|a_3 - 1/2| < \eps$. This step is the same as Step 1b with the roles of $a$ and $b$ exchanged. Write $a_3 = a + kb + c_3$ with $k \in \Z$ and $c_3 \in \{a,b\}^\perp$. We assume $k \neq 0$, otherwise we are done by Step 2a. Using Lemmas \ref{lem:rank3dense} and \ref{lem:primdense}, choose a primitive $c_2 \in \{a,b,c_3\}^\perp$ such that $a_2 = a + kb + c_2$ satisfies $|a_2 - 1/2| < \eps$. Choose $d_2 \in \{a,b\}^\perp$ such that $c_2 \cdot d_2 = 1$, and use Lemma \ref{lem:rank3dense} to choose $e \in \{a,b,c_2\}^\perp$ such that $a_1 = a + k(d_2 + e)$ satisfies $|a_1 - 1/2| < \eps$. Then $a_1 \in a + \{a,b\}^\perp$, $a_2 \in a_1 + \{a_1,b\}^\perp$, and $a_3 \in a_2 + \{a_2,b\}^\perp$, and thus by Step 2a, $(a,b) \sim_\Lam (a_3,b)$. \\

\paragraph{\bf Step 3.} In the previous steps, we showed that $(a,b) \sim_\Lam (a,b^\pr)$ for all $b^\pr \in b + a^\perp$ with $|b^\pr - i/2| < \eps$, and that $(a,b) \sim_\Lam (a^\pr,b)$ for all $a^\pr \in a + b^\perp$ with $|a^\pr - 1/2| < \eps$.

To conclude, fix $(c,d) \in \Lam_{(0,1)}$ such that $|c - 1/2| < \eps$ and $|d - i/2| < \eps$. Choose $b_1 \in \Lam$ such that $a \cdot b_1 = 1$. Since $a \cdot b_1 = 1$, there are $m,n \in \Z$ such that $d - m a - n b_1 \in \{a,b_1\}^\perp$. Write $d - m a - n b_1 = p a_2$ with $p \in \Z$ and $a_2 \in \{a,b_1\}^\perp$ primitive, and choose $b_2 \in \{a,b_1\}^\perp$ such that $a_2 \cdot b_2 = 1$. Since $d$ is primitive, $\gcd(m,n,p) = 1$, so there are integers $k_1,k_2,k_3$ such that $-k_1 m + k_2 n - k_3 p = n + 1$. Letting $b^\pr = b_1 + b_2$, and letting $c^\pr = (k_2 - 1) a + k_1 b_1 - (k_2 - 2) a_2 + k_3 b_2$, we have $a \cdot b^\pr = c^\pr \cdot b^\pr = c^\pr \cdot d = 1$. Lastly, by Lemma \ref{lem:rank3dense} there is $b^{\pr\pr} \in b^\pr + \{a,c^\pr\}^\perp$ such that $|b^{\pr\pr} - i/2| < \eps$. Similarly, there is $c^{\pr\pr} \in c^\pr + \{b^{\pr\pr},d\}^\perp$ such that $|c^{\pr\pr} - 1/2| < \eps$. We then have $a \cdot b^{\pr\pr} = c^{\pr\pr} \cdot b^{\pr\pr} = c^{\pr\pr} \cdot d = 1$, so $(a,b^{\pr\pr}), (c^{\pr\pr},b^{\pr\pr}),(c^{\pr\pr},d) \in \Lam_{(0,1)}$, and thus
\be
(a,b) \sim_\Lam (a,b^{\pr\pr}) \sim_\Lam (c^{\pr\pr},b^{\pr\pr}) \sim_\Lam (c^{\pr\pr},d) \sim_\Lam (c,d) .
\ee
\end{proof}

We have now established Theorem \ref{thm:conn} in the case of strata with two zeros.

\begin{thm} \label{thm:connsum}
Fix $g \geq 3$. Let $\cC$ be the nonhyperelliptic component of a stratum $\Om\cM_g(m_1,m_2)$ with $m_1,m_2$ odd. If $\phi \in H^1(S_g;\C)$ is a positive cohomology class such that $\Per(\phi) \cong \Z^{2g}$ and $\Per(\phi) \cap \R z \subset \Q z$ for all $z \in \Per(\phi)$, then $\cC(\phi)$ is connected.
\end{thm}

\begin{proof}
By scaling by a positive real number, we may assume that $\lr{\phi,\phi} = 1$. Let $\Lam$ be the polarized module $\Per(\phi)$. By Lemma \ref{lem:relnalg}, any two elements of $\Lam_{(0,1)}$ are equivalent with respect to $\sim_\Lam$. Therefore, $\cC(\phi)$ is connected by Lemma \ref{lem:connalg}.
\end{proof}

\begin{thm} \label{thm:connsplit}
Fix $g \geq 3$. Let $\cC$ be a connected stratum $\Om\cM_g(\kap)$ with $|\kap| > 1$, and suppose that $m \geq 2$ for some $m \in \kap$. Fix $1 \leq j < m$, let $\kap^\pr = (\kap \sm (m)) \cup (m-j,j)$, and let $\cC^\pr = \Om\cM_g(\kap^\pr)$. Let $\phi \in H^1(S_g;\C)$ be a positive cohomology class such that $\Per(\phi) \cong \Z^{2g}$ and $\Per(\phi) \cap \R z \subset \Q z$ for all $z \in \Per(\phi)$. If $\cC(\phi)$ is connected, then $\cC^\pr(\phi)$ is connected.
\end{thm}

\begin{proof}
By Corollary \ref{cor:KZconn}, since $\cC$ is connected, $\cC^\pr$ is connected. Fix $(X_1,\om_1), (X_2,\om_2) \in \cC^\pr(\phi)$. By Lemma \ref{lem:generic}, after replacing $(X_j,\om_j)$ with a nearby holomorphic $1$-form in $L(\om_j)$, we may assume that the $\GL^+(2,\R)$-orbit of $(X_j,\om_j)$ is dense in $\cC^\pr$. The image of the zero splitting map $\Phi : \cS(\kap;m) \ra \cC^\pr$ is nonempty, open, and $\GL^+(2,\R)$-invariant, and therefore dense. Since splitting zeros does not change the absolute periods, we can write $(X_j,\om_j) = \Phi(X_j^\pr,\wt{\om}_j^\pr,\gam_j)$ with $(X_j^\pr,\om_j^\pr) \in \cC(\phi)$. By assumption, $\cC(\phi)$ is connected, so $(X_1^\pr,\om_1^\pr)$ and $(X_2^\pr,\om_2^\pr)$ lie on the same leaf of $\cA(\kap)$. By Lemma \ref{lem:generic}, the leaves $L(\om_j^\pr)$ contain holomorphic $1$-forms whose $\GL^+(2,\R)$-orbits are dense in $\cC$, so by Lemma \ref{lem:coverleaf}, $(X_1^\pr,\wt{\om}_1^\pr)$ and $(X_2^\pr,\wt{\om}_2^\pr)$ lie on the same leaf of $\cA(\kap;m)$. Then by Lemma \ref{lem:splitleaf}, $(X_1^\pr,\wt{\om}_1^\pr,\gam_1)$ and $(X_2^\pr,\wt{\om}_2^\pr,\gam_2)$ lie on the same leaf of $\cF_\cS$. Since $\Phi$ maps leaves of $\cF_\cS$ into leaves of $\cA(\kap^\pr)$, we conclude that $(X_1,\om_1)$ and $(X_2,\om_2)$ lie on the same leaf of $\cA(\kap^\pr)$. Thus, $\cC^\pr(\phi)$ is connected.
\end{proof}

We now complete the proof of Theorem \ref{thm:conn}.

\begin{proof} (of Theorem \ref{thm:conn}) Induct on $|\kap|$, using Theorem \ref{thm:connsum} for the base case $|\kap| = 2$, and using Lemma \ref{lem:split} and Theorem \ref{thm:connsplit} for the inductive step.
\end{proof}

\paragraph{\bf Absolute periods in $\R + i\Z$.} We now prove Theorem \ref{thm:connR+iZ}, following a similar approach to Theorem \ref{thm:conn}. Fix $g \geq 3$, let $\phi \in H^1(S_g;\C)$ be a positive cohomology class with $\Per(\phi) \cong \Z^{2g}$ not dense in $\C$, and let $\cC$ be a nonhyperelliptic component of a stratum $\Om\cM_g(m_1,m_2)$ with $m_1 \geq m_2$ odd. The closure of $\Per(\phi)$ in $\C$ has the form $M \cdot (\R + i\Z)$ with $M \in \SL(2,\R)$. Fix $z,w \in \Per(\phi)$ such that $z \notin M \cdot \R$. Let $a,b \in H_1(S_g;\Z)$ be the unique homology classes such that $z = \phi(a)$ and $w = \phi(b)$. Suppose that $a \cdot b = 1$ and that $0 < \im(\ol{z} w) < \lr{\phi,\phi}$. If $m_2 \geq 3$, let $m_\cC = m_2$, and otherwise, let $m_\cC = m_1$. Define
\be
\cC(\phi,z,w) = \left\{ (X,\om) \in \cC(\phi) : \begin{matrix} (X,\om) \text{ has a splitting at a zero of $\om$ of order $m_\cC$} \\ \text{with associated periods } (z,w) \end{matrix} \right\} .
\ee
If $m_\cC = m_2$, let $\cC^\pr = \Om\cM_{g-1}(m_1,m_2 - 2)$, and if $m_\cC = m_1$, let $\cC^\pr = \Om\cM_{g-1}(m_1-2,m_2)$, so $\cC^\pr$ is connected by Corollary \ref{cor:KZconn}. Let $\phi^\pr \in H^1(S_{g-1};\C)$ be a positive cohomology class such that $\Per(\phi^\pr) = \phi(\{a,b\}^\perp)$ as a polarized module. Since $g \geq 3$ and $\phi^\pr$ is positive, $\Per(\phi^\pr) \cong \Z^{2(g-1)}$ is dense in $M \cdot (\R + in\Z)$ for some $n \in \Z_{>0}$.

\begin{lem} \label{lem:CzwconnR+iZ}
Suppose that $\cC^\pr(\phi^\pr)$ is connected. Then $\cC(\phi,z,w)$ is connected.
\end{lem}

\begin{proof}
The proof is the same as that of Lemma \ref{lem:Czwconn}, with Lemmas \ref{lem:preservedR+iZ} and \ref{lem:genericR+iZ} in place of Lemmas \ref{lem:preserved} and \ref{lem:generic}, respectively.
\end{proof}

\begin{lem} \label{lem:CzwunionR+iZ}
Suppose that Theorem \ref{thm:connR+iZ} is true for $\cC^\pr$. Then every component of $\cC(\phi)$ contains $\cC(\phi,z^\pr,w^\pr)$ for some $z^\pr,w^\pr \in \C$.
\end{lem}

\begin{proof}
The proof is the same as that of Lemma \ref{lem:Czwunion}, with Lemmas \ref{lem:preservedR+iZ}, \ref{lem:genericR+iZ}, and \ref{lem:CzwconnR+iZ} in place of Lemmas \ref{lem:preserved}, \ref{lem:generic}, and \ref{lem:Czwconn}, respectively.
\end{proof}

\begin{lem} \label{lem:Czwequal1R+iZ}
For all $n \in \Z$, we have $\cC(\phi,z,w) = \cC(\phi,z,nz + w)$.
\end{lem}

\begin{proof}
This is immediate from the definitions, as in Lemma \ref{lem:Czwequal1}.
\end{proof}

\begin{lem} \label{lem:Czwequal2R+iZ}
Let $a_1 \in \{a,b\}^\perp$ be a primitive homology class, and let $z_1 = \phi(a_1)$. Suppose that $z,w,z_1$ satisfy the inequalities in (\ref{eq:gen}). If $z_1 \in M \cdot \R$, suppose additionally that there is $b_1 \in \{a,b\}^\perp$ such that $a_1 \cdot b_1 = 1$ and such that $w_1 = \phi(b_1)$ satisfies $0 < \im(\ol{z}z_1) < \im(\ol{z}_1 w_1) < 1 - \im(\ol{z}w)$. Then $\cC(\phi,z,w) \cap \cC(\phi,-z_1-w,z+w)$ is nonempty.
\end{lem}

\begin{proof}
We may assume $\Per(\phi)$ is dense in $\R + i\Z$. The construction in the proof of Lemma \ref{lem:sum2R+iZ} provides holomorphic $1$-forms $(X,\om) \in \cC(\phi)$ with two splittings whose associated periods are $(z,w)$ and $(-z_1-w,z+w)$, respectively, as follows.

Recall that we first glue a pair of flat tori $T_1 = (\C / (\Z z_1 + \Z w_1), dz)$ and $T_2 = (\C / (\Z z_2 + \Z w_2), dz)$ along a pair of homologous saddle connections, and then iteratively form connected sums, first using the flat torus $T_0 = (\C / (\Z z + \Z w), dz)$ and a segment in $T_1$ with holonomy $z$, and then using flat tori $T_j = (\C / (\Z z_j + \Z w_j), dz)$ and short horizontal segments on $T_2$ with holonomy $z_j$ for $3 \leq j \leq g - 1$.

The periods $z,w,z_1$ are given to us. We need to choose $w_1 \in \Per(\phi)$ such that the homology class $b_1$ determined by $w_1 = \phi(b_1)$ satisfies $b_1 \in \{a,b\}^\perp$, $a_1 \cdot b_1 = 1$, and $0 < \im(\ol{z}z_1) < \im(\ol{z}_1 w_1) < 1 - \im(\ol{z} w)$. If $z_1 \in \R$, then the desired $w_1$ exists by assumption, so suppose $z_1 \notin \R$. Fix $b_1^\pr \in \{a,b\}^\perp$ such that $a_1 \cdot b_1^\pr = 1$. Since $z_1 \notin \R$, we have $\Per(\phi) \cap \R z_1 \subset \Q z_1$. Then as in the proof of Lemma \ref{lem:Czwequal2}, the set
\be
V = \{\im(\ol{z}_1 w_1^\pr) : w_1^\pr \in \phi(b_1^\pr) + \phi(\{a,b,a_1,b_1^\pr\}^\perp) \}
\ee
is dense in $\R$. Thus, the desired $w_1 \in \Per(\phi)$ exists.

If $g = 3$, then once $z,w,z_1,w_1$ are fixed, we can choose $z_2,w_2 \in \Per(\phi)$ to be the periods of any pair of homology classes such that $\Z a_2 + \Z b_2 = \{a,b,a_1,b_1\}^\perp$ and $a_2 \cdot b_2 = 1$. If $g \geq 4$, then we can choose $z_2,w_2,\dots,z_{g-1},w_{g-1} \in \Per(\phi)$ to be the periods of any symplectic basis of $\{a,b,a_1,b_1\}^\perp$ such that $\im(\ol{z}_j w_j) > 0$ for $2 \leq j \leq g - 1$ and such that $z_3,\dots,z_{g-1} \in \R_{>0}$ are short (depending on $z_2,w_2$). Since the conditions on the real parts of the periods $z_2,w_2,\dots,z_{g-1},w_{g-1}$ are open conditions, as in the proof of Lemma \ref{lem:Czwequal2}, the $\R + i\Z$ case of Kapovich's classification of $\Sp(2(g-2),\Z)$-orbit closures from Lemma \ref{lem:Kapovich} ensures that such a choice of $z_2,w_2,\dots,z_{g-1},w_{g-1}$ is possible.
\end{proof}

Suppose $\Lam$ is a polarized module of rank $2g$ that is not dense in $\C$, such that for any symplectic basis $\{a_j,b_j\}_{j=1}^g$ of $\Lam$, we have $\sum_{j=1}^g \im(\ol{a}_j b_j) = 1$. The closure of $\Lam$ in $\C$ is given by $M \cdot (\R + i\Z)$ for some $M \in \SL(2,\R)$. Let $\Lam_0 = \Lam \cap (M \cdot \R)$. Define
\be
\Lam_{(0,1)} = \left\{(a,b) \in (\Lam \sm \Lam_0) \times \Lam : a \cdot b = 1, 0 < \im(\ol{a}b) < 1 \right\} .
\ee
Let $\sim_\Lam$ be an equivalence relation on $\Lam_{(0,1)}$ satisfying the following. We suppose that
\begin{equation} \label{eq:gen1algR+iZ}
(a,b) \sim_\Lam (a,na + b)
\end{equation}
for all $(a,b) \in \Lam_{(0,1)}$ and all $n \in \Z$. We suppose that
\begin{equation} \label{eq:gen2algR+iZ}
(a,b) \sim_\Lam (-c-b,a+b)
\end{equation}
for all $(a,b) \in \Lam_{(0,1)}$ and all primitive $c \in \{a,b\}^\perp \sm \Lam_0$ satisfying the inequalities in (\ref{eq:gen3alg}) and such that $-c-b \notin \Lam_0$. Lastly, we suppose that
\begin{equation} \label{eq:gen3algR+iZ}
(a,b) \sim_\Lam (-c-b,a+b)
\end{equation}
for all $(a,b) \in \Lam_{(0,1)}$ with $b \notin \Lam_0$ and all primitive $c \in \Lam_0 \cap \{a,b\}^\perp$ such that $a,b,c$ satisfy the inequalities in (\ref{eq:gen3alg}) and such that there is $d \in \{a,b\}^\perp$ satisfying $c \cdot d = 1$ and
\begin{equation} \label{eq:gen4algR+iZ}
0 < \im(\ol{a}c) < \im(\ol{c}d) < 1 - \im(\ol{a}b) .
\end{equation}
Let $\phi \in H^1(S_g;\C)$ be a positive cohomology class such that $\Per(\phi) = \Lam$ as a polarized module. As in the previous subsection, Lemmas \ref{lem:CzwconnR+iZ}-\ref{lem:Czwequal2R+iZ} reduce the connectivity of $\cC(\phi)$ to the following algebraic problem.

\begin{lem} \label{lem:connalgR+iZ}
If $(a,b) \sim_\Lam (a^\pr,b^\pr)$ for all $(a,b),(a^\pr,b^\pr) \in \Lam_{(0,1)}$, then $\cC(\phi)$ is connected.
\end{lem}

We will need the following analogous density results to Lemmas \ref{lem:rank3dense} and \ref{lem:relndense}. If $\Lam$ is dense in $M \cdot (\R + i\Z)$ with $M \in \SL(2,\R)$, then we have an inclusion
\be
J_\Lam : \Lam_{(0,1)} \ra \cT_{(0,1)}^{M \cdot (\R + i\Z)}, \quad J_\Lam(z,w) = (z,w,n) ,
\ee
where $n \in \Z_{>0}$ is the unique integer such that $\{z,w\}^\perp$ is dense in $M \cdot (\R + in\Z)$.

\begin{lem} \label{lem:rank3denseR+iZ}
Fix $g \geq 3$. Let $\Lam \subset \C$ be a polarized module of rank $2g$ such that $\Lam$ is not dense in $\C$, so $\Lam$ is dense in $M \cdot (\R + i\Z)$ for some $M \in \SL(2,\R)$. If $V \subset \Lam$ is a submodule of rank at least $3$, then $V$ is dense in $M \cdot (\R + in\Z)$ for some $n \in \Z_{\geq 0}$.
\end{lem}

\begin{proof}
It is enough to show that $V \cap (M \cdot \R)$ is dense in $M \cdot \R$. Since $V$ has rank at least $3$, $V \cap (M \cdot \R)$ has rank at least $2$, and therefore $V \cap (M \cdot \R)$ is dense in $M \cdot \R$.
\end{proof}

\begin{lem} \label{lem:relndenseR+iZ}
Fix $g \geq 3$. Let $\Lam \subset \C$ be a polarized module of rank $2g$ such that $\Lam$ is not dense in $\C$, so $\Lam$ is dense in $M \cdot (\R + i\Z)$ for some $M \in \SL(2,\R)$. The image of any equivalence class for $\sim_\Lam$ under $J_\Lam$ is dense in $\cT_{(0,1)}^{M \cdot (\R + i\Z)}$.
\end{lem}

\begin{proof}
The proof is similar to that of Lemma \ref{lem:relndense}. We may assume that $\Lam$ is dense in $\R + i\Z$. Fix $(z,w) \in \Lam_{(0,1)}$, so $J_\Lam(z,w) = (z,w,n) \in \cT_{(0,1)}^{\R + i\Z}$ for some $n \in \Z_{>0}$. For all $k \in \Z$, we have $(z,kz+w) \in \Lam_{(0,1)}$, and since $\{z,kz+w\}^\perp = \{z,w\}^\perp$, we have $J_\Lam(z,kz+w) = (z,kz+w,n)$. Fix $z_1 \in (\R + in\Z) \sm \R$ such that $\im(-z_1-w) \neq 0$ and such that $z,w,z_1$ satisfy the inequalities in (\ref{eq:gen}). Since $\{z,w\}^\perp$ is dense in $\R + in\Z$, by Lemma \ref{lem:primdense} there is a primitive $z_1^\pr \in \{z,w\}^\perp$ close to $z_1$. Then $z,w,z_1^\pr$ satisfy the inequalities in (\ref{eq:gen3alg}), and $-z_1^\pr - w \notin \Lam_0$, so we have $(z,w) \sim_\Lam (-z_1^\pr-w,z+w)$. Lastly, suppose $w \notin \Lam_0$, and fix $z_1 \in \R$ such that $z,w,z_1$ satisfy the inequalities in (\ref{eq:gen}) and such that there is $w_1 \in \R + in\Z$ satisfying the inequalities in (\ref{eq:gen4R+iZ}). By Lemma \ref{lem:primdense}, there is a primitive $w_1^\pr \in \{z,w\}^\perp$ close to $w_1$. Then there is $z_0^\pr \in \{z,w\}^\perp$ such that $z_0^\pr \cdot w_1^\pr = 1$. Moreover, $z_0^\pr + \{z,w,w_1^\pr\}^\perp$ is a translate of a submodule of rank $2g - 3 \geq 3$, so it intersects $\R$ in a dense subset of $\R$. Thus, there is $z_1^\pr \in (z_0^\pr + \{z,w,w_1^\pr\}^\perp) \cap \R$ close to $z_1$. Then $z_1^\pr \cdot w_1^\pr = 1$ and since the inequalities in (\ref{eq:gen4R+iZ}) are open conditions, $z_1^\pr$ and $w_1^\pr$ satisfy the inequalities in (\ref{eq:gen4algR+iZ}). Thus, $(z,w) \sim_\Lam (-z_1^\pr - w, z + w)$. Following the proof of Lemma \ref{lem:relndense} with Lemma \ref{lem:relnR+iZ} in place of Lemma \ref{lem:reln}, we conclude that the image of the equivalence class of $(z,w)$ for $\sim_\Lam$ under $J_\Lam$ is dense in $\cT_{(0,1)}^{\R + i\Z}$.
\end{proof}

\begin{lem} \label{lem:relnR+iZalg}
Fix $g \geq 3$, and let $\Lam \subset \C$ be a polarized module of rank $2g$ such that $\Lam$ is not dense in $\C$. For all $(a,b) \in \Lam_{(0,1)}$ and $(c,d) \in \Lam_{(0,1)}$, we have $(a,b) \sim_\Lam (c,d)$.
\end{lem}

\begin{proof}
We may assume that $\Lam$ is dense in $\R + i\Z$. By Lemma \ref{lem:rank3denseR+iZ}, any submodule of $\Lam$ of rank at least $3$ is dense in $\R + in\Z$ for some $n \in \Z_{\geq 0}$. By Lemma \ref{lem:relndenseR+iZ}, the image of any equivalence class for $\sim_\Lam$ under $J_\Lam$ is dense in $\cT_{(0,1)}^{\R + i\Z}$. Fix $\eps_1 > 0$ small, and fix $0 < \eps < \eps_1 / 100$. It is enough to show that $(a,b) \sim_\Lam (c,d)$ for all $(a,b) \in \Lam_{(0,1)}$ and $(c,d) \in \Lam_{(0,1)}$ sufficiently close to $(i,-\eps_1)$ such that $\{a,b\}^\perp$ and $\{c,d\}^\perp$ are dense in $\R + i\Z$. We will do this using a similar strategy to the proof of Lemma \ref{lem:relnalg}. Fix $(a,b) \in \Lam_{(0,1)}$ such that $|a - i| < \eps$ and $|b + \eps_1| < \eps$, and such that $\{a,b\}^\perp$ is dense in $\R + i\Z$.

Fix $\eps_2,\eps_3 > 0$ such that $\eps_1 + 10\eps < \eps_2 < 2\eps_1 - 10\eps$ and $\eps_1 < \eps_3 < 2\eps_1$. Additionally, fix $b_1 \in b^\perp$ such that $a\cdot b_1 = 1$ and $|b_1 + (\eps_1 - i)| < \eps$. Applying the relations in (\ref{eq:gen2algR+iZ}) twice, we see that for all primitive $a_1 \in \{a,b\}^\perp$ such that $|a_1 + (\eps_2 - 2i)| < 4\eps$, and for all primitive $a_2 \in \{-a_1-b,a+b\}^\perp$ such that $|a_2 - (\eps_3 + 2i)| < 4\eps$,
\be
(a,b) \sim_\Lam (-a_1-b,a+b) \sim_\Lam (-a_2-a-b,-a_1+a) .
\ee
Similarly, for all primitive $a_1 \in \{a,b_1\}^\perp$ such that $|a_1 + (\eps_2 - 2i)| < 4\eps$, and for all primitive $a_2^\pr \in \{-a_1-b_1,a+b_1\}^\perp$ such that $|a_2^\pr - (\eps_3 + i)| < 4\eps$,
\be
(a,b_1) \sim_\Lam (-a_1-b_1,a+b_1) \sim_\Lam (-a_2^\pr-a-b_1,-a_1+a) .
\ee

\paragraph{\bf Step 1a.} Recall that $b_1 \in b^\perp$ is such that $a \cdot b_1 = 1$ and $|b_1 + (\eps_1 - i)| < \eps$. We will show that $(a,b) \sim_\Lam (a,b_1)$. Since $\{a,b,b_1\}^\perp \cap \R$ has rank at least $2g - 4 \geq 2$, it is dense in $\R$. Then since $b_1 - b \in \{a,b,b_1\}^\perp \cap (\R + i)$, the primitive submodule $\{a,b,b_1\}^\perp$ is dense in $\R + i \Z$. Then by Lemma \ref{lem:primdense}, there is a primitive $a_1 \in \{a,b,b_1\}^\perp$ such that $|a_1 + (\eps_2 - 2i)| < \eps$. The primitive submodule $V_1 = \{-a_1-b,a+b,-a_1-b_1,a+b_1\}^\perp$ has rank at least $2g - 3 \geq 3$ by the relation in (\ref{eq:relnalg1}), and since $b_1 - b \in V_1 \cap (\R + i)$, we see that $V_1$ is dense in $\R + i\Z$. Then by Lemma \ref{lem:primdense}, there is $a_2 \in V_1$ such that $|a_2 - (\eps_3 + 2i)| < \eps$ and such that both $a_2$ and $a_2^\pr = a_2 + b - b_1$ are primitive. Note that $a_2^\pr \in V_1$ and that $|a_2^\pr - (\eps_3 + i)| < 3\eps$. Since $V_1 = \{-a_1-b,a+b\}^\perp \cap \{-a_1-b_1,a+b_1\}^\perp$, we have
\be
(a,b) \sim_\Lam (-a_2-a-b,-a_1+a) = (-a_2^\pr-a-b_1,-a_1+a) \sim_\Lam (a,b_1) .
\ee

\paragraph{\bf Step 1b.} Fix $b_4 \in \Lam$ such that $a \cdot b_4 = 1$, $|b_4 + \eps_1| < \eps$, and $\{a,b_4\}^\perp$ is dense in $\R + i\Z$. We will find $b_1,b_2,b_3 \in \Lam$ such that $b_1 \in b + \{a,b\}^\perp$, $b_2 \in b_1 + \{a,b_1\}^\perp$, $b_3 \in b_2 + \{a,b_2\}^\perp$, $b_4 \in b_3 + \{a,b_3\}^\perp$, and such that $|b_1 + (\eps_1 - i)| < \eps$, $|b_2 + \eps_1| < \eps$, $|b_3 + (\eps_1 - i)| < \eps$. Step 1a then implies $(a,b) \sim_\Lam (a,b_1) \sim_\Lam (a,b_2) \sim_\Lam (a,b_3) \sim_\Lam (a,b_4)$.

First, fix $b_2 \in \Lam$ such that $a \cdot b_2 = 1$, $|b_2 + \eps_1| < \eps$, and $\{a,b_2\}^\perp$ is dense in $\R + i\Z$. Additionally, suppose that $b_2 = k_2 a + b + c_2$ with $k_2 \in \Z$, $c_2 \in \{a,b\}^\perp$ primitive, and $\{a,b,c_2\}^\perp$ dense in $\R + i\Z$. Since $c_2 \in \{a,b\}^\perp$ is primitive, there is $d_2 \in \{a,b\}^\perp$ such that $c_2 \cdot d_2 = 1$. Since $\{a,b,c_2\}^\perp$ is dense in $\R + i\Z$, there is $e \in \{a,b,c_2\}^\perp$ such that $|b - k_2 d_2 + e + (\eps_1 - i)| < \eps$. Let $c_1 = -k_2 d_2 + e \in \{a,b\}^\perp$, and let $b_1 = b + c_1$. We have $|b_1 + (\eps_1 - i)| < \eps$ and $b_1 \in b + \{a,b\}^\perp$. Since $b_2 \cdot b_1 = k_2 + c_2 \cdot c_1 = 0$, we have $b_2 \in b_1 + \{a,b_1\}^\perp$. Thus, applying Step 1a twice gives us $(a,b) \sim_\Lam (a,b_1) \sim_\Lam (a,b_2)$.

Now write $b_4 = k_4 a + b + c_4$ with $k_4 \in \Z$ and $c_4 \in \{a,b\}^\perp$. By the previous paragraph, it is enough to find $k_2 \in \Z$ and $c_2 \in \{a,b\}^\perp$ primitive, such that $b_2 = k_2 a + b + c_2$ satisfies $|b_2 + \eps_1| < \eps$ and $\{a,b,c_2\}^\perp$ is dense in $\R + i\Z$, and such that $b_4 = k_4^\pr a + b_2 + c_4^\pr$ with $k_4^\pr \in \Z$ and $c_4^\pr \in \{a,b_2\}^\perp$ primitive and $\{a,b_2,c_4^\pr\}^\perp$ dense in $\R + i\Z$. Note that $k_4^\pr = b_4 \cdot b_2 = k_4 - k_2 + (c_4 \cdot c_2)$, from which it follows that
\be
c_4^\pr = c_4 - c_2 - (c_4 \cdot c_2) a.
\ee
Since $c_2 \in \{a,b\}^\perp$ is primitive, $\{a,b,c_2\}^\perp$ is dense in $\R + i\Z$ if and only if there is $d_2 \in \{a,b\}^\perp \cap \R$ such that $c_2 \cdot d_2 = 1$. We have $\{a,b_2,c_4^\pr\}^\perp = \{a,b+c_4,c_4^\pr\}^\perp$, so since $c_4^\pr \in \{a,b+c_4\}^\perp$ is primitive, $\{a,b+c_4,c_4^\pr\}^\perp$ is dense in $\R + i\Z$ if and only if there is $d_4^\pr \in \{a,b+c_4\}^\perp \cap \R$ such that $c_4^\pr \cdot d_4^\pr = 1$.

By Lemma \ref{lem:Kapovich}, since $\{a,b\}^\perp$ is dense in $\R + i\Z$, there is a symplectic basis $x_2,y_2,\dots,x_g,y_g$ for $\{a,b\}^\perp$ such that $x_2 \in \R + i$, $y_2 \in \R$, and $x_j,y_j \in \R$ for $3 \leq j \leq g$. Since $\Lam \cong \Z^{2g}$, these basis elements are all nonzero. We have
\be
\{a,b\}^\perp \cap \R = \Z y_2 + \Z x_3 + \Z y_3 + \cdots + \Z x_g + \Z y_g .
\ee
Write $c_4 = \sum_{j=2}^g (m_j x_j + n_j y_j)$ with $m_j,n_j \in \Z$. Note that $m_2 = \im(c_4) = -k_4$. We have
\be
\{a,b+c_4\}^\perp = \Z (-n_2 a + x_2) + \Z (m_2 a + y_2) + \cdots + \Z (-n_g a + x_g) + \Z (m_g a + y_g) ,
\ee
so an element $\sum_{j=2}^g \left(s_j(-n_j a + x_j) + t_j(m_j a + y_j)\right) \in \{a,b + c_4\}^\perp$ lies in $\R$ if and only if
\be
s_2(-n_2 + 1) + t_2 m_2 + \sum_{j=3}^g (-s_j n_j + t_j m_j) = 0 .
\ee
There is a symplectic isomorphism $f : \{a,b+c_4\}^\perp \ra \{a,b\}^\perp$ given by $f(-n_j a + x_j) = x_j$ and $f(m_j a + y_j) = y_j$ for $2 \leq j \leq g$. For $d_4^\pr \in \{a,b+c_4\}^\perp$, we have
\be
c_4^\pr \cdot d_4^\pr = (c_4 - c_2) \cdot f(d_4^\pr) .
\ee
Suppose that at least one of $m_3,n_3,\dots,m_g,n_g$ is nonzero. After swapping $(x_3,y_3)$ with $(x_j,y_j)$ for some $3 \leq j \leq g$, and possibly replacing $(x_3,y_3)$ with $(-y_3,x_3)$, we may assume $n_3 \neq 0$. Let $s_2 = -n_3 / \gcd(n_2 - 1,n_3)$ and $s_3 = (n_2 - 1) / \gcd(n_2 - 1,n_3)$, and define
\be
d_4^\pr = s_2 (-n_2 a + x_2) + s_3 (-n_3 a + x_3) \in \{a,b+c_4\}^\perp \cap \R.
\ee
Note that if $n_2 - 1 = 0$, then $\gcd(n_2 - 1,n_3) = |n_3| \neq 0$. If all of $m_3, n_3, \dots, m_g, n_g$ are zero, then $\{a,b+c_4\}^\perp \cap \R$ contains $x_3,y_3,\dots,x_g,y_g$. In this case, let $s_2 = 0$ and $s_3 = 1$, so $d_4^\pr = x_3 \in \{a,b + c_4\}^\perp \cap \R$. Note that $d_4^\pr$ is primitive since $\gcd(s_2,s_3) = 1$. To ensure that $\{a,b_2,c_4^\pr\}^\perp$ is dense in $\R + i\Z$, we will require that
\be
c_4^\pr \cdot d_4^\pr = (c_4 - c_2) \cdot (s_2 x_2 + s_3 x_3) = 1 .
\ee
Next, suppose $c_2 = \sum_{j=2}^g (p_j x_j + q_j y_j)$ with $p_j,q_j \in \Z$. Since $p_2 = \im(c_2) = -k_2$, there is $d_2 \in \{a,b\}^\perp \cap \R$ such that $c_2 \cdot d_2 = 1$ if and only if $\gcd(k_2,p_3,q_3,\dots,p_g,q_g) = 1$. Note that this implies $c_2$ is primitive. The equation $c_4^\pr \cdot d_4^\pr = 1$ reduces to
\be
((n_2 y_2 + n_3 y_3) - (q_2 y_2 + q_3 y_3)) \cdot (s_2 x_2 + s_3 x_3) = (q_2 - n_2)s_2 + (q_3 - n_3)s_3 = 1 .
\ee
Equivalently, $q_2 s_2 + q_3 s_3 = 1 + n_2 s_2 + n_3 s_3$. Since $\gcd(s_2,s_3) = 1$, there is a solution to this equation in $q_2,q_3 \in \Z$. Let $Q_2,Q_3 \in \Z$ be a solution. Then $Q_2 + r s_3$, $Q_3 - r s_2$ is also a solution for all $r \in \Z$.

We now produce the desired $c_2 \in \{a,b\}^\perp$. Fix $k_2 \in \Z$. Let $p_2 = -k_2$. Choose $p_3 = n k_2 + 1$ with $n \in \Z$. Choose $q_2 = Q_2 + r s_3$ and $q_3 = Q_3 - r s_2$ with $r \in \Z$. Now let
\be
c_2 = p_2 x_2 + q_2 y_2 + p_3 x_3 + q_3 y_3.
\ee
We have $\im(c_2) = p_2 = -k_2$. Since $\gcd(k_2,p_3) = 1$, there is $d_2 \in \{a,b\}^\perp \cap \R$ such that $c_2 \cdot d_2 = 1$, so $c_2$ is primitive and $\{a,b,c_2\}^\perp$ is dense in $\R + i\Z$. Since $(q_2 - n_2)s_2 + (q_3 - n_3)s_3 = 1$, we have $c_4^\pr \cdot d_4^\pr = 1$, which implies $c_4^\pr$ is primitive and $\{a,b_2,c_4^\pr\}^\perp$ is dense in $\R + i\Z$. Moreover, these properties are all preserved under translating $p_3$ by an integer multiple of $k_2$ and translating $(q_2,q_3)$ by an integer multiple of $(s_3, -s_2)$, which amounts to translating $c_2$ by an element of $\Z k_2 x_3 + \Z (s_3 y_2 - s_2 y_3)$. Since $\Lam \cong \Z^{2g}$, and since at least one of $s_2,s_3$ is nonzero, we have $(s_3 y_2 - s_2 y_3) / k_2 x_3 \notin \Q$, so $\Z k_2 x_3 + \Z (s_3 y_2 - s_2 y_3)$ is dense in $\R$. Thus, by translating $c_2$ by an element of $\Z k_2 x_3 + \Z (s_3 y_2 - s_2 y_3)$, we can ensure that $b_2 = k_2 a + b + c_2$ satisfies $|b_2 + \eps_1| < \eps$. This concludes Step 1.

Before we begin Step 2, we make the following observation similar to the setup before Step 1. Fix $a_1 \in a^\perp$ such that $a_1 \cdot b = 1$ and $|a_1  - 2i| < \eps$. Applying the relations in (\ref{eq:gen2algR+iZ}) twice, we see that for all primitive $a_1^\pr \in \{a,b\}^\perp$ such that $|a_1^\pr + (\eps_2 - i)| < 4\eps$, and for all primitive $a_2 \in \{-a_1^\pr - b, a+ b\}^\perp$ such that $|a_2 - (\eps_3 + 2i)| < 4\eps$,
\be
(a,b) \sim_\Lam (-a_1^\pr - b, a + b) \sim_\Lam (-a_2 - a - b, -a_1^\pr + a) .
\ee
Similarly, for all primitive $a_1^{\pr\pr} \in \{a_1,b\}^\perp$ such that $|a_1^{\pr\pr} + (\eps_2 - 2i)| < 4\eps$, and for all primitive $a_2^\pr \in \{-a_1^{\pr\pr} - b, a_1 + b\}^\perp$ such that $|a_2^\pr - (\eps_3 + i)| < 4\eps$,
\be
(a_1,b) \sim_\Lam (-a_1^{\pr\pr} - b, a_1 + b) \sim_\Lam (-a_2^\pr - a_1 - b, -a_1^{\pr\pr} + a).
\ee

\paragraph{\bf Step 2a.} Recall that $a_1 \in a^\perp$ is such that $a_1 \cdot b = 1$ and $|a_1 - 2i| < \eps$. An argument similar to Step 1a will show that $(a_1,b) \sim_\Lam (a,b)$. Since $\{a,b,a_1\}^\perp \cap \R$ has rank at least $2g - 4 \geq 2$, it is dense in $\R$. Then since $a_1 - a \in \{a,b,a_1\}^\perp \cap (\R + i)$, the primitive submodule $\{a,b,a_1\}^\perp$ is dense in $\R + i\Z$. Then by Lemma \ref{lem:primdense}, there is $a_1^\pr \in \{a,b,a_1\}^\perp$ such that $|a_1^\pr + (\eps_2 - i)| < \eps$ and such that both $a_1^\pr$ and $a_1^{\pr\pr} = a_1^\pr + a_1 - a$ are primitive. We have $a_1^{\pr\pr} \in \{a,b,a_1\}^\perp$ and $|a_1^{\pr\pr} + (\eps_2 - 2i)| < 3\eps$. The primitive submodule $V_2 = \{-a_1^\pr - b, a + b, -a_1^{\pr\pr} - b, a_1 + b\}^\perp$ has rank at least $2g - 3 \geq 3$ by the relations in (\ref{eq:relnalg2}), and since $a_1 - a \in V_2 \cap (\R + i)$, we see that $V_2$ is dense in $\R + i\Z$. Then by Lemma \ref{lem:primdense}, there is $a_2 \in V_2$ such that $|a_2 - (\eps_3 + 2i)| < \eps$ and such that both $a_2$ and $a_2^\pr = a_2 + a - a_1$ are primitive. We have $a_2^\pr \in V_2$ and $|a_2^\pr - (\eps_3 + i)| < 3\eps$. Since $V_2 = \{-a_1^\pr - b, a + b\}^\perp \cap \{-a_1^{\pr\pr} - b, a_1 + b\}^\perp$, we have
\be
(a,b) \sim_\Lam (-a_2 - a - b, -a_1^\pr + a) = (-a_2^\pr - a_1 - b, -a_1^{\pr\pr} + a_1) \sim_\Lam (a_1,b) .
\ee

\paragraph{\bf Step 2b.} Fix $a_4 \in \Lam$ such that $a_4 \cdot b = 1$, $|a_4 - i| < \eps$, and $\{a_4,b\}^\perp$ is dense in $\R + i\Z$. As in Step 1b, we will find $a_1,a_2,a_3 \in \Lam$ such that $a_1 \in a + \{a,b\}^\perp$, $a_2 \in a_1 + \{a_1,b\}^\perp$, $a_3 \in a_2 + \{a_2,b\}^\perp$, $a_4 \in a_3 + \{a_3,b\}^\perp$, and such that $|a_1 - 2i| < \eps$, $|a_2 - i| < \eps$, $|a_3 - 2i| < \eps$. Step 2a then implies $(a,b) \sim_\Lam (a_1,b) \sim_\Lam (a_2,b) \sim_\Lam (a_3,b) \sim_\Lam (a_4,b)$.

First, fix $a_2 \in \Lam$ such that $a_2 \cdot b = 1$, $|a_2 - i| < \eps$, and $\{a_2,b\}^\perp$ is dense in $\R + i\Z$. Additionally, suppose that $a_2 = a + k_2 b + c_2$ with $k_2 \in \Z$, $c_2 \in \{a,b\}^\perp$ primitive, and $\{a,b,c_2\}^\perp$ dense in $\R + i\Z$. As in Step 1b, there is $d_2 \in \{a,b\}^\perp$ such that $c_2 \cdot d_2 = 1$ and $e \in \{a,b,c_2\}^\perp$ such that $a_1 = a + k_2 d_2 + e$ satisfies $|a_1 - 2i| < \eps$. We have $a_1 \in a + \{a,b\}^\perp$, $a_2 \in a_1 + \{a_1,b\}^\perp$, so by Step 1a, $(a,b) \sim_\Lam (a_1,b) \sim_\Lam (a_2,b)$.

Now write $a_4 = a + k_4 b + c_4$ with $k_4 \in \Z$ and $c_4 \in \{a,b\}^\perp$. It is enough to find $k_2 \in \Z$ and $c_2 \in \{a,b\}^\perp$ primitive, such that $a_2 = a + k_2 b + c_2$ satisfies $|a_2 - 2i| < \eps$ and $\{a,b,c_2\}^\perp$ is dense in $\R + i\Z$, and such that $a_4 = a_2 + k_4^\pr b + c_4^\pr$ with $k_4^\pr \in \Z$ and $c_4^\pr \in \{a_2,b\}^\perp$ primitive such that $\{a,b_2,c_4^\pr\}^\perp$ is dense in $\R + i\Z$. The same argument as in Step 1b, with the roles of $a$ and $b$ exchanged, produces the desired $k_2$ and $c_2$. This concludes Step 2b. \\

\paragraph{\bf Step 3.} Exactly the same as Step 3 in the proof of Lemma \ref{lem:relnalg}.
\end{proof}

\begin{proof} (of Theorem \ref{thm:connR+iZ})
The proof exactly follows that of Theorems \ref{thm:connsum} and \ref{thm:connsplit}.
\end{proof}

\paragraph{\bf The transfer principle and proving Theorems \ref{thm:ergdense} and \ref{thm:ergdenseR+iZ}.} Theorem \ref{thm:conn} can be used to prove Theorem \ref{thm:ergdense}, using the transfer principle from \cite{CDF:transfer} and applications of general results in homogeneous dynamics to the action of $\Sp(2g,\Z)$ on $\Sp(2g,\R)/\Sp(2g-2,\R)$ from \cite{Kap:periods}, which we briefly explain. For $\phi \in H^1(S_g;\C)$, let $V(\phi) \subset H^1(S_g;\R)$ be the real vector subspace spanned by the real and imaginary parts $\re(\phi)$ and $\im(\phi)$. The symplectic automorphism group $\Aut(H^1(S_g;\R)) \cong \Sp(2g;\R)$ acts transitively on the set of $\phi \in H^1(S_g;\C)$ such that $\langle \phi, \phi \rangle = 1$ by acting on $\re(\phi)$ and $\im(\phi)$ simultaneously, and the stabilizer of $\phi$ is $\Aut(V(\phi)^\perp) \cong \Sp(2g-2;\R)$. Let $\cC \subset \Om\cM_g(\kap)$ be a stratum component satisfying the hypotheses of Theorem \ref{thm:conn}, and let $\Pi : \wt{\cC} \ra \cC$ be the Torelli cover of this stratum component. Recall that points in $\wt{\cC} \subset \Om\cS_g(\kap)$ are holomorphic $1$-forms $(X,\om) \in \cC$ equipped with a symplectic isomorphism $H^1(S_g;\C) \cong H^1(X;\C)$ identifying $H^1(S_g;\Z)$ and $H^1(X;\Z)$. Consider the restriction of the period map $\Per_g : \wt{\cC} \ra H^1(S_g;\C)$. Since $\Per_g$ is a holomorphic submersion on $\wt{\cC}$, the image $\Per_g(\wt{\cC})$ is open in $H^1(S_g;\C)$. Moreover, the image of $\Per_g$ is invariant under the action of $\Aut(H^1(S_g;\Z)) \cong \Sp(2g,\Z)$. The set
\be
\mathcal{H}_g = \left\{\phi \in H^1(S_g;\C) : \langle \phi, \phi \rangle = 1, \; \Per(\phi) \cong \Z^{2g}, \; \Per(\phi) \cap \R z \subset \Q z \text{ for all } z \in \Per(\phi) \right\}
\ee
is then $\Aut(H^1(S_g;\Z))$-invariant, and is contained in the image of $\Per_g$ by Lemma \ref{lem:Kapovich}. Thus, we can identify $\mathcal{H}_g$ with an $\Sp(2g;\Z)$-invariant full measure subset of $\Sp(2g,\R)/\Sp(2g-2,\R)$. The set
\be
\mathcal{G}_\cC = \left\{(X,\om) \in \cC_1 : \Per(\om) \cong \Z^{2g}, \; \Per(\om) \cap \R z \subset \Q z \text{ for all } z \in \Per(\om) \right\}
\ee
is saturated for the absolute period foliation of $\cC_1$ and is a full measure subset of $\cC_1$.

Now, Theorem \ref{thm:conn} tells us that for all $\phi \in \mathcal{H}_g$, the associated space of isoperiodic forms $\cC(\phi)$ is connected. Since $\Per(\phi) \cong \Z^{2g}$, this is equivalent to the fiber $\Per_g^{-1}(\phi)$ being connected. Following \cite{CDF:transfer}, since the fibers $\Per_g^{-1}(\phi)$ are connected for all $\phi \in \mathcal{H}_g$, we get a bijection $A \mapsto \Per_g(\Pi^{-1}(A))$ between subsets of $\mathcal{G}_\cC$ that are saturated for the absolute period foliation of $\cC_1$ and subsets of $\mathcal{H}_g$ that are invariant under the action of $\Aut(H^1(S_g;\Z))$. Under this bijection, positive measure subsets correspond to positive measure subsets, and dense subsets correspond to dense subsets. It follows from Moore's ergodicity theorem \cite{Zim:book} that the action of $\Sp(2g,\Z)$ on $\Sp(2g,\R)/\Sp(2g-2,\R)$ is ergodic, and thus the absolute period foliation of $\cC_1$ is ergodic. From the classification of $\Sp(2g,\Z)$-orbit closures in $\Sp(2g,\R)/\Sp(2g-2,\R)$ in Lemma \ref{lem:Kapovich}, which is an application of Ratner's orbit closure theorem \cite{Rat:ICM}, we deduce that leaves of the absolute period foliation in $\mathcal{G}_\mathcal{C}$ are dense in $\mathcal{C}_1$. This establishes Theorem \ref{thm:ergdense}. Similarly, Theorem \ref{thm:connR+iZ} implies Theorem \ref{thm:ergdenseR+iZ}. \\

\begin{rmk} \label{rem:genus4}
If we already knew the genus $3$ case of Theorems \ref{thm:conn} and \ref{thm:connR+iZ}, the equivalence relations in Lemmas \ref{lem:connalg} and \ref{lem:connalgR+iZ} would have simpler definitions, and their proofs would be much easier. This is because in genus at least $4$, the holomorphic $1$-forms in question admit a pair of splittings whose associated cylinders are disjoint. See Figure \ref{fig:sumdisjoint} for an example in $\Om\cM_4(5,1)$. For the inductive steps in genus at least $4$, we would know that $\sim_\Lam$ satisfies $(z_1,w_1) \sim_\Lam (z_2,w_2)$ whenever $z_2,w_2 \in \{z_1,w_1\}^\perp$ and $\im(\ol{z}_1 w_1) + \im(\ol{z}_2 w_2) < 1$. Using Lemma \ref{lem:Kapovich}, one can then show that for any $(z_1,w_1),(z_2,w_2) \in \Lam_{(0,1)}$, there is $(z_3,w_3) \in \Lam_{(0,1)}$ with $z_3,w_3 \in \{z_1,w_1,z_2,w_2\}^\perp$ and $\im(\ol{z}_3 w_3) > 0$ arbitrarily small, and thus $\Lam_{(0,1)}$ consists of a single equivalence class for $\sim_\Lam$. This approach crucially relies on being in genus at least $4$. However, our proofs are inductive and crucially rely on the genus $3$ case, which is the hardest case.
\end{rmk}

\begin{figure}
    \centering
    \includegraphics[width=0.8\textwidth]{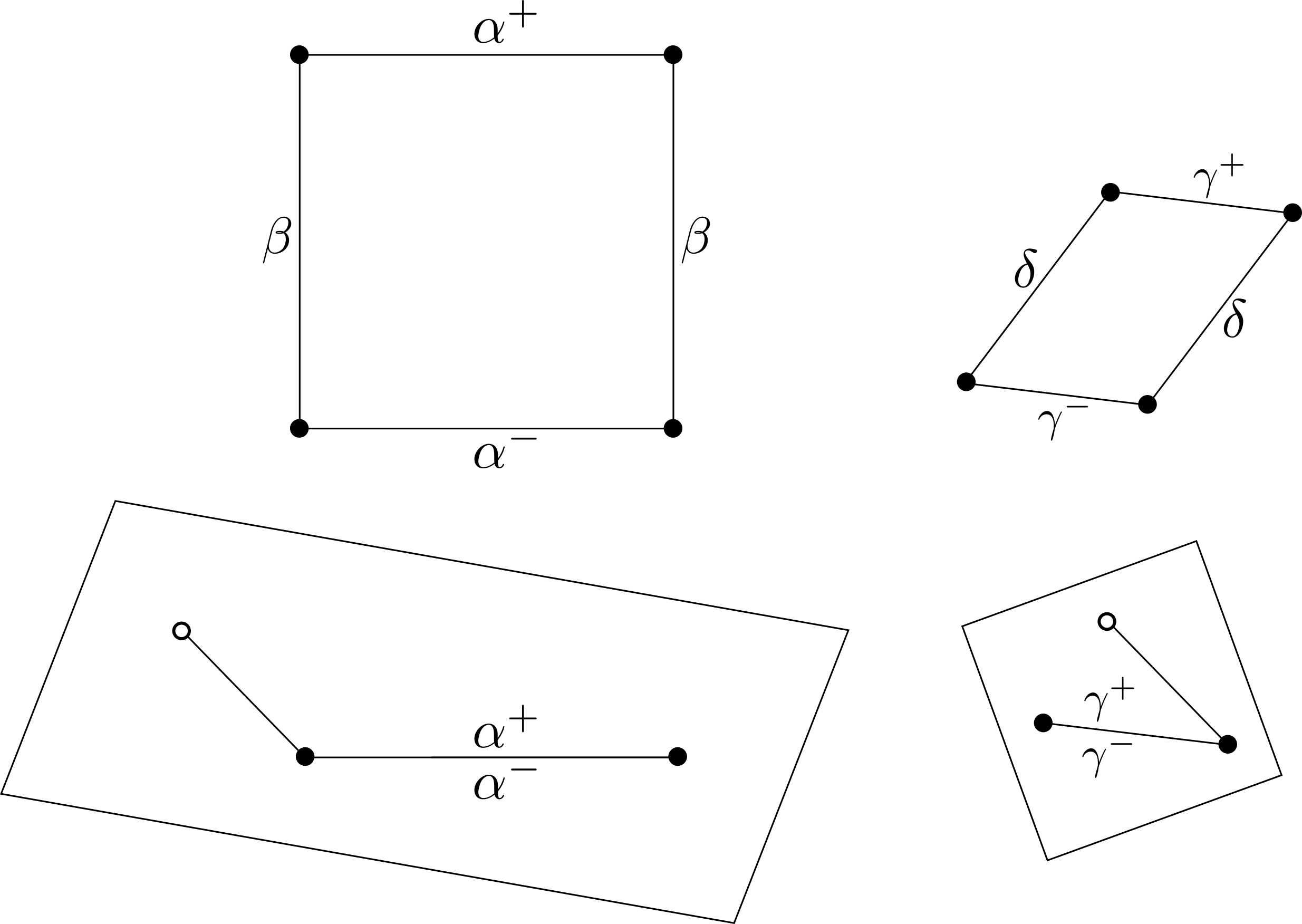}
    \caption{A holomorphic $1$-form in $\Om\cM_4(5,1)$ with a pair of splittings $\al^\pm$ and $\gam^\pm$ whose associated cylinders are disjoint.}
    \label{fig:sumdisjoint}
\end{figure}

\begin{rmk} \label{rmk:noEMM}
Our proof of the ergodicity part of Theorem \ref{thm:ergdense} does not crucially rely on the rigidity theorems of Eskin-Mirzakhani-Mohammadi \cite{EMM:closures}, and only relies on Moore's ergodicity theorem and the ergodicity of the $\GL^+(2,\R)$-action on stratum components. We can replace each stratum component $\cC$ appearing in Theorem \ref{thm:ergdense} with a certain nonempty open subset $\cU \subset \cC$ that is $\GL^+(2,\R)$-invariant and saturated for the absolute period foliation of $\cC$. When $\cC$ is a stratum component with two zeros, $\cU$ consists of leaves of the absolute period foliation containing a holomorphic $1$-form admitting a presentation as an iterated connected sum, for instance as in the proof of Lemma \ref{lem:sum2}. When $\cC$ is a stratum component with more than two zeros, $\cU$ is similarly defined using iterated connected sums following by iterated zero splittings. With this modification, the proof of Theorem \ref{thm:conn} shows that the intersection $\cC(\phi) \cap \cU$ is connected for all cohomology classes $\phi$ as in Theorem \ref{thm:ergdense}. Since $\cC_1 \cap \cU$ is a full measure subset of $\cC_1$, this suffices to establish the ergodicity of the absolute period foliation of $\cC_1$ via the transfer principle from \cite{CDF:transfer} as above.
\end{rmk}

We conclude this subsection with a question that proposes a possible classification of closures of leaves of $\cA(\kap)$ in $\Om\cM_g(\kap)$. This question asks whether all of the possible constraints on closures of leaves of $\cA(\kap)$ come from closed subgroups of $\C$ containing the absolute periods, closed $\SL(2,\R)$-invariant subsets of $\Om\cM_g(\kap)$ that are saturated for $\cA(\kap)$, and loci of branched covers.

\begin{ques} \label{ques:clos2}
Let $\Om\cM_g(\kap)$ be a stratum with $|\kap| > 1$, and fix $(X,\om) \in \Om\cM_g(\kap)$. Let $L$ be the leaf of $\cA(\kap)$ through $(X,\om)$, and let $\ol{L}$ be the closure of $L$ in $\Om\cM_g(\kap)$. Let $\Lam$ be the closure of $\Per(\om)$ in $\C$, and let $\Lam_0$ be the identity component of $\Lam$. Lastly, let $\cM$ be the closure of $\SL(2,\R) \cdot L$ in $\Om\cM_g(\kap)$. Is one of the following true?
\begin{enumerate}
\item $\Lam = \C$ and $\ol{L} = \cM$.
\item $\Lam = M \cdot (\R + i\Z)$ with $M \in \SL(2,\R)$ and $\ol{L}$ is a connected component of the set of holomorphic $1$-forms $(X^\pr,\om^\pr) \in \cM$ such that $\Per(\om^\pr) + \Lam_0 = \Lam$.
\item $\ol{L} = L$ and $L$ consists of branched covers of holomorphic $1$-forms of lower genus.
\end{enumerate}
\end{ques}
Note that the statement in (3) applies to some holomorphic $1$-forms whose absolute periods are dense in $\C$ or $\R + i\Z$, since the branched cover may be of a holomorphic $1$-form of genus greater than $1$. A positive answer to Question \ref{ques:clos2} would imply that closures of leaves of the absolute period foliation enjoy rigidity properties similar to those of $\GL^+(2,\R)$-orbit closures. Specifically, for any leaf $L$ of $\cA(\kap)$, either $\R_{>0} \cdot \ol{L}$ is $\GL^+(2,\R)$-invariant and therefore locally defined by homogeneous $\R$-linear equations, or $\ol{L}$ is locally defined by inhomogeneous $\R$-linear equations in the real and imaginary parts of local period coordinates. \\


\paragraph{\bf Monodromy.} Let $\cC$ be a stratum component, and let $\pi_1(\cC)$ be its orbifold fundamental group. The projection $\cC \ra \cM_g$ induces a homomorphism $\pi_1(\cC) \ra \Mod_g$. Choose a symplectic basis $\{a_j,b_j\}_{j=1}^g$ for $H_1(S_g;\Z)$. The choice of symplectic basis gives us an action of $\Sp(2g,\Z)$ on $H_1(S_g;\Z)$, and an action on $H^1(S_g;\C)$ by acting on homomorphisms $H_1(S_g;\Z) \ra \C$ by precomposition. The action of $\Mod_g$ on $H_1(S_g;\Z)$ then induces a homomorphism
\be
\rho_\cC : \pi_1(\cC) \ra \Sp(2g,\Z)
\ee
called the {\em monodromy representation} of $\pi_1(\cC)$ on absolute homology. We now describe the implications of the connectivity of spaces of isoperiodic forms in $\cC$ for the image of $\rho_\cC$, and we prove Theorem \ref{thm:mono}.

Fix $(X,\om) \in \cC$ without automorphisms, choose a symplectic isomorphism $m : H_1(S_g;\Z) \ra H_1(X;\Z)$, and let $\phi \in H^1(S_g;\C)$ be the cohomology class satisfying $\phi(c) = \int_{m(c)} \om$ for all $c \in H_1(S_g;\Z)$, so $(X,\om) \in \cC(\phi)$. The symplectic basis $\{m(a_j),m(b_j)\}_{j=1}^g$ for $H_1(X;\Z)$ determines an action of $\Sp(2g,\Z)$ on $H_1(X;\Z)$. Let
\be
P(X,\om) = \left(\int_{m(a_1)}\om,\int_{m(b_1)}\om,\dots,\int_{m(a_g)}\om,\int_{m(b_g)}\om\right) \in \C^{2g}
\ee
be the vector of absolute periods with respect to this symplectic basis of $H_1(X;\Z)$.  Choose a path $\gam_1 : [0,1] \ra \cC$ from $(X,\om)$ to some $(X^\pr,\om^\pr)$. Parallel transport along $\gam_1$ determines an identification $H_1(X;\Z) \cong H_1(X^\pr;\Z)$, so we can consider the vector of absolute periods $P(X^\pr,\om^\pr)$ with respect to the corresponding symplectic basis of $H_1(X^\pr;\Z)$. Fix $A \in \Sp(2g,\Z)$, and suppose that
\begin{equation} \label{eq:spper}
A \cdot P(X,\om) = P(X^\pr,\om^\pr) .
\end{equation}
The action of $A$ on $H_1(X;\Z)$ satisfies
\be
\int_{A \cdot c} \om = \int_c \om^\pr
\ee
for all $c \in H_1(X;\Z)$. Since the homomorphism $\Mod_g \ra \Aut(H_1(S_g;\Z))$ is surjective, we then have $(X^\pr,\om^\pr) \in \cC(\phi)$. Suppose that $\cC(\phi)$ is connected. Then $(X,\om)$ and $(X^\pr,\om^\pr)$ lie on the same leaf of the absolute period foliation of $\cC$, so there is another path $\gam_2 : [0,1] \ra \cC$ from $(X^\pr,\om^\pr)$ to $(X,\om)$ along which the absolute periods are constant. Concatenating $\gam_1$ and $\gam_2$ gives a loop $\gam$ in $\cC$ based at $(X,\om)$. Letting $B = \rho_\cC(\gam)$, the action of $B$ on $H_1(X;\Z)$ also satisfies
\be
\int_{B \cdot c} \om = \int_c \om^\pr
\ee
for all $c \in H_1(X;\Z)$, so $B^{-1} A$ stabilizes $\phi$. We also have a monodromy homomorphism
\be
\rho_{\cC(\phi)} : \pi_1(\cC(\phi)) \ra \Sp(2g,\Z) .
\ee
In the typical case where $\Per(\phi) \cong \Z^{2g}$, the matrix $B^{-1} A$ is the identity matrix, and thus $B$ lies in the image of $\rho_\cC$. In the special case where $\Per(\phi)$ has rank less than $2g$, if we additionally assume that the image of $\rho_{\cC(\phi)}$ contains the stabilizer of $\phi$ in $\Sp(2g,\Z)$, then $B$ lies in the image of $\rho_\cC$.

Now consider all of the paths $\gam_1 : [0,1] \ra \cC$ such that $\gam_1(0) = (X,\om)$ and such that $\gam_1(1) = (X^\pr,\om^\pr)$ lies in $\cC(\phi)$. Associated to each such path is a matrix $A \in \Sp(2g,\Z)$ such that $(X,\om)$, $(X^\pr,\om^\pr)$, and $A$ satisfy (\ref{eq:spper}). Let $S_\om$ denote this set of matrices, and let $G_\om$ be the subgroup of $\Sp(2g,\Z)$ generated by $S_\om$. Note that $S_\om$ and $G_\om$ implicitly depend on the symplectic basis $\{m(a_j),m(b_j)\}_{j=1}^g$. Our goal is to show that the connectivity of $\cC(\phi)$, together with the above assumption on $\rho_{\cC(\phi)}$, imply that $\rho_\cC$ is surjective. Thus, our task is to find a collection of paths satisfying (\ref{eq:spper}) for which the associated matrices generate $\Sp(2g,\Z)$, which we restate in the following lemma.

\begin{lem} \label{lem:pathgen}
Let $\cC$ be a component of a stratum $\Om\cM_g(\kap)$ with $|\kap| > 1$, and let $\phi \in H^1(S_g;\C)$ be positive. Suppose that $\cC(\phi)$ is connected. If $\Per(\phi)$ has rank less than $2g$, suppose also that the image of $\rho_{\cC(\phi)}$ contains the stabilizer of $\phi$ in $\Sp(2g,\Z)$. If there is $(X,\om) \in \cC(\phi)$ such that $G_\om = \Sp(2g,\Z)$, then $\rho_\cC$ is surjective.
\end{lem}

First, we use Kapovich's classification \cite{Kap:periods} of $\Sp(2g,\Z)$-orbit closures in the space of positive cohomology classes in $H^1(S_g;\C)$, restated in Lemma \ref{lem:Kapovich}, to restrict our attention to a small open subset of a stratum component.

\begin{lem} \label{lem:spgenopen}
Fix $g \geq 3$, and let $\cC$ be a component of a stratum $\Om\cM_g(\kap)$. Suppose there is a nonempty open subset $\cU \subset \cC$ such that for all $(X,\om) \in \cU$, we have $G_\om = \Sp(2g,\Z)$. Then for all positive $\phi \in H^1(S_g;\C)$ such that $\Per(\phi)$ is not discrete, there is $(Y,\eta) \in \cC(\phi)$ such that $G_\eta = \Sp(2g,\Z)$.
\end{lem}

\begin{proof}
The $\R$-linear action of $\GL^+(2,\R)$ on $\C$ induces an action on $H^1(S_g;\C)$ by acting on homomorphisms $H_1(S_g;\Z) \ra \C$ by postcomposition. Fix a positive $\phi \in H^1(S_g;\C)$ such that $\Per(\phi)$ is not discrete. By Lemma \ref{lem:Kapovich}, $\GL^+(2,\R) \cdot (\Sp(2g,\Z) \cdot \phi)$ is dense in the space of positive cohomology classes in $H^1(S_g;\C)$. This implies $\GL^+(2,\R) \cdot \cC(\phi)$ is dense in $\cC$. The $\GL^+(2,\R)$-action on $\cC$ respects (\ref{eq:spper}), in the sense that if $\gam_1 : [0,1] \ra \cC$ is a path from $(X,\om)$ to $(X^\pr,\om^\pr)$ and $A$ is a matrix in $\Sp(2g,\Z)$ such that $\gam_1$ and $A$ satisfy (\ref{eq:spper}), then for all $M \in \GL^+(2,\R)$, the path $M \gam_1$ and the matrix $A$ also satisfy (\ref{eq:spper}). Thus, the set of $(X,\om) \in \cC$ such that $G_\om  = \Sp(2g,\Z)$ is $\GL^+(2,\R)$-invariant. By assumption, this set also contains a nonempty open subset of $\cC$, so it intersects $\GL^+(2,\R) \cdot \cC(\phi)$.
\end{proof}

Next, we recall a convenient generating set for $\Sp(2g,\Z)$ from Section 6.1 in \cite{FM:primer}, which we describe in terms of the chosen symplectic basis $\{a_j,b_j\}_{j=1}^g$ for $H_1(S_g;\Z)$. The {\em shears} $U_1,U_2$ are given by
\be
U_1(b_1) = a_1 + b_1, \quad U_2(a_1) = a_1 + b_1.
\ee
The {\em factor mix} $M$ is given by
\be
M(b_1) = b_1 + a_2, \quad M(b_2) = b_2 + a_1 .
\ee
The {\em factor swaps} $W_j$, $1 \leq j \leq g - 1$, are given by
\be
W_j(a_j) = a_{j+1}, \quad W_j(b_j) = b_{j+1}, \quad W_j(a_{j+1}) = a_j, \quad W_j(b_{j+1}) = b_j .
\ee
In each case, elements of $\{a_j,b_j\}_{j=1}^g$ not mentioned are fixed. The group $\Sp(2g,\Z)$ is generated by the $2$ shears, the factor mix, and the $g-1$ factor swaps.

We now begin our construction of holomorphic $1$-forms $(X,\om) \in \cC(\phi)$ with $G_\om = \Sp(2g,\Z)$, for all components $\cC$ of strata $\Om\cM_g(\kap)$ with $|\kap| > 1$, and all positive $\phi \in H^1(S_g;\C)$ with $\Per(\phi)$ not discrete. As usual, we reduce to the case of strata with two zeros, by splitting zeros.

\begin{lem} \label{lem:spgensplit}
Fix $(X,\om) \in \cC$, and suppose that $(Y,\eta) \in \cC^\pr$ arises from $(X,\om)$ by splitting a zero. Then $S_\om \subset S_\eta$, and therefore $G_\om \subset G_\eta$.
\end{lem}

\begin{proof}
Suppose $A \in S_\om$. This means there is a path $\gam_1 : [0,1] \ra \cC$ from $(X,\om)$ to $(X^\pr,\om^\pr)$ and a matrix $A \in \Sp(2g,\Z)$ satisfying (\ref{eq:spper}). Splitting a zero on $(X,\om)$ to obtain $(Y,\eta)$ results in a saddle connection $s$ on $(Y,\eta)$ with distinct endpoints. Let $z = \int_s \eta$, and let $Z_1$ and $Z_2$ be the starting and ending points of $s$, respectively. Let $Z_3,\dots,Z_{n+1}$ be the other zeros of $\eta$, and for $3 \leq j \leq n+1$, let $c_j \in H_1(Y,Z(\eta);\Z)$ be represented by a path from $Z_1$ to $Z_j$.

Fix $\eps > 0$ small, and let $\gam _0: [0,1] \ra L(\eta)$ be the path starting at $(Y,\eta)$ such that on $\gam_0(t) = (Y_t,\eta_t)$, we have
\be
\int_s \eta_t = \left(1 - t\left(1 - \frac{\eps}{|z|}\right)\right) z, \quad \int_{c_j} \eta_t = \int_{c_j} \eta \text{ for } 3 \leq j \leq n + 1.
\ee
Along $\gam_0$, the zero $Z_2$ moves toward $Z_1$. The holomorphic $1$-form $(Y_1,\eta_1)$ arises from $(X,\om)$ by splitting the same zero to get a saddle connection whose holonomy has absolute value $\eps$.

By choosing a lift $\wt{\gam}_1$ of $\gam_1$ to a stratum cover by prong-marked differentials, making a continuous choice along $\wt{\gam}_1$ of short segments with holonomy $\eps z / |z|$ emanating from the distinguished zero, and applying a zero splitting map, we obtain a path $\gam_1^\pr : [0,1] \ra \cC^\pr$ starting at $(Y_1,\eta_1)$. Since splitting zeros does not change the absolute periods, the path $\gam_1^\pr$ and the matrix $A$ also satisfy (\ref{eq:spper}). Concatenating $\gam_0$ and $\gam_1^\pr$ then gives us $A \in S_\eta$.
\end{proof}

For the case of nonhyperelliptic components of strata with two zeros, we will use connected sums with a torus to reduce to a small number of base cases.

\begin{lem} \label{lem:spgensum}
Let $\cC$ be a nonhyperelliptic component of a stratum $\Om\cM_g(m_1,m_2)$ with $g \geq 4$. Suppose that $(X,\om) \in \cC$ has a pair of splittings $\al_j^\pm$, $j = 1,2$, such that the associated cylinders $C_j$ are disjoint. Let $(z_j,w_j)$ be the associated periods of $\al_j^\pm$. Let $(X_j,\om_j)$ be the holomorphic $1$-form in genus $g - 1$ obtained from $(X,\om)$ by slitting and regluing $\al_j^\pm$, and let $\cC_j$ be the stratum component containing $(X_j,\om_j)$. Suppose that $S_{\om_j}$ contains paths $\gam_{j,k} : [0,1] \ra \cC_j$, $1 \leq k \leq n_j$, from $(X_j,\om_j)$ to $(X_{j,k},\om_{j,k})$, with the following properties.
\begin{itemize}
    \item The associated matrices $A_{j,k}$, $1 \leq k \leq n_j$, generate $\Sp(2g - 2, \Z)$.
    \item The endpoints $(X_{j,k},\om_{j,k})$ of $\gam_{j,k}$, $1 \leq k \leq n_j$, do not have any saddle connections whose holonomies lie in $I_j = \{t z_j : 0 \leq t \leq 1\}$.
\end{itemize}
Then $G_\om = \Sp(2g,\Z)$.
\end{lem}

\begin{proof}
For $j = 1,2$, let $\al_j \subset C_j$ be a closed geodesic, and let $\bet_j \subset C_j \cup Z(\om)$ be a saddle connection crossing $C_j$, oriented so that $\al_j$ and $\bet_j$ intersect exactly once positively. Let $a_j = [\al_j] \in H_1(X;\Z)$, and let $b_j = [\bet_j] \in H_1(X;\Z)$. Then any homology class in $\{a_1,b_1,a_2,b_2\}^\perp$ is represented by a union of closed loops contained in $X \sm (\ol{C}_1 \cup \ol{C}_2)$.

Fix $\eps > 0$ small. For $j = 1,2$, let $\gam_j : [0,1] \ra \cC$ be the path starting at $(X,\om)$ such that on $\gam_j(t) = (X_t,\om_t)$, we have
\be
\int_{a_j} \om_t = \left(1 - t\left(1 - \frac{\eps}{|z_j|}\right)\right) z_j ,
\ee
while $\int_{b_j} \om_t$ remains constant and the relative period of any path whose interior lies in $X \sm \ol{C}_j$ remains constant. Along $\gam_j$, the splitting $\al_j^\pm$ shrinks until its period has absolute value $\eps$, while $X \sm \ol{C}_j$ remains unchanged. Let $(X_j^\pr,\om_j^\pr) = \gam_j(1)$.

Next, fix $1 \leq k \leq n_j$. By choosing a lift $\wt{\gam}_{j,k}$ of $\gam_{j,k}$ to a stratum cover by prong-marked differentials, making a continuous choice along $\wt{\gam}_{j,k}$ of short segments with holonomy $\eps z_j / |z_j|$ emanating from the distinguished zero, and applying a connected sum map, we obtain a path $\gam_{j,k}^\pr : [0,1] \ra \cC$ starting at $(X_j^\pr,\om_j^\pr)$. Let $(X_{j,k}^\pr,\om_{j,k}^\pr) = \gam_{j,k}^\pr(1)$. For $0 \leq t \leq 1$, define $z_j(t) = (\eps / |z_j| + t(1 - \eps/|z_j|)) z_j$, and let $T_j(t)$ be the flat torus $(\C / (\Z z_j(t) + \Z w_j), dz)$. Then $(X_{j,k}^\pr,\om_{j,k}^\pr)$ arises from $(X_{j,k},\om_{j,k})$ by a connected sum with the torus $T_j(0)$. Moreover, since $(X_{j,k},\om_{j,k})$ does not have any saddle connections whose holonomies lie in $I_j$ by assumption, we can similarly form connected sums with $(X_{j,k},\om_{j,k})$ and the torus $T_j(t)$ for all $0 \leq t \leq 1$. In this way, we obtain a path $\gam_{j,k}^{\pr\pr} : [0,1] \ra \cC$ starting at $(X_{j,k}^\pr,\om_{j,k}^\pr)$. Along $\gam_{j,k}^{\pr\pr}$, the splitting $\al_j^\pm$ on $(X_{j,k}^\pr,\om_{j,k}^\pr)$ grows until its period is $z_j$, while $X_{j,k}^\pr \sm \ol{C}_j$ remains unchanged.

Now extend $a_1,b_1,a_2,b_2$ to a symplectic basis $a_1,b_1,\dots,a_g,b_g$ for $H_1(X;\Z)$. Let $H_1 \cong \Sp(2g - 2, \Z)$ be the subgroup of $\Sp(2g,\Z)$ fixing $a_1$ and $b_1$, and let $H_2 \cong \Sp(2g - 2, \Z)$ be the subgroup of $\Sp(2g,\Z)$ fixing $a_2$ and $b_2$. For $j = 1,2$, and for $1 \leq k \leq n_j$, the concatenation $\gam_j \cup \gam_{j,k}^\pr \cup \gam_{j,k}^{\pr\pr}$ and the matrix $A_{j,k}^\pr \in H_j$ determined by $A_{j,k}$ satisfy (\ref{eq:spper}). Since the matrices $A_{j,k}$, $1 \leq k \leq n_j$, generate $\Sp(2g - 2, \Z)$, the matrices $A_{j,k}^\pr$, $1 \leq k \leq n_j$, generate $H_j$. This shows that $G_\om$ contains the subgroups $H_1$ and $H_2$. Lastly, since $g \geq 4$, the subgroups $H_1,H_2$ generate $\Sp(2g,\Z)$, thus $G_\om = \Sp(2g,\Z)$.
\end{proof}

The construction in the proof of Lemma \ref{lem:sum2} can be adapted to show that, with a few exceptions in low genus, any nonhyperelliptic component $\cC$ of a stratum with two zeros contains holomorphic $1$-forms with a pair of splittings for which the associated cylinders are disjoint. See Figure \ref{fig:sumdisjoint} for an example in $\Om\cM_4(5,1)$. This will allow us to run an inductive argument with Lemma \ref{lem:spgensum} to construct holomorphic $1$-forms $(X,\om) \in \cC(\phi)$ with $G_\om = \Sp(2g,\Z)$ for all positive $\phi \in H^1(S_g;\C)$ such that $\Per(\phi)$ is not discrete. We now address the base cases needed for this inductive argument.

\begin{lem} \label{lem:spgenbase}
Let $\cC$ be one of the strata $\Om\cM_2(1,1)$, $\Om\cM_3(3,1)$, the nonhyperelliptic component of $\Om\cM_3(2,2)$, or a component of $\Om\cM_4(4,2)$. Fix a positive $\phi \in H^1(S_g;\C)$ such that $\Per(\phi)$ is not discrete. There is $(X,\om) \in \cC(\phi)$ such that $G_\om = \Sp(2g,\Z)$, where $g$ is the genus of $X$.
\end{lem}

\begin{proof}
The proof is in cases, one for each stratum component in the statement of Lemma \ref{lem:spgenbase}. \\

\paragraph{\bf Case 1.} Suppose $\cC = \Om\cM_2(1,1)$. Every holomorphic $1$-form in $\Om\cM_2(1,1)$ can be presented as a pair of flat tori glued along a pair of homologous saddle connections. (See, for instance, Proposition 1.16 in \cite{Wri:survey}.) Fix $(X,\om) \in \cC(\phi)$, and suppose $(X,\om)$ is obtained by gluing $T_1 = (\C / (\Z z_1 + \Z w_1), dz)$ and $T_2 = (\C / (\Z z_2 + \Z w_2), dz)$ along a pair of homologous saddle connections $s^{\pm}$. Since the set of holomorphic $1$-forms $(Y,\eta) \in \cC$ with $G_\eta = \Sp(4,\Z)$ is invariant under $\GL^+(2,\R)$, by applying an element of $\GL^+(2,\R)$ to $(X,\om)$ we may assume that $z_1 = 1$ and $w_1 = i$. Since $\Per(\phi)$ is not discrete, the lattices $\Z + \Z i$ and $\Z z_2 + \Z w_2$ are not commensurable. Thus, if $a^2 = \Area(T_2)$, by applying an element of $\SL(2,\Z)$ to $(X,\om)$ we may assume that $(z_2,w_2)$ is close to $(a,ia)$. Lastly, by moving along $L(\om)$ we may assume that the saddle connections $s^\pm$ are short. Let $u = \int_{s^{\pm}} \om$.

To show that $G_\om = \Sp(4,\Z)$, it is enough to show that $S_\om$ contains the generators $U_1,U_2,M,W_1$ for $\Sp(4,\Z)$. We will do this by describing paths $\gam_{U_1},\gam_{U_2},\gam_M,\gam_{W_1} : [0,1] \ra \cC$ starting at $(X,\om)$ in terms of the period coordinates $z_1,w_1,z_2,w_2,u$. Any period coordinates not mentioned will remain constant, and the period $u$ of the saddle connections $s^\pm$ will always remain constant. The paths are as follows.
\begin{itemize}
    \item On $\gam_{U_1}(t)$, we have $w_1(\gam_{U_1}(t)) = w_1 + tz_1$.
    \item On $\gam_{U_2}(t)$, we have $z_1(\gam_{U_2}(t)) = z_1 + tw_1$.
    \item On $\gam_M(t)$, we have $w_j(\gam_M(t)) = w_j + tz_{3-j}$ for $j = 1,2$.
    \item On $\gam_{W_1}(t)$, we have $z_j(\gam_{W_1}(t)) = (1 - t)z_j + tz_{3-j}$ and $w_j(\gam_{W_1}(t)) = (1 - t)w_j + tw_{3-j}$ for $j = 1,2$.
\end{itemize}
The paths $\gam_{U_j}$ are closed loops in $\cC$, and $\gam_{U_j}$ and the shear $U_j$ satisfy (\ref{eq:spper}). Since $(z_1,w_1) = (1,i)$ and $(z_2,w_2)$ is close to $(a,ia)$, the paths $\gam_M$ and $\gam_{W_1}$ are well-defined. The path $\gam_M$ and the factor mix $M$ satisfy (\ref{eq:spper}), and the path $\gam_{W_1}$ and the factor swap $W_1$ also satisfy (\ref{eq:spper}).

We are done with Case 1. The rest of the cases involve stratum components in genus $g \geq 3$, so by Lemma \ref{lem:spgenopen}, it is enough to find a nonempty open subset of each of these stratum components consisting of holomorphic $1$-forms $(Y,\eta)$ with $G_\eta = \Sp(2g,\Z)$. \\

\paragraph{\bf Case 2.} Suppose $\cC$ is the nonhyperelliptic component of $\Om\cM_3(2,2)$. This case is similar to Case 1. We have a nonempty open subset $\cU \subset \cC$ consisting of holomorphic $1$-forms $(X,\om)$ that can be presented as a triple of flat tori $T_j = (\C / (\Z z_j + \Z w_j), dz)$, $1 \leq j \leq 3$, glued along a triple of homologous saddle connections $s_j$, $1 \leq j \leq 3$, such that $(z_j,w_j)$ is close to $(1,i)$ and $u = \int_{s_j} \om$ is small. Here, $T_j$ is bounded by $s_j \cup s_{j+1}$, indices taken modulo $3$. To see that $(X,\om)$ lies in the nonhyperelliptic component of $\Om\cM_3(2,2)$, recall from Theorem \ref{thm:KZ} that this component is also the odd component of $\Om\cM_3(2,2)$. Since $H_1(X;\Z)$ has a symplectic basis represented by closed geodesics $\{\al_j,\bet_j\}_{j=1}^3$ with $\al_j,\bet_j \subset T_j$ such that $\int_{\al_j} \om = z_j$ and $\int_{\bet_j} \om = w_j$, and since these closed geodesics all have index $0$, the spin parity is $\phi(\om) = \sum_{j=1}^3 (0 + 1)(0 + 1) = 1 \hmod 2$.

We describe paths $\gam_{U_1},\gam_{U_2},\gam_M,\gam_{W_1},\gam_{W_2} : [0,1] \ra \cC$ starting at $(X,\om)$ in terms of the period coordinates $z_1,w_1,z_2,w_2,z_3,w_3,u$. The paths $\gam_{U_1},\gam_{U_2},\gam_M,\gam_{W_1}$ are as in Case 1, and they satisfy (\ref{eq:spper}) with associated matrices $U_1,U_2,M,W_1$, respectively. On $\gam_{W_2}(t)$, we have
\be
z_j(\gam_{W_2}(t)) = (1 - t)z_j + tz_{5-j}, \quad w_j(\gam_{W_2}(t)) = (1 - t)w_j + tw_{5-j}, \quad j = 2,3 .
\ee
The path $\gam_{W_2}$ and the factor swap $W_2$ satisfy (\ref{eq:spper}). Thus, $S_\om$ contains a generating set for $\Sp(6,\Z)$ and $G_\om = \Sp(6,\Z)$. \\

\paragraph{\bf Case 3.} Suppose $\cC = \Om\cM_3(3,1)$. There is a nonempty open subset of $\Om\cM_2(1,1)$ consisting of holomorphic $1$-forms that can be presented as a pair of flat tori $T_j = (\C / (\Z z_j + \Z w_j), dz)$, $1 \leq j \leq 2$, glued along a pair of homologous saddle connections $s_j$, $1 \leq j \leq 2$, such that $(z_j,w_j)$ is close to $(1,i)$ and $u = \int_{s_j} \om$ is close to $i/4$. By forming connected sums with a torus $T_3 = (\C / (\Z z_3 + \Z w_3), dz)$ with $(z_3,w_3)$ close to $((1+i)/\sqrt{2},(-1+i)/\sqrt{2})$, using the segment in $T_2$ that starts at the starting point of $s_j$ and has holonomy $(1+i)/\sqrt{2}$, we get a nonempty open subset $\cU \subset \Om\cM_3(3,1)$.

Fix $(X,\om) \in \cU$ with a presentation as above. The paths $\gam_{U_1},\gam_{U_2},\gam_M,\gam_{W_1}$ are as in Case 1, and they satisfy (\ref{eq:spper}) with associated matrices $U_1,U_2,M,W_1$, respectively. To describe $\gam_{W_2}$, let $\gam_1,\gam_2,\gam_3 : [0,1] \ra \cC$ be paths such that $\gam_1$ starts at $(X,\om)$ and $\gam_{j+1}$ starts at $\gam_j(1)$ for $j = 1,2$, defined as follows. Fix $\eps > 0$ small.
\begin{itemize}
    \item On $\gam_1(t)$, we have $z_3(\gam_1(t)) = (1 - t(1 - \eps/|z_3|))z_3$.
    \item On $\gam_2(t)$, we have $z_2(t) = e^{2\pi it/8} z_2$, $w_2(t) = e^{2\pi it/8} w_2$, $z_3(t) = e^{-2\pi it/8} \eps z_3/|z_3|$, $w_3(t) = e^{-2\pi it/8} w_3$.
    \item On $\gam_3(t)$, we have $z_2(t) = (1 - t) e^{2\pi i/8} z_2 + t z_3$, $w_2(t) = (1 - t) e^{2\pi i/8} w_2 + tw_3$, $z_3(t) = (1 - t) e^{-2\pi it/8} \eps z_3/|z_3| + t z_2$, $w_3(t) = (1 - t) e^{-2\pi it/8} w_3 + t w_2$.
\end{itemize}
Let $\gam_{W_2}$ be the concatenation $\gam_1 \cup \gam_2 \cup \gam_3$. Along $\gam_{W_2}$, we first shrink the splitting bounding $T_3$ until it has length $\eps$. We then rotate the periods $z_2,w_2$ counterclockwise and simultaneously rotate $z_3,w_3$ clockwise until they have approximately the desired arguments. Lastly, we linearly interpolate to get the desired final values of $z_2,w_2,z_3,w_3$, which amounts to perturbing $z_2,w_2,w_3$ slightly while restoring the length of $z_3$. The path $\gam_{W_2}$ and the factor swap $W_2$ satisfy (\ref{eq:spper}). Thus, $S_\om$ contains a generating set for $\Sp(6,\Z)$ and $G_\om = \Sp(6,\Z)$. \\

\paragraph{\bf Case 4.} Suppose $\cC$ is the odd component of $\Om\cM_4(4,2)$. This case is similar to Case 3. Recall we have a nonempty open subset of $\Om\cM_3(2,2)$ consisting of holomorphic $1$-forms that can be presented as a triple of flat tori $T_j = (\C / (\Z z_j + \Z w_j), dz)$, $1 \leq j \leq 3$, glued along a triple of homologous saddle connections $s_j$, $1 \leq j \leq 3$, such that $(z_j,w_j)$ is close to $(1,i)$ and $u = \int_{s_j} \om$ is close to $i/4$. By forming connected sums with a torus $T_4 = (\C / (\Z z_4 + \Z w_4), dz)$ with $(z_4,w_4)$ close to $((1+i)/\sqrt{2},(-1+i)/\sqrt{2})$, using the segment in $T_3$ that starts at the starting point of $s_j$ and has holonomy $(1 + i)/\sqrt{2}$, we get a nonempty open subset $\cU \subset \Om\cM_4(4,2)$. Moreover, since we started in the odd component of $\Om\cM_3(2,2)$ and forming connected sums with a torus preserves spin parity, $\cU$ is contained in the odd component of $\Om\cM_4(4,2)$.

The paths $\gam_{U_1},\gam_{U_2},\gam_M,\gam_{W_1},\gam_{W_2}$ are as in Case 2, and they satisfy (\ref{eq:spper}) with associated matrices $U_1,U_2,M,W_1,W_2$, respectively. The path $\gam_{W_3}$ is similar to the path for the last factor swap in Case 3. Let $\gam_1,\gam_2,\gam_3 : [0,1] \ra \cC$ be paths such that $\gam_1$ starts at $(X,\om)$ and $\gam_{j+1}$ starts at $\gam_j(1)$ for $j = 1,2$, defined as follows. Fix $\eps > 0$ small.
\begin{itemize}
    \item On $\gam_1(t)$, we have $z_4(\gam_1(t)) = (1 - t(1 - \eps/|z_4|))z_4$.
    \item On $\gam_2(t)$, we have $z_3(t) = e^{2\pi it/8} z_3$, $w_3(t) = e^{2\pi it/8} w_3$, $z_4(t) = e^{-2\pi it/8} \eps z_4/|z_4|$, $w_4(t) = e^{-2\pi it/8} w_4$.
    \item On $\gam_3(t)$, we have $z_3(t) = (1 - t) e^{2\pi i/8} z_3 + t z_4$, $w_3(t) = (1 - t) e^{2\pi i/8} w_3 + tw_4$, $z_4(t) = (1 - t) e^{-2\pi it/8} \eps z_4/|z_4| + t z_3$, $w_4(t) = (1 - t) e^{-2\pi it/8} w_4 + t w_3$.
\end{itemize}
Let $\gam_{W_3}$ be the concatenation $\gam_1 \cup \gam_2 \cup \gam_3$. Then $\gam_{W_3}$ and the factor swap $W_3$ satisfy (\ref{eq:spper}). Thus, $S_\om$ contains a generating set for $\Sp(8,\Z)$ and $G_\om = \Sp(8,\Z)$. \\

\paragraph{\bf Case 5.} Suppose $\cC$ is the even component of $\Om\cM_4(4,2)$. Recall from Theorem \ref{thm:KZ} that the hyperelliptic component of $\Om\cM_3(2,2)$ is also an even component. We can construct a holomorphic $1$-form in this hyperelliptic component from $6$ parallelograms as follows. Fix $z_0 \in \C$ close to $1/2$, and for $1 \leq j \leq 3$, fix $z_j,w_j \in \C$ such that $(z_j,w_j)$ is close to $(1,i)$. Up to translation, a parallelogram in $\C$ is specified by a pair of complex numbers, one for each pair of parallel sides. Let $P_1,Q_1,P_2,Q_2,P_3,Q_3$ be the following parallelograms.
\begin{itemize}
    \item $P_1$ has sides given by $z_0,w_1$.
    \item $Q_1$ has sides given by $z_1 - z_0, w_1$.
    \item $P_2$ has sides given by $z_1 - z_0, w_2$.
    \item $Q_2$ has sides given by $z_2 - z_1 + z_0, w_2$.
    \item $P_3$ has sides given by $z_2 - z_1 + z_0, w_3$.
    \item $Q_3$ has sides given by $z_3 - z_2 + z_1 - z_0, w_3$.
\end{itemize}
Glue left and right sides together in pairs, and top and bottom sides together in pairs, such that glued sides correspond to the same complex number. Specifically, for $1 \leq j \leq 3$, glue the left side of $P_j$ and the right side of $Q_j$, and glue the right side of $P_j$ and the left side of $Q_j$. Glue the top and bottom sides of $P_1$ together. For $1 \leq j \leq 2$, glue the top side of $Q_j$ and the bottom side of $P_{j+1}$, and glue the bottom side of $Q_j$ and the top side of $P_{j+1}$. Glue the top and bottom sides of $Q_3$ together.

Let $(X_0,\om_0)$ be the resulting holomorphic $1$-form. The corners on the left sides of $P_1,Q_2,P_3$ and the right sides of $Q_1,P_2,Q_3$ are identified to a zero of order $2$. Similarly, the corners on the right sides of $P_1,Q_2,P_3$ and the left sides of $Q_1,P_2,Q_3$ are identified to a zero of order $2$. There is an isometric involution of $(X_0,\om_0)$ that acts on the interior of each parallelogram by rotating by $\pi$. This involution exchanges the two zeros, and has a fixed point in the center of each parallelogram and in the center of the top side of $P_1$ and the bottom side of $Q_3$, for a total of $8$ fixed points. This verifies that $(X_0,\om_0)$ lies in the hyperelliptic component of $\Om\cM_3(2,2)$. The holomorphic $1$-forms with presentations by parallelograms as above form a nonempty open subset of the hyperelliptic component of $\Om\cM_3(2,2)$.

Next, by forming connected sums with the torus $T_4 = (\C / (\Z z_4 + \Z w_4), dz)$ with $(z_4,w_4)$ close to $(e^{2\pi i/16}, ie^{2\pi i/16})$, using the segment in $P_3 \cup Q_3$ that starts at the bottom-left corner of $Q_3$ and has holonomy $z_4$, we obtain a nonempty open subset $\cU \subset \Om\cM_4(4,2)$. Since we started in the even component of $\Om\cM_3(2,2)$, $\cU$ is contained in the even component of $\Om\cM_4(4,2)$.

Fix $(X,\om) \in \cU$. We describe paths $\gam_{U_1},\gam_{U_2},\gam_M,\gam_{W_1},\gam_{W_2},\gam_{W_3} : [0,1] \ra \cC$ starting at $(X,\om)$ in terms of the period coordinates $z_1,w_1,\dots,z_4,w_4,z_0$. Here, $z_1,w_1,\dots,z_4,w_4$ arise from a symplectic basis for absolute homology, and $z_0$ arises from a path joining the two zeros. As before, any period coordinates not mentioned remain constant.
\begin{itemize}
    \item On $\gam_{U_1}(t)$, we have $w_1(\gam_{U_1}(t)) = w_1 + tz_1$.
    \item On $\gam_{U_2}(t)$, we have $z_1(\gam_{U_2}(t)) = z_1 + tw_1$ and $z_0(\gam_{U_2}(t)) = z_0 + tw_1$.
    \item On $\gam_M(t)$, we have $w_j(\gam_M(t)) = w_j + tz_{3-j}$ for $j = 1,2$.
    \item For $j = 1,2$, on $\gam_{W_j}(t)$, we have $z_k(\gam_{W_j}(t)) = (1 - t)z_k + tz_{2j+1-k}$ and $w_k(\gam_{W_j}(t)) = (1 - t)w_k + tw_{2j+1-k}$ for $k = j,j+1$.
\end{itemize}
Along $\gam_{U_1}$, the parallelograms $P_1$ and $Q_1$ are sheared approximately horizontally. Along $\gam_{U_2}$, the parallelogram $P_1$ is sheared approximately vertically. Along $\gam_M$, the parallelograms $P_1,Q_1,P_2,Q_2$ are all sheared approximately horizontally. For $j = 1,2$, along $\gam_{W_j}$, the parallelograms $P_j,Q_j,P_{j+1},Q_{j+1}$ are slightly perturbed. Lastly, the path $\gam_{W_3}$ is a concatenation of $3$ paths $\gam_1,\gam_2,\gam_3$, similarly to Case 4, as follows. Fix $\eps > 0$ small.
\begin{itemize}
    \item On $\gam_1(t)$, we have $z_4(\gam_1(t)) = (1 - t(1 - \eps/|z_4|))z_4$.
    \item On $\gam_2(t)$, we have $z_3(t) = e^{2\pi it/16} z_3$, $w_3(t) = e^{2\pi it/16} w_3$, $z_4(t) = e^{-2\pi it/16} \eps z_4/|z_4|$, $w_4(t) = e^{-2\pi it/16} w_4$.
    \item On $\gam_3(t)$, we have $z_3(t) = (1 - t) e^{2\pi i/16} z_3 + t z_4$, $w_3(t) = (1 - t) e^{2\pi i/16} w_3 + tw_4$, $z_4(t) = (1 - t) e^{-2\pi it/16} \eps z_4/|z_4| + t z_3$, $w_4(t) = (1 - t) e^{-2\pi it/16} w_4 + t w_3$.
\end{itemize}
The paths $\gam_{U_1},\gam_{U_2},\gam_M,\gam_{W_1},\gam_{W_2},\gam_{W_3}$ satisfy (\ref{eq:spper}) with associated matrices $U_1,U_2,M,W_1,W_2,W_3$, respectively. Thus, $S_\om$ contains a generating set for $\Sp(8,\Z)$ and $G_\om = \Sp(8,\Z)$.
\end{proof}

\begin{lem} \label{lem:spgen2zeros}
Fix $g \geq 3$, and let $\cC$ be a nonhyperelliptic component of a stratum $\Om\cM_g(m_1,m_2)$. Fix a positive $\phi \in H^1(S_g;\C)$ such that $\Per(\phi)$ is not  discrete. There is $(X,\om) \in \cC(\phi)$ such that $G_\om = \Sp(2g,\Z)$.
\end{lem}

\begin{proof}
Lemma \ref{lem:spgenbase} handles the cases where $\cC$ is one of the strata $\Om\cM_2(1,1)$, $\Om\cM_3(3,1)$, the nonhyperelliptic component of $\Om\cM_3(2,2)$, or a component of $\Om\cM_4(4,2)$. By Lemma \ref{lem:sum2}, all other nonhyperelliptic components $\cC$ of strata $\Om\cM_g(m_1,m_2)$ with $m_1,m_2$ odd contain a nonempty open $\GL^+(2,\R)$-invariant subset $\cS_\cC \subset \cC$ consisting of holomorphic $1$-forms with a pair of splittings whose associated cylinders are disjoint. In Cases 4 and 5 of the proof of Lemma \ref{lem:spgenbase}, we constructed holomorphic $1$-forms in each component of $\Om\cM_4(4,2)$ with a splitting. By iteratively forming connected sums with a torus as in the proof of Lemma \ref{lem:sum2}, we see that all other nonhyperelliptic components $\cC$ of strata $\Om\cM_g(m_1,m_2)$ with $m_1,m_2$ even contain a nonempty open $\GL^+(2,\R)$-invariant subset $\cS_\cC \subset \cC$ consisting of holomorphic $1$-forms with a pair of splittings whose associated cylinders are disjoint.

Let $\cC \subset \Om\cM_g(m_1,m_2)$ be a nonhyperelliptic component not covered by Lemma \ref{lem:spgenbase}. By induction, each nonhyperelliptic component of $\Om\cM_{g-1}(m_1,m_2-2)$ and $\Om\cM_{g-1}(m_1-2,m_2)$ contains a nonempty open $\GL^+(2,\R)$-invariant set of holomorphic $1$-forms $(Y,\eta)$ such that $G_\eta = \Sp(2g-2,\Z)$. Thus, $\cS_\cC$ contains a nonempty open $\GL^+(2,\R)$-invariant subset of holomorphic $1$-forms $(X,\om)$ with a pair of splittings $\al_j^\pm$, $j = 1,2$, whose associated cylinders are disjoint, such that the holomorphic $1$-forms $(X_j,\om_j)$ in genus $g - 1$ obtained by slitting and regluing $\al_j^\pm$ lie in nonhyperelliptic stratum components $\cC_j$ and have $G_{\om_j} = \Sp(2g-2,\Z)$. For $j = 1,2$, let $z_j = \int_{\al_j^\pm} \om$ and let $I_j = \{t z_j : 0 \leq t \leq 1\}$. Let $\gam_{j,1},\dots,\gam_{j,n_j} : [0,1] \ra \cC_j$ be paths starting at $(X_j,\om_j)$ realizing $G_{\om_j} = \Sp(2g-2,\Z)$, and let $A_{j,1},\dots,A_{j,n_j}$ be the associated matrices that generate $\Sp(2g-2,\Z)$.

There are small open neighborhoods $\cU_j \subset \cC_j$ of $(X_j,\om_j)$ such that, as $(X_j,\om_j)$ varies over $\cU_j$, we can vary the paths $\gam_{j,k}$ slightly to keep the associated matrices $A_{j,k}$ constant. The resulting endpoints of the paths $\gam_{j,k}$ are contained in a finite union of small open subsets $\cU_{j,k} \subset \cC_j$. The set of holomorphic $1$-forms in $\cC$ with no saddle connections whose holonomy lies in $I_1 \cup I_2$ is dense in $\cC$. Moreover, by Lemma \ref{lem:Iclosed}, this set is open in $\cC$. Thus, possibly after replacing $(X,\om)$ with a nearby holomorphic $1$-form in $\cC$ and shrinking the neighborhoods $\cU_j$ of $(X_j,\om_j)$, we can ensure that the subsets $\cU_{j,k}$ do not contain any holomorphic $1$-forms with a saddle connection whose holonomy lies in $I_1 \cup I_2$. We can then apply Lemma \ref{lem:spgensum} to conclude that $G_\om = \Sp(2g,\Z)$, and we are then done by Lemma \ref{lem:spgenopen}.
\end{proof}

It remains to address the case of hyperelliptic stratum components. Since we are only addressing strata with at least two zeros, we only need to consider the hyperelliptic component of $\Om\cM_g(g-1,g-1)$, and we use a separate argument here.

\begin{lem} \label{lem:spgenhyp}
Fix $g \geq 2$, and let $\cC$ be the hyperelliptic component of $\Om\cM_g(g-1,g-1)$. Fix a positive $\phi \in H^1(S_g;\C)$ such that $\Per(\phi)$ is not discrete. There is $(X,\om) \in \cC(\phi)$ such that $G_\om = \Sp(2g,\Z)$.
\end{lem}

\begin{proof}
The case where $\cC$ is the stratum $\Om\cM_2(1,1)$ is covered by Lemma \ref{lem:spgenbase}. For the other cases, since $g \geq 3$, by Lemma \ref{lem:spgenopen} it is enough to find a nonempty open subset of $\cC$ consisting of holomorphic $1$-forms $(Y,\eta)$ with $G_\eta = \Sp(2g,\Z)$.

We generalize the construction of holomorphic $1$-forms in the hyperelliptic component of $\Om\cM_3(2,2)$ from Case 5 of the proof of Lemma \ref{lem:spgenbase}. Fix $z_0 \in \C$ close to $1/2$, and for $1 \leq j \leq g$, fix $z_j,w_j \in \C$ such that $(z_j,w_j)$ is close to $(1,i)$. For $1 \leq j \leq g$, let $P_j,Q_j$ be parallelograms given by a pair of complex numbers as follows.
\begin{itemize}
    \item $P_j$ has sides given by $z_{j-1} - z_{j-2} + \cdots + (-1)^{j-1} z_0$ and $w_j$.
    \item $Q_j$ has sides given by $z_j - z_{j-1} + \cdots + (-1)^j z_0$ and $w_j$.
\end{itemize}
For $1 \leq j \leq g$, glue $P_j$ and $Q_j$ along the pairs of sides given by $w_j$. For $1 \leq j \leq g-1$, glue $P_{j+1}$ and $Q_j$ along the pairs of sides given by $z_j - z_{j-1} + \cdots + (-1)^j z_0$. Glue the two sides of $P_1$ given by $z_0$ together, and glue the two sides of $Q_g$ given by $z_g - z_{g-1} + \cdots + (-1)^g z_0$ together. Let $(X,\om)$ be the resulting holomorphic $1$-form. The corners on the left sides of $P_1,Q_2,P_3,\dots$ and the right sides of $Q_1,P_2,Q_2,\dots$ are identified to form a zero of order $g - 1$. The corners on the right sides of $P_1,Q_2,P_3,\dots$ and the left sides of $Q_1,P_2,Q_3,\dots$ are identified to form another zero of order $g - 1$. There is an isometric involution of $(X,\om)$ that acts on the interior of each parallelogram by rotating by $\pi$, and so exchanges the two zeros of $\om$. This involution fixes the center of each parallelogram, and fixes the midpoints of the top side of $P_1$ and the bottom side of $Q_g$, for a total of $2g + 2$ fixed points. This verifies that $(X,\om)$ lies in the hyperelliptic component $\cC$ of $\Om\cM_g(g-1,g-1)$. The holomorphic $1$-forms with presentations by $2g$ parallelograms as above form a nonempty open subset $\cU \subset \cC$.

Now fix $(X,\om) \in \cU$. We describe paths $\gam_{U_1},\gam_{U_2},\gam_M,\gam_{W_1},\dots,\gam_{W_{g-1}} : [0,1] \ra \cC$ starting at $(X,\om)$ in terms of the period coordinates $z_1,w_1,\dots,z_g,w_g,z_0$. Here, $z_1,w_1,\dots,z_g,w_g$ arise from a symplectic basis for absolute homology, and $z_0$ arises from a path joining the two zeros. As before, any period coordinates not mentioned remain constant.
\begin{itemize}
    \item On $\gam_{U_1}(t)$, we have $w_1(\gam_{U_1}(t)) = w_1 + t z_1$.
    \item On $\gam_{U_2}(t)$, we have $z_1(\gam_{U_2}(t)) = z_1 + t w_1$ and $z_0(\gam_{U_2}(t)) = z_0 + t w_1$.
    \item On $\gam_M(t)$, we have $w_j(\gam_M(t)) = w_j + tz_{3-j}$ for $j = 1,2$.
    \item For $1 \leq j \leq g - 1$, on $\gam_{W_j}(t)$, we have $z_k(\gam_{W_j}(t)) = (1 - t)z_k + tz_{2j+1 - k}$ and $w_k(\gam_{W_j}(t)) = (1 - t)w_k + tw_{2j+1-k}$ for $k = j,j+1$.
\end{itemize}
Along $\gam_{U_1}$, the parallelograms $P_1$ and $Q_1$ are sheared approximately horizontally. Along $\gam_{U_2}$, the parallelogram $P_1$ is sheared approximately vertically. Along $\gam_M$, the parallelograms $P_1,Q_1,P_2,Q_2$ are all sheared approximately horizontally. Along $\gam_{W_j}$, the parallelograms $P_j,Q_j,P_{j+1},Q_{j+1}$ are slightly perturbed. The paths $\gam_{U_1},\gam_{U_2},\gam_M,\gam_{W_1},\dots,\gam_{W_{g-1}}$ satisfy (\ref{eq:spper}) with associated matrices $U_1,U_2,M,W_1,\dots,W_{g-1}$, thus $G_\om = \Sp(2g,\Z)$.
\end{proof}

We have completed the task formulated in Lemma \ref{lem:pathgen}, so we can now conclude with our results on disconnected spaces of isoperiodic forms.

\begin{lem} \label{lem:spgen}
Let $\cC$ be a component of a stratum $\Om\cM_g(\kap)$ with $|\kap| > 1$, and let $\phi \in H^1(S_g;\C)$ be a positive cohomology class such that $\Per(\phi)$ is not discrete. There is $(X,\om) \in \cC(\phi)$ such that $G_\om = \Sp(2g,\Z)$.
\end{lem}

\begin{proof}
Induct on $|\kap|$, using Lemma \ref{lem:spgen2zeros} for the base case $|\kap| = 2$, and Lemma \ref{lem:spgensplit} and Lemmas \ref{lem:split}-\ref{lem:splitspin} for the inductive step.
\end{proof}

\begin{thm} \label{thm:mono2}
Fix $g \geq 3$, and let $\cC$ be a component of a stratum $\Om\cM_g(\kap)$ with $|\kap| > 1$, and suppose that $\cC$ is a spin component or a hyperelliptic component. Fix $\phi \in H^1(S_g;\C)$ positive such that $\Per(\phi)$ is not discrete.
\begin{enumerate}
    \item If $\Per(\phi) \cong \Z^{2g}$, then $\cC(\phi)$ is disconnected.
    \item If $\Per(\phi)$ has rank less than $2g$ and $\cC(\phi)$ is connected, then the image of $\rho_{\cC(\phi)}$ does not contain the stabilizer of $\phi$ in $\Sp(2g,\Z)$.
\end{enumerate}
\end{thm}

\begin{proof}
If $\Per(\phi) \cong \Z^{2g}$, suppose that $\cC(\phi)$ is connected. If $\Per(\phi)$ has rank less than $2g$, suppose that $\cC(\phi)$ is connected and that the image of $\rho_{\cC(\phi)}$ contains the stabilizer of $\phi$ in $\Sp(2g,\Z)$. In either case, Lemma \ref{lem:spgen} tells us there is $(X,\om) \in \cC(\phi)$ such that $G_\om = \Sp(2g,\Z)$, and then by Lemma \ref{lem:pathgen}, the monodromy homomorphism $\rho_\cC : \pi_1(\cC) \ra \Sp(2g,\Z)$ is surjective. However, by Corollary 1.3 in \cite{Gut:Zorich}, this is a contradiction since $\cC$ is a spin component or a hyperelliptic component.
\end{proof}

Theorem \ref{thm:mono} is Case 1 of Theorem \ref{thm:mono2}. We note that for components of strata with at least two zeros that are not spin or hyperelliptic components, we recover the result in \cite{Gut:Zorich} that $\rho_\cC$ is surjective. \\

\paragraph{\bf Disconnected spaces of isoperiodic forms and covering constructions.} Lastly, we describe our examples of spaces of isoperiodic forms that have positive dimension and infinitely many connected components, which arise from covering constructions.

\begin{proof} (of Theorem \ref{thm:disconn})
Write $g = 2h$ with $h \geq 2$. Recall that $\cC = \Om\cM_g(2g-3,1)$, and let $\cC^\pr = \Om\cM_h(2h-2)$. Fix $(Y,\eta) \in \cC^\pr$ such that $\Per(\eta)$ is dense in $\C$. Choose an oriented geodesic segment $\gam$ on $(Y,\eta)$ that starts at the zero of $\eta$ and is otherwise disjoint from $Z(\eta)$. Take two copies of $(Y,\eta)$, slit each copy along $\gam$, and reglue opposite sides of the slits to obtain a holomorphic $1$-form $(X,\om) \in \cC$. The two zeros are identified to form a zero of order $2g-3$, and the other endpoints of the slits are identified to form a zero of order $1$. There is a degree $2$ holomorphic branched covering $f : X \ra Y$ such that $f^\ast \eta = \om$ that is branched over the two endpoints of $\gam$.

Choose $\phi^\pr \in H^1(S_h;\C)$ such that $(Y,\eta) \in \cC^\pr(\phi^\pr)$, and choose $\phi \in H^1(S_g;\C)$ such that $(X,\om) \in \cC(\phi)$. The connected component of $(X,\om)$ in $\cC(\phi)$ is a closed leaf of $\cA(2g-3,1)$ consisting of degree $2$ branched coverings of $(Y,\eta)$ branched over the zero of $\eta$ and a point in $Y \sm Z(\eta)$. Note that connected components of $\cC(\phi)$ have complex dimension $1$, while connected components of $\cC(\phi^\pr)$ are points.

Since $\Per(\phi^\pr)$ is dense in $\C$, when $h \geq 3$, Lemma \ref{lem:Kapovich} implies that $\cC^\pr(\phi^\pr)$ is dense in a fixed-area locus in $\cC^\pr$. For $h = 2$, there is one other possibility when $(Y,\eta)$ is an eigenform for real multiplication, in which case $\cC^\pr(\phi^\pr)$ is dense in the $\SL(2,\R)$-orbit closure of $(Y,\eta)$. In particular, $\cC^\pr(\phi^\pr)$ is infinite. For any $(Y^\pr,\eta^\pr) \in \cC^\pr(\phi^\pr)$, and any oriented geodesic segment $\gam^\pr$ starting at a zero of $\eta^\pr$ and otherwise disjoint from $Z(\eta^\pr)$, by slitting and regluing two copies of $(Y^\pr,\eta^\pr)$ along $\gam^\pr$ as above, we obtain a holomorphic $1$-form $(X^\pr,\om^\pr) \in \cC(\phi)$ as a branched double cover of $(Y^\pr,\eta^\pr)$. The cover $(X^\pr,\om^\pr)$ has an automorphism of order $2$ that exchanges the two sheets of the cover. Moreover, the automorphism group of any holomorphic $1$-form in $\cC$ has order at most $2$, since such an automorphism must fix the unique simple zero. Thus, each element of $\cC^\pr(\phi^\pr)$ determines a distinct connected component of $\cC(\phi)$ in this way. Since $\cC^\pr(\phi^\pr)$ is infinite, $\cC(\phi)$ has infinitely many connected components.
\end{proof}


\bibliographystyle{math}
\bibliography{my.bib}

{\small
\noindent
Email: kgwinsor@gmail.com

\noindent
Department of Mathematics, Stony Brook University, Stony Brook, New York, USA
}

\end{document}